\documentclass[11pt]{article}
\nonstopmode
\def\version{
August 11, 2019
}

\usepackage[dvips]{graphics}
\usepackage{cancel} 
\usepackage{ulem}
\def\XXint#1#2#3{{\setbox0=\hbox{$#1{#2#3}{\int}$}
  \vcenter{\hbox{$#2#3$}}\kern-.5\wd0}}

\usepackage[usenames,dvipsnames]{color} 

\usepackage{latexsym,epsfig,bm}

\usepackage{bm}
\usepackage{upgreek}

\usepackage{mathrsfs}
\usepackage{times}

\definecolor{MyDarkBlue}{rgb}{0,0.08,0.45}
\definecolor{MyDarkGreen}{rgb}{0,0.6,0.1}
\definecolor{Pomegranate}{rgb}{0.6,0.1,0.15}
\definecolor{purple}{rgb}{0.6,0.1,0.15}
\usepackage[backref=none]{hyperref}
\hypersetup{pdfborder={0 0 0},
  colorlinks,
  urlcolor={MyDarkBlue},
  linkcolor={MyDarkBlue},
  citecolor={MyDarkBlue},
  breaklinks=true}
\usepackage{doi}

\usepackage{amssymb}
\usepackage{amsmath}
\usepackage{amsthm}

\providecommand{\eqref}[1]{{\rm (\ref{#1})}}
\providecommand{\itref}[1]{{\it (\ref{#1})}}

\DeclareSymbolFont{AMSb}{U}{msb}{m}{n}
\DeclareSymbolFontAlphabet{\mathbb}{AMSb}

\DeclareSymbolFont{EUR}{U}{eur}{m}{n}
\SetSymbolFont{EUR}{bold}{U}{eur}{b}{n}
\DeclareSymbolFontAlphabet{\eur}{EUR}

\DeclareSymbolFont{EUB}{U}{eur}{b}{n}
\SetSymbolFont{EUB}{bold}{U}{eur}{b}{n}
\DeclareSymbolFontAlphabet{\eub}{EUB}

\newcommand{\dom}{\mathfrak{D}}
\newcommand{\range}{\mathfrak{R}}
\renewcommand{\ker}{\mathop{\mathbf{ker}}}

\newcommand{\frakL}{\mathfrak{L}}
\newcommand{\jj}{\mathrm{i}}
\newcommand{\bmK}{\bm{K}}

\newcommand{\Span}{\mathop{\mathrm{Span}}}
\newcommand{\End}{\,{\rm End}\,}
\providecommand{\Ln}{\mathop{\rm Ln}}

\newcommand{\scrX}{\mathscr{X}}

\newcommand{\eubA}{\eub{A}}
\newcommand{\eubD}{\eub{D}}
\newcommand{\eubJ}{\eub{J}}
\newcommand{\eubK}{\eub{K}}
\newcommand{\eubL}{\eub{L}}

\newcommand{\eubV}{\eub{V}}

\newcommand{\eubX}{\eub{X}}
\newcommand{\eubY}{\eub{Y}}

\newcommand{\eurl}{\eur{l}}
\newcommand{\eubz}{\eub{z}}
\newcommand{\e}{\bm{e}}

\newcommand{\bmupalpha}{\bm\upalpha}
\newcommand{\bmupbeta}{\bm\upbeta}
\newcommand{\bmuppsi}{\bm\uppsi}
\newcommand{\bmupphi}{\bm\upphi}
\newcommand{\bmuprho}{\bm\uprho}

\newcommand{\bmPsi}{\bm\Psi}
\newcommand{\bmXi}{{\bm\Xi}}
\newcommand{\eubP}{\eub{P}}

\newcommand{\tr}{\mathop{\rm Tr}}

\newcommand{\rank}{\mathop{\rm rank}}

\newcommand{\supp}{\mathop{\rm supp}}

\newcommand{\p}{\partial}
\newcommand{\at}[1]{\vert\sb{\sb{#1}}}

\def\R{\mathbb{R}}
\newcommand{\C}{\mathbb{C}}

\newcommand{\bfW}{\mathbf{W}}
\newcommand{\bfX}{\mathbf{X}}
\newcommand{\bfY}{\mathbf{Y}}

\newcommand\fra[2]{#1/#2}

\newcommand{\N}{\mathbb{N}}
\newcommand{\Abs}[1]{\left\vert#1\right\vert}
\newcommand{\abs}[1]{\vert #1 \vert}
\newcommand{\Norm}[1]{\Big\Vert #1 \Big\Vert}

\newcommand{\norm}[1]{\Vert #1 \Vert}
\newcommand{\bignorm}[1]{\big\Vert #1 \big\Vert}
\newcommand{\Bignorm}[1]{\Big\Vert #1 \Big\Vert}

\newcommand{\sothat}{\,\,{\rm ;}\ \,}
\newcommand{\ac}[1]{\noindent\textcolor{red}
{{\rm [\![}\mbox{{AC}$\blacktriangleright\!\!\blacktriangleright$}: { #1}{\rm ]\!]}}}

\renewcommand{\ac}[1]{}

\DeclareMathSymbol{\varGamma}{\mathord}{letters}{"00}
\DeclareMathSymbol{\varDelta}{\mathord}{letters}{"01}
\DeclareMathSymbol{\varTheta}{\mathord}{letters}{"02}
\DeclareMathSymbol{\varLambda}{\mathord}{letters}{"03}
\DeclareMathSymbol{\varXi}{\mathord}{letters}{"04}
\DeclareMathSymbol{\varPi}{\mathord}{letters}{"05}
\DeclareMathSymbol{\varSigma}{\mathord}{letters}{"06}
\DeclareMathSymbol{\varUpsilon}{\mathord}{letters}{"07}
\DeclareMathSymbol{\varPhi}{\mathord}{letters}{"08}
\DeclareMathSymbol{\varPsi}{\mathord}{letters}{"09}
\DeclareMathSymbol{\varOmega}{\mathord}{letters}{"0A}

\theoremstyle{plain}
\newtheorem{lemma}{Lemma}[section]
\newtheorem{theorem}{Theorem}[section]

\newtheorem{proposition}[lemma]{Proposition}

\theoremstyle{definition}
\newtheorem{definition}[lemma]{Definition}
\newtheorem{assumption}[lemma]{Assumption}

\theoremstyle{remark}
\newtheorem{remark}[lemma]{Remark}
\providecommand{\sep}{; }

\newcounter{step}

\makeatletter\@addtoreset{equation}{section}
\makeatletter\@addtoreset{lemma}{section}
\makeatletter\@addtoreset{theorem}{section}
\makeatother

\def\dist{\mathop{\rm dist}\nolimits}
\renewcommand{\Re}{\mathop{\rm{R\hskip -1pt e}}\nolimits}
\renewcommand{\Im}{\mathop{\rm{I\hskip -1pt m}}\nolimits}

\begin{document}

\title{
Spectral stability of 
small amplitude solitary waves 
of the Dirac equation with the Soler-type nonlinearity
}

\author{
{\sc Nabile Boussa{\"\i}d}
\\
{\it\small Universit\'e Bourgogne Franche-Comt\'e, 25030 Besan\c{c}on CEDEX, France}
\\ \\
{\sc Andrew Comech}
\\
{\it\small Texas A\&M University, College Station, Texas 77843, USA}
\\
{\it\small Institute for Information Transmission Problems, Moscow 127051, Russia}
}

\date{\version}

\maketitle

\begin{abstract}
We study the point spectrum
of the linearization at a solitary wave solution
$\phi\sb\omega(x)e^{-\mathrm{i}\omega t}$
to the nonlinear Dirac equation in $\mathbb{R}^n$, for all $n\ge 1$,
with the nonlinear term given by $f(\psi^*\beta\psi)\beta\psi$
(known as the Soler model).
We focus on the spectral stability,
that is, the absence of eigenvalues with positive real part,
in the non-relativistic limit $\omega\to m-0$,
in the case when
$f\in C^1(\mathbb{R}\setminus\{0\})$, $f(\tau)=|\tau|^\kappa+O(|\tau|^K)$
for $\tau\to 0$,
with $0<\kappa<K$.
For $n\ge 1$, we prove the spectral stability
of small amplitude solitary waves ($\omega\lessapprox m$)
for the charge-subcritical cases $\kappa\lessapprox 2/n$
(in particular, $1<\kappa\le 2$ when $n=1$)
and for the ``charge-critical case'' $\kappa=2/n$ (with $K>4/n$).

An important part of the stability analysis
is the proof of the absence of bifurcations
of nonzero-real-part eigenvalues
from the embedded threshold points at $\pm 2m\mathrm{i}$.
Our approach is based on constructing
a new family of exact bi-frequency solitary wave solutions
in the Soler model,
on using this family to determine the multiplicity of
$\pm 2\omega\mathrm{i}$ eigenvalues of the linearized operator,
and on the analysis of the behaviour of
``nonlinear eigenvalues'' (characteristic roots
of holomorphic operator-valued functions).
\end{abstract}


\begin{verse}
{\bf Keywords:\ }
nonlinear Dirac equation\sep
Soler model\sep
spectral stability\sep
stability of solitary waves
\end{verse}


\section{Introduction}
We study stability of solitary waves in the nonlinear Dirac equation
with the scalar self-interaction \cite{jetp.8.260,PhysRevD.1.2766},
known as the Soler model:
\begin{equation}\label{nld}
 \jj \p\sb t\psi=D\sb m\psi-f(\psi\sp\ast\beta\psi)\beta\psi,
\qquad
\psi(t,x)\in\C\sp N,
\quad
x\in\R\sp n,
\quad
n\ge 1.
\end{equation}
Here
the Dirac operator is given by
$D\sb m=
-\jj \bm\alpha\cdot\nabla_x+\beta m
$,
with $m>0$
and the self-adjoint $N\times N$
Dirac matrices
$\alpha\sp i$, $1\le i\le n$,
and $\beta$
chosen so that
$D_m^2=-\Delta+m^2$;
for details, see
\emph{notations} at the end of this section.
We assume that
$
f\in C^1(\R\setminus\{0\})$
is real-valued,
$f(\tau)=\abs{\tau}^\kappa+O(\abs{\tau}^K)$
as
$\tau\to 0$,
with $0<\kappa<K$.
The structure of the nonlinearity,
$f(\psi\sp\ast\beta\psi)\beta\psi$,
is such that
the equation is $\mathbf{U}(1)$-invariant and hamiltonian.

Given a solitary wave solution
$\psi(t,x)=\phi\sb\omega(x)e\sp{-\jj \omega t}$
to \eqref{nld},
with $\omega\in\R$ and
$
\phi\sb\omega\in H\sp 1(\R\sp n,\C\sp N),
$
we consider its perturbation,
$(\phi\sb\omega(x)+\rho(t,x))e\sp{-\jj \omega t}$,
and study the spectrum of the linearized equation on $\rho$
(that is, the spectrum of the linearization operator).
We will say that this particular solitary wave is
\emph{spectrally stable}
if the spectrum of the linearization operator
has no points in the right half-plane.
In the present work,
we prove the spectral stability
of small amplitude 
solitary waves
corresponding to the nonrelativistic limit
$\omega\lessapprox m$,\footnote{
We write $\omega\lessapprox m$
to indicate that
$\omega\in(m-\varepsilon,m)$
for some $\varepsilon>0$ small enough.
}
in the case
$\kappa\lessapprox 2/n$, $K>\kappa$
(``charge-subcritical'')
and $\kappa=2/n$, $K>4/n$
(``charge-critical'').
This is the first rigorous analytic result on spectral stability
of solitary wave solutions of the nonlinear Dirac equation
(with the exception of massive Thirring model
in (1+1)D  \cite{pelinovsky2014orbital,MR3462129}
which is completely integrable);
it opens the way to the proofs of asymptotic stability
in the nonlinear Dirac context.

The question of stability of solitary waves
is answered in many cases
for the nonlinear Schr\"odinger,
Klein--Gordon, and Korteweg--de Vries equations
(see e.g. the review \cite{MR1032250}).
In these systems,
at the points represented by solitary waves,
the hamiltonian function
is of finite Morse index.
In simpler cases, the Morse index is equal to one,
and the perturbations in the corresponding direction
are prohibited
by a conservation law
when the Kolokolov stability condition \cite{kolokolov-1973} is satisfied.
In other words,
the solitary waves could be demonstrated to correspond
to conditional minimizers of the energy
under the charge constraint; this results
not only in spectral stability
but also in orbital stability \cite{MR677997,weinstein1985modulational,shatah1985,MR820338,MR901236}.
The nature of stability of solitary wave solutions
of the nonlinear Dirac equation
seems completely different from this picture \cite[Section V]{ranada1983classical}.
In particular, the hamiltonian function is
not bounded from below, and is of infinite Morse index;
the NLS-type approach to stability fails.
As a consequence,
we do not know how to prove the \emph{orbital stability}
but via proving the asymptotic stability first. 
The only known exception is the above-mentioned
one-dimensional completely integrable massive  Thirring model,
where the orbital stability was proved by means
of a coercive conservation law \cite{pelinovsky2014orbital,MR3462129}
coming from higher-order integrals of motion.

The first numerical evidence of spectral stability of solitary waves
to the cubic nonlinear Dirac equation in (1+1)D
(known as the massive Gross--Neveu model)
was given in~\cite{MR2892774},
where the spectrum of the linearization at solitary waves
was computed via the Evans function technique;
no nonzero-real-part eigenvalues have been detected;
this result
was confirmed by numerical simulations of the dynamics
in \cite{lakoba2018numerical}.
A similar model
in dimension $1$ is given by the nonlinear coupled-mode equations;
the numerical analysis of spectral stability
of solitary waves in such models has been done in
\cite{PhysRevLett.80.5117,MR2217129,MR2513792}. 
In the absence of spectral stability,
one expects to be able to prove \emph{orbital instability},
in the sense of~\cite{MR901236};
in the context of the nonlinear Schr\"odinger equation,
such instability is proved in e.g. \cite{MR2356215,MR2916078}.
If instead a particular solitary wave is spectrally stable,
one hopes to prove the asymptotic stability.
Let us give a brief account on asymptotic stability results
in dispersive models with unitary invariance.
The asymptotic stability for the nonlinear Schr\"odinger equation
is
proved using the dispersive properties; see for instance the seminal works~\cite{MR1071238,MR1170476} for
small amplitude
solitary waves bifurcating from the ground state of the linear Schr\"odinger equation (thus with a potential) and \cite{MR1192111,MR1221351,MR1199635,MR1334139}
in the translation-invariant case,
in dimension $1$.  Under {\it ad hoc} assumptions on the spectral stability some extensions to the nonlinear Schr\"odinger equation in any dimension can be expected, see~\cite{MR1835384,MR2027616, MR2571965}. The analysis of the dynamics of excited states and possible relaxation to the ground state solution for small solitary waves (in any dimension) was considered in~\cite{MR1865414,tsai2002classification,tsai2002relaxation,tsai2002stable,MR1972870,soffer2004,MR2219305, MR2480603}.
These have been improved in~\cite{MR1488355,MR1799850,MR2101699,MR2187292,MR2182081,MR2276543,MR2351368,MR2300249,MR2402517,MR2422523,MR2443298,MR2542729,MR2557723,MR2528489,MR2805462,MR3180019}. 
This path is also developed for the nonlinear Dirac equation in
\cite{boussaid2006stable, boussaid2008asymptotic, boussaid2012stability,   MR2985264, MR3447012,MR3592683}.
The needed dispersive properties of Dirac type models have been studied in these references and separately in~\cite{MR1434039,machihara2005endpoint,dancona2007decay, MR2424390, MR2474172,MR2763073,MR2857360,cacciafesta2013,MR3134742,MR3404556,erdogan2017dirac,2016arXiv160905164E}. Note that the most famous class of dispersive estimates is the one of the Strichartz estimates; this class was commonly used as a major tool for well-posedness in some of the above references. We also refer, for the well-posedness problem, to the review~\cite{MR2883845}. The question of the existence of stationary solutions, related to the indefiniteness of the energy, is discussed in the review~\cite{MR2434346}. 

While the purely imaginary essential spectrum
of the linearization operator is readily available
via Weyl's theorem on the essential spectrum
(see \cite{MR3530581} for more details and for the background on the topic),
the discrete spectrum is much more delicate.
Our aim in this work is to investigate the
presence of eigenvalues with positive real part,
which would be responsible for the linear
instability of a particular solitary wave.
As $\omega$ changes, such eigenvalues
can bifurcate from the point spectrum
on the imaginary axis or even from the essential spectrum.
In~\cite{MR3530581},
we have already shown that
the bifurcations of eigenvalues from the essential spectrum
into the half-planes with $\Re\lambda\ne 0$
are only possible
from the embedded thresholds at $\pm\jj(m+\abs{\omega})$ (see e.g. \cite{PhysRevLett.80.5117}),
from the collisions of eigenvalues
on the imaginary axis,
from the embedded eigenvalues
(by \cite[Theorem 2.2]{MR3530581},
there are no embedded eigenvalues
beyond the embedded thresholds),
and from the point of the collision of
the edges of the continuous spectrum
at $\lambda=0$ when $\omega=\pm m$ \cite{MR3208458}
and at $\lambda=\pm m \jj $ when $\omega=0$ \cite{MR1897705}.

Let us mention that the linear instability
in the nonrelativistic limit $\omega\lessapprox m$
in the ``charge-supercritical'' case $\kappa>2/n$
(complementary to cases which we consider in this work)
follows from \cite{MR3208458};
the restrictions in that article were $\kappa\in\N$ and
$n\le 3$, but they are easily removed by using the
nonrelativistic asymptotics of solitary waves
obtained in \cite{MR3670258}.
By numerics of \cite{PhysRevLett.116.214101},
in the case of the pure-power nonlinearity
$f(\tau)=\abs{\tau}^\kappa$, $\kappa>2/n$,
the spectral instability disappears
when $\omega\in(0,m)$
becomes sufficiently small.

We note that quintic nonlinear Schr\"odinger equation
in (1+1)D and the cubic one in (2+1)D
are ``charge critical'' (all solitary waves have the same charge),
and as a consequence
the linearization at any solitary wave
has a $4\times 4$ Jordan block at $\lambda=0$,
resulting in dynamic instability of all solitary waves;
moreover, there is a blow-up phenomenon
in the charge-critical
as well as in the charge-supercritical cases;
see in particular
\cite{1971ZhPmR..14..564Z,ZhETF.68.465,MR0460850,weinstein1983,MR1048692}.
On the contrary, for the nonlinear Dirac
with the critical-power nonlinearity,
the charge of solitary waves is no longer constant:
by \cite{MR3670258}, one has
$\p\sb\omega Q(\phi\sb\omega)<0$ for $\omega\lessapprox m$,
where $Q(\phi\sb\omega)$
is the corresponding charge (\eqref{linear-b-def-Q} below).
As a consequence, the linearization at solitary waves
in the nonrelativistic limit
has no $4\times 4$ Jordan block,
which resolves into $2\times 2$ Jordan block
(corresponding to the unitary invariance)
and two purely imaginary eigenvalues.



Here is the plan of the present work.
The results are stated in Section~\ref{linear-b-sect-results}.
In Section~\ref{linear-b-sect-technical},
we introduce the linearization operator
and derive some technical results.
In Sections~\ref{linear-b-sect-vb-thresholds-1}
and~\ref{linear-b-sect-nonlinear},
we study bifurcations of eigenvalues
from the embedded thresholds
at $\pm(m+\abs{\omega})\jj$
in the nonrelativistic limit $\omega\to m-0$.
In particular,
we develop the theory of characteristic roots
of operator-valued holomorphic functions,
in the spirit of \cite{MR0041353,MR0300125,MR0312306,MR0313856}
(this is done in Section~\ref{linear-b-sect-nonlinear}).
The bifurcations of eigenvalues from the origin
are analyzed in Section~\ref{linear-b-sect-vb-zero}.

In Appendix~\ref{linear-b-sect-analytic-continuation},
we construct the analytic continuation
of the resolvent of the free Laplace operator,
extending the three-dimensional approach of \cite{MR0495958}
to all dimensions $n\ge 1$.
In Appendix~\ref{linear-b-sect-nls}
we give details on the spectral theory
for the nonlinear Schr\"odinger equation
linearized at a solitary wave.
The limiting absorption principle for a particular Schr\"odinger operator
near the threshold point is covered in Appendix~\ref{sect-lap-lm}.
In Appendix~\ref{linear-b-sect-Schur},
we state the Schur complement in the form which we need in the proofs.

\medskip

\noindent{\bf Notations.\ }
We denote
$\N=\{1,2,\dots\}$ and $\N_0=\{0\}\cup\N$.
For $\rho>0$,
an open disc of radius $\rho$ in the complex plane
centered at $z_0\in\C$
is denoted by
$
\mathbb{D}_\rho(z_0)=\{z\in\C\sothat \abs{z-z_0}<\rho\};
$
we also denote
$
\mathbb{D}_\rho=\mathbb{D}_\rho(0).
$
We denote $r=\abs{x}$
for $x\in\R^n$, $n\in\N$,
and, abusing notations,
we will also denote the operator of multiplication
with $\abs{x}$ and $\langle x\rangle=(1+\abs{x}^2)^{1/2}$
by $r$ and $\langle r\rangle$, respectively.

For $s,\,\alpha\in\R$,
we define the weighted Sobolev spaces
of $\C\sp N$-valued functions by
\[
H\sp\alpha\sb{s} (\R\sp n,\C\sp N)
=\big\{ u\in \mathscr{S}'(\R\sp n,\C\sp N),\,
\norm{u}\sb{H\sp\alpha\sb s}
<\infty\big\},
\qquad
\norm{u}\sb{ H\sp\alpha\sb s}=
\norm{\langle r\rangle\sp s\langle -\jj\nabla\rangle\sp\alpha u}\sb{L\sp 2}.
\]
We write
$L\sp{2}\sb{s}(\R\sp n,\C\sp N)$
for~$H\sp{\alpha}\sb{s}(\R\sp n,\C\sp N)$
with $\alpha=0$
and
$H\sp\alpha(\R\sp n,\C\sp N)$
for~$H\sp\alpha\sb s(\R\sp n,\C\sp N)$ with $s=0$.
For $u\in L\sp 2(\R\sp n,\C\sp N)$,
we denote $\norm{u}=\norm{u}\sb{L\sp 2}$.

For any pair of normed vector spaces $E$ and $F$,
let $\mathscr{B}(E,F)$
denote the set of bounded linear operators from $E$ to $F$.
For an unbounded linear operator $A$
acting in a Banach space $X$
with a dense domain $\dom(A)\subset X$,
the spectrum $\sigma(A)$
is the set of values $\lambda \in\C$ such that
the operator $A-\lambda:\;\dom(A)\to X$
does not have a bounded inverse.
The null space (the kernel) of $A$ is defined by
\[
\ker(A)
=
\{
v\in \dom(A)\sothat
A v=0
\}.
\]
The generalized null space of $A$ is defined by
\[
\frakL(A)
:=
\mathop{\cup}\limits\sb{k\in\N} \ker(A^k)
=
\mathop{\cup}\limits\sb{k\in\N}\{v\in \dom(A)\sothat A^j v\in \dom(A)\ \forall j<k,\ A^k v=0
\}.
\]
The discrete spectrum $\sigma\sb{\mathrm{d}}(A)$ is the
set of isolated eigenvalues
$\lambda\in\sigma(A)$
of finite algebraic multiplicity,
such that
$
\dim\frakL(A-\lambda)<\infty.
$
The essential spectrum $\sigma\sb{\mathrm{ess}}(A)$ is the
complementary set of discrete spectrum in the spectrum.
The point spectrum $\sigma\sb{\mathrm{p}}(A)$
is the set of eigenvalues (isolated or embedded
into the essential spectrum).

We denote the free Dirac operator by
\[
D\sb m
=-\jj \bm\alpha\cdot\nabla_x+\beta m
=-\jj\sum\sb{i=1}\sp{n}\alpha\sp i\frac{\p}{\p x\sp i}
+\beta m,
\qquad
m>0,
\]
and the massless Dirac operator by
$D_0=-\jj \bm\alpha\cdot\nabla_x$.
Above,
$\alpha\sp i$ and $\beta$
are self-adjoint $N\times N$ Dirac matrices
which satisfy
$
(\alpha\sp i)\sp 2=\beta\sp 2=I_{N},
$
$
\alpha\sp i\alpha\sp j
+\alpha\sp j\alpha\sp i=2\delta\sb{i j}I_{N},
$
$
\alpha\sp i\beta+\beta\alpha\sp i=0,
$
$
1\le i,\,j\le n$,
so that $D_m^2=(-\Delta+m^2)I_{N}$.
Here $I_{N}$ denotes the $N\times N$ identity matrix.
The anticommutation relations lead to e.g.
$
\tr\alpha\sp i
=\tr\beta^{-1}\alpha\sp i\beta=-\tr\alpha\sp i=0,
$
$1\le i\le n$,
and similarly $\tr\beta=0$;
together with $\sigma(\alpha\sp i)=\sigma(\beta)=\{\pm 1\}$,
this yields the conclusion that $N$ is even.
Let us mention that the Clifford algebra representation theory
(see e.g. \cite[Chapter 1, \S5.3]{MR1401125})
shows that there is a relation
$
N=2^{[(n+1)/2]}M,
$
$
M\in\N.
$
Without loss of generality,
we may assume that
$
\beta
=
\begin{bmatrix}
I_{{N/2}}&0\\0&-I_{{N/2}}
\end{bmatrix}.
$
Then the anticommutation relations
$\{\alpha\sp i,\beta\}=0$, $1\le i\le n$,
show that
the matrices $\alpha\sp i$
are block-antidiagonal,
$
\alpha\sp i
=
\begin{bmatrix}
0&\upsigma_i\sp\ast\\\upsigma_i&0
\end{bmatrix},
$
$
1\le i\le n,
$
where the matrices
$\upsigma_i$, $1\le i\leq n$,
satisfy
$
\upsigma_i\sp\ast\upsigma_j+\upsigma_j\sp\ast\upsigma_i
=2\delta\sb{i j},
$
$
\upsigma_i\upsigma_j\sp\ast
+\upsigma_j\upsigma_i\sp\ast
=2\delta\sb{i j},
$
$
1\le i,\,j\le n;
$
the Dirac operator is thus given by
\begin{eqnarray}\label{linear-b-def-dm}
D_m=-\jj\bm\alpha\cdot\nabla_x+\beta m
=
-\jj
\sum\sb{i=1}\sp{n}
\begin{bmatrix}
0&\upsigma_i\sp\ast\\\upsigma_i&0
\end{bmatrix}
\p\sb{x\sp i}
+m\begin{bmatrix}I_{{N/2}}&0\\0&-I_{{N/2}}\end{bmatrix},
\quad
m>0.
\end{eqnarray}



\medskip

\noindent{\bf Acknowledgment.\ }
We are indebted to the anonymous referee for
corrections and suggestions.
Support from the grant
ANR-10-BLAN-0101 of the French Ministry of Research is gratefully acknowledged by the first author.
During the work on the project,
the second author was partially supported by Universit\'{e} Bourgogne Franche--Comt\'e
and by the Metchnikov fellowship from the French Embassy.
Both authors acknowledge the financial support of the r\'egion Bourgogne Franche--Comt\'e.

\section{Main results}
\label{linear-b-sect-results}

We consider the nonlinear Dirac equation \eqref{nld},
\[
\jj\p\sb t\psi
=D_m\psi-f(\psi\sp\ast\beta\psi)\beta\psi,
\qquad
\psi(t,x)\in\C^N,
\qquad
x\in\R^n,
\]
with the Dirac operator $D_m$ of the form \eqref{linear-b-def-dm}.

\begin{assumption}\label{linear-b-ass-f-c0-c1}
One has
$f\in C^1(\R\setminus\{0\})\cap C(\R)$,
and there are
$\kappa>0$ and $K>\kappa$
such that
\[
\abs{f(\tau)-\abs{\tau}^{\kappa}}
=O(\abs{\tau}^{ K }),
\qquad
\abs{\tau f'(\tau)-\kappa\abs{\tau}^\kappa}=O(\abs{\tau}^K);
\qquad
\abs{\tau}\le 1.
\]
If $n\ge 3$, we additionally assume that
$\kappa<2/(n-2)$.
\end{assumption}

In the nonrelativistic limit $\omega\lessapprox m$,
the solitary waves to nonlinear Dirac equation
could be obtained as bifurcations
from the solitary wave solutions
$\varphi\sb\omega(y)e^{-\jj\omega t}$
to the nonlinear Schr\"odinger equation
\begin{eqnarray}\label{linear-b-nls-k}
\jj \dot\psi=-\frac{1}{2m}\Delta\psi-\abs{\psi}\sp{2\kappa}\psi,
\qquad
\psi(t,y)\in\C,
\quad
y\in\R\sp n.
\end{eqnarray}
By \cite{MR0454365,MR695535}
and \cite{MR734575} (for the two-dimensional case),
the stationary nonlinear Schr\"odinger equation
\begin{equation}\label{dirac-existence-def-uu}
-\frac{1}{2m}u
=-\frac{1}{2m}\Delta u
-\abs{u}\sp{2\kappa}u,
\qquad
u(y)\in\R,
\quad
y\in\R^n,
\quad
n\ge 1
\end{equation}
has
a strictly positive spherically symmetric exponentially decaying solution
\begin{eqnarray}\label{linear-b-def-uk}
u_\kappa\in C^2(\R^n)\cap H^1(\R^n)
\end{eqnarray}
(called the ground state)
if and only if
$0<\kappa<2/(n-2)$
(any $\kappa>0$ if $n\le 2$).
Moreover,
\begin{eqnarray}\label{u-k-is-smooth}
u_\kappa\in H^s(\R^n),
\qquad
\forall s<n/2+2
\end{eqnarray}
(see \cite[Lemma A.1]{MR3670258})
and there are constants $0<c_{n,\kappa}<C_{n,\kappa}<\infty$
(which also depend on the value of the constant $m>0$)
such that
\begin{eqnarray}\label{linear-b-estimates-on-uu}
c_{n,\kappa}
\langle y\rangle^{-(n-1)/2}e^{-\abs{y}}
\le
u_\kappa(y)
\le
C_{n,\kappa}
\langle y\rangle^{-(n-1)/2}e^{-\abs{y}},
\qquad
\forall y\in\R^n;
\end{eqnarray}
see e.g. \cite[Lemma 4.5]{MR3670258}.
We set
\begin{equation}\label{dirac-existence-Vhatdef}
 \hat V(t):= u_\kappa(\abs{t}),
\qquad
 \hat U(t):=-(2m)^{-1}\hat V'(t),
\qquad
t\in\R;
\end{equation}
above, a spherically symmetric function
$u_\kappa(y)$, $y\in\R^n$, is understood as a function of $r=\abs{y}$.
By~\eqref{dirac-existence-def-uu},
the functions $\hat V\in C^2(\R)$
and $\hat U\in C^1(\R)$
(which are even and odd, respectively)
satisfy
\begin{equation}\label{linear-b-def-hat-phi}
\frac{1}{2m}\hat V
+
\p\sb t \hat U + \frac{n-1}{t}\hat U
=|\hat V|\sp{2\kappa}\hat V,
\qquad
\p\sb t\hat V+2m\hat U =0,
\qquad
t\in\R,
\end{equation}
where $\hat U(t)/t$
at $t=0$ is understood in the limit sense,
$\lim\sb{t\to 0}\hat U(t)/t=\hat U'(0)$.

In the nonrelativistic limit $\omega\lessapprox m$,
the solitary wave solutions to \eqref{nld}
are obtained as bifurcations from $(\hat V,\,\hat U)$
\cite{MR3670258};
we start with summarizing their asymptotics.

\begin{theorem}\label{dirac-existence-theorem-solitary-waves-c1}
Let $n\in\N$,
$N=2^{[(n+1)/2]}$.
Assume that the function $f$ in \eqref{nld}
satisfies Assumption~\ref{linear-b-ass-f-c0-c1}
with some $\kappa,\,K$.
There is
$\upomega_1\in(m/2,\,m)$
such that for all $\omega\in (\upomega_1,\,m)$
there are solitary wave solutions
$\phi\sb\omega(x)e\sp{- \jj \omega t}$
to \eqref{nld},
such that
$
\phi\sb\omega\in H^2(\R^n,\C^N)\cap C(\R^n,\C^N),
$
$
\omega\in (\upomega_1,\,m),
$
with
\begin{eqnarray}\label{dirac-existence-phi-beta-phi-large}
\phi\sb\omega(x)\sp\ast\beta\phi\sb\omega(x)
\ge
\abs{\phi\sb\omega(x)}^2/2
\qquad
\forall x\in\R^n,
\quad
\forall\omega\in (\upomega_1,\,m),
\end{eqnarray}
and
\[
\norm{\phi\sb\omega}\sb{L^\infty(\R^n,\C^N)}
=O\big((m^2-\omega^2)^{\frac{1}{2\kappa}}\big),
\qquad
\omega\lessapprox m.
\]
More explicitly,
\begin{equation}\label{dirac-existence-sol-forms}
\phi\sb\omega(x)
=\begin{bmatrix}
 v(r,\omega)\xi\\
 \jj u(r,\omega)\,\frac{x}{r}\cdot\bm\upsigma\,\xi
 \end{bmatrix},
\qquad
r=\abs{x},
\qquad
\xi\in\C^{N/2},
\qquad
\abs{\xi}=1,
\end{equation}
where
$
 v(r,\omega) =\epsilon\sp{\frac 1 \kappa}
 V(\epsilon r,\epsilon),
$
$
 u(r,\omega) =\epsilon\sp{1+\frac 1 \kappa}
 U(\epsilon r,\epsilon),
$
$
r\ge 0,
$
with
$\epsilon=\sqrt{m^2-\omega^2}$,
$\lim\sb{r\to 0}u(r,\epsilon)=0$,
and
\[
V(t,\epsilon)=\hat V(t)+\tilde V(t,\epsilon),
\qquad
U(t,\epsilon)=\hat U(t)+\tilde U(t,\epsilon),
\qquad
t\in\R,
\quad
\epsilon>0,
\]
with $\hat V(t)$, $\hat U(t)$
defined in \eqref{dirac-existence-Vhatdef}.

There is $b_0<\infty$ such that
\begin{equation}\label{dirac-existence-phi-asymptotics}
\abs{V(t,\epsilon)}
+
\abs{U(t,\epsilon)}
\le
b_0
\langle t\rangle^{-(n-1)/2}
e^{-\abs{t}},
\qquad
\forall t\in\R,
\quad
\forall\epsilon\in(0,\upepsilon_1).
\end{equation}
There are $\gamma>0$ and $b_1<\infty$
such that
$\tilde W(t,\epsilon)
=\begin{bmatrix}\tilde V(t,\epsilon)
\\
\tilde U(t,\epsilon)\end{bmatrix}
$ satisfies
\begin{equation}\label{dirac-existence-v-u-tilde-small-better}
\norm{
e^{\gamma \langle t\rangle}
\tilde W
}\sb{H^1(\R,\R^2)}
\le b_1\epsilon\sp{2\varkappa},
\qquad
\forall\epsilon\in(0,\upepsilon_1),
\end{equation}
with
$\upepsilon_1=\sqrt{m^2-\upomega_1^2}$
and
\begin{eqnarray}\label{dirac-existence-def-varkappa}
\varkappa:=\min\big(1,{K}/{\kappa}-1\big).
\end{eqnarray}
There is a $C^1$ map
$\omega\mapsto \phi\sb\omega\in H^1(\R^n,\C^N)$,
with
\[
\p\sb\omega\phi\sb\omega\in H^1(\R^n,\C^N),
\qquad
\p\sb\epsilon \tilde W(\cdot,\epsilon)\in
H^1_{\mathrm{even}}(\R)\times H^1_{\mathrm{odd}}(\R)
\cap C^1(\R,\R^2),
\]
where
$H^1_{\mathrm{even}}(\R)$
and
$H^1_{\mathrm{odd}}(\R)$
denote functions from $H^1(\R)$
which are even and odd, respectively;
\begin{equation}\label{linear-b-norm-p-epsilon-w}
\norm{
e^{\gamma\langle t\rangle}
\p\sb\epsilon \tilde W(\cdot,\epsilon)
}\sb{H^1(\R,\R^2)}
=
O(\epsilon^{2\varkappa-1}),
\qquad
\epsilon\in(0,\upepsilon_1),
\end{equation}
and there is $c>0$ such that
\[
\norm{\p\sb\omega\phi\sb\omega}_{L^2(\R^n,\C^N)}^2
=c\epsilon^{-n+\frac{2}{\kappa}}
(1+O(\epsilon^{2\varkappa})),
\qquad
\omega=\sqrt{m^2-\epsilon^2},
\qquad
\epsilon\in(0,\upepsilon_1).
\]
Additionally, assume that
$\kappa,\,K$ from Assumption~\ref{linear-b-ass-f-c0-c1} satisfy
either
$
\kappa<2/n
$,
or
$
\kappa=2/n,
$
$K>4/n$.
Then there is $\omega_c<m$ such that
$\p\sb\omega Q(\phi\sb\omega)<0$
for all $\omega\in(\omega_c,m)$.
If instead
$
\kappa>2/n,
$
then there is
$\omega_c<m$
such that
$\p\sb\omega Q(\phi\sb\omega)>0$
for all $\omega\in(\omega_c,m)$.
\end{theorem}

Above,
\begin{eqnarray}\label{linear-b-def-Q}
Q(\psi)
=\int\sb{\R^n}\psi\sp\ast(t,x)\psi(t,x)\,dx
\end{eqnarray}
is the charge functional
which is conserved due to the $\mathbf{U}(1)$-invariance
of the nonlinear Dirac equation \eqref{nld}.

\begin{remark}
If $f$ satisfies Assumption~\ref{linear-b-ass-f-c0-c1},
then we may assume that
there are $c,\,C>0$
such that
\begin{eqnarray}\label{linear-b-ass-f}
\abs{f(\tau)-\abs{\tau}^{\kappa}}
\le
c\abs{\tau}^{ K },
\quad
\qquad
\abs{f(\tau)}
\le
(c+1)\abs{\tau}^{\kappa},
\ \qquad
\forall\tau\in\R,
\\
\label{linear-b-ass-fp}
\abs{\tau f'(\tau)-\kappa\abs{\tau}^\kappa}
\le
C\abs{\tau}^K,
\qquad
\abs{\tau f'(\tau)}
\le
(C+\kappa)\abs{\tau}^\kappa,
\qquad
\forall\tau\in\R.
\end{eqnarray}\label{linear-b-Rem:NewAssumptionOnf}
Indeed, we could achieve \eqref{linear-b-ass-f}
and \eqref{linear-b-ass-fp}
by modifying $f(\tau)$ for $\abs{\tau}>1$,
and since the
$L^\infty$-norm of the resulting family of solitary waves
goes to zero as $\omega\to m$
(see Theorem~\ref{dirac-existence-theorem-solitary-waves-c1}),
we could then take $\upomega_1\lessapprox m$
sufficiently close to $m$ so that
$\norm{\phi\sb\omega}\sb{L^\infty}$
remains smaller than one
for $\omega\in(\upomega_1,m)$.
\end{remark}


The eigenvalues of the linearization at solitary waves
$\phi\sb\omega e^{-\jj\omega t}$
with $\omega=\omega_j$, $\omega_j\to m$,
can only accumulate to
$\lambda=\pm 2 m\jj$ and $\lambda=0$;
see \cite[Theorem 2.19]{MR3530581}.
We are going to relate
the families of eigenvalues
of the linearized nonlinear Dirac equation
bifurcating from $\lambda=0$
and from $\lambda=\pm 2 m\jj$
to the eigenvalues
of the linearization of the nonlinear Schr\"odinger equation
at a solitary wave.
Given $u_\kappa(x)$,
a strictly positive spherically symmetric
exponentially decaying solution to \eqref{dirac-existence-def-uu},
then $u_\kappa(x)e\sp{-\jj \omega t}$
with $\omega=-\frac{1}{2m}$
is a solitary wave solution to the nonlinear Schr\"odinger equation
\eqref{linear-b-nls-k}.
The linearization at this solitary wave
(see \eqref{def-l0-l1})
is given by
\begin{equation}\label{linear-b-lin-nls}
\p\sb t
\begin{bmatrix}\Re\rho\\\Im\rho\end{bmatrix}
=
\begin{bmatrix}0&\eurl\sb{-}\\-\eurl\sb{+}&0\end{bmatrix}
\begin{bmatrix}\Re\rho\\\Im\rho\end{bmatrix}
\end{equation}
(see e.g. \cite{kolokolov-1973}),
with $\eurl\sb\pm$ defined by
\begin{equation}\label{linear-b-def-l-small-pm}
\eurl\sb{-}=\frac{1}{2m}-\frac{\Delta}{2m}-u_\kappa\sp{2\kappa},
\qquad
\eurl\sb{+}=\frac{1}{2m}-\frac{\Delta}{2m}-(1+2\kappa)u_\kappa\sp{2\kappa},
\qquad
\dom(\eurl\sb\pm)=H^2(\R^n).
\end{equation}

\begin{theorem}[Bifurcations from the origin at $\omega=m$]
\label{linear-b-theorem-scaled-lambda}
Let $n\ge 1$.
Let $f\in C\sp 1(\R\setminus\{0\})\cap C(\R)$
satisfy Assumption~\ref{linear-b-ass-f-c0-c1}
with some values of $\kappa$, $K$.
Let
$\phi\sb\omega e\sp{-\jj \omega t}$,
$\omega\in(\upomega_1,m)$,
be the family of solitary wave solutions to \eqref{nld}
constructed in Theorem~\ref{dirac-existence-theorem-solitary-waves-c1}.

Let $\big(\omega_j\big)\sb{j\in\N}$,
$\omega_j\in(\upomega_1,m)$,
be a sequence such that $\omega_j\to m$,
and assume that
$\lambda\sb j$ is an eigenvalue of the linearization of \eqref{nld}
at the solitary wave $\phi\sb{\omega_j}(x)e^{-\jj \omega_j t}$
(see Section~\ref{linear-b-sect-technical})
and that
\ac{ADDED:}
$\Re\lambda_j\ne 0$,
$
\lambda\sb j\to 0.
$
Denote
\[
\Lambda_j=
\frac{\lambda\sb j}{\epsilon_j^2}\in\C,
\qquad
\epsilon_j=(m\sp 2-\omega_j\sp 2)^{1/2},
\qquad
j\in\N,
\]
and let $\Lambda\sb 0\in\C\cup\{\infty\}$
be an accumulation point
of the sequence $(\Lambda_j)\sb{j\in\N}$.
Then:
\ac{TO BE REMOVED:
\[
\Lambda\sb 0
\in
\sigma\Big(
\begin{bmatrix}0&\eurl\sb{-}\\-\eurl\sb{+}&0\end{bmatrix}
\Big)
\cup\sigma(\jj\eurl\sb{-})
\cup\sigma(-\jj\eurl\sb{-});
\]
in particular, $\Lambda\sb 0\ne\infty$.
If moreover $N=2$, then $\Lambda\sb 0
\in\sigma\Big(
\begin{bmatrix}0&\eurl\sb{-}\\-\eurl\sb{+}&0\end{bmatrix}
\Big)$.
}
\begin{enumerate}
\item\label{linear-b-theorem-scaled-lambda-two}
\ac{REMOVED:
If $\Re\lambda\sb j\ne 0$ for all $j\in\N$,
then}
$
\Lambda\sb 0
\in
\sigma\sb{\mathrm{p}}
\Big(
\begin{bmatrix}0&\eurl\sb{-}\\-\eurl\sb{+}&0\end{bmatrix}
\Big)
\cap
\R
$.
\ac{Added:}
In particular, $\Lambda\sb 0\ne\infty$.
\item\label{linear-b-theorem-scaled-lambda-three}
\ac{REMOVED: If $\Re\lambda\sb j\ne 0$ for all $j\in\N$,
then
}
$\Lambda\sb 0=0$ could be an accumulation point
of $\big(\Lambda_j\big)_{j\in\N}$
only when $\kappa=2/n$
and $\p\sb\omega Q(\phi\sb\omega)>0$
for $\omega\in(\omega\sb\ast,m)$,
with some $\omega\sb\ast<m$.
If, moreover, $\Lambda\sb j\to\Lambda_0=0$,
then
$\lambda\sb j\in\R$ for all but finitely many $j\in\N$.
\end{enumerate}
\end{theorem}


In other words,
as long as $\p\sb\omega Q(\phi\sb\omega)<0$
for $\omega\lessapprox m$,
there can be no linear instability due to
bifurcations from the origin:
there would be no
eigenvalues $\lambda\sb j$
of the linearization
at solitary waves with $\omega_j\to m$
such that
$\Re\lambda\sb j\ne 0$, $\lambda\sb j\to 0$.
We prove Theorem~\ref{linear-b-theorem-scaled-lambda}
in Section~\ref{linear-b-sect-vb-zero}.

\begin{theorem}[Bifurcations from $\pm 2 m \jj$ at $\omega=m$]
\label{linear-b-theorem-2m}
Let $n\ge 1$.
Let $f\in C\sp 1(\R\setminus\{0\})\cap C(\R)$
satisfy Assumption~\ref{linear-b-ass-f-c0-c1}
with some values of $\kappa$, $K$.
Let
$\phi\sb\omega(x) e\sp{-\jj \omega t}$,
$\omega\in(\upomega_1,m)$,
be the family of solitary wave solutions to \eqref{nld}
constructed in Theorem~\ref{dirac-existence-theorem-solitary-waves-c1}.

Let
$\big(\omega_j\big)\sb{j\in\N}$,
$\omega_j\in(\upomega_1,m)$,
be a sequence such that
$\omega_j\to m$ and
assume that
$\lambda\sb j$ is an eigenvalues of the linearization of \eqref{nld}
at the solitary wave $\phi\sb{\omega_j}(x) e^{-\jj \omega_j t}$
(see Section~\ref{linear-b-sect-technical})
and that
$
\lambda\sb j\to 2m\jj
$.
Denote
\begin{equation}\label{linear-b-def-Lambda-j-0}
Z_j=
-\frac{2\omega_j+\jj\lambda_j}{\epsilon_j^2}
\in\C,
\qquad
\epsilon_j=(m^2-\omega_j^2)^{1/2},
\qquad
j\in\N,
\end{equation}
and let $Z_0\in\C\cup\{\infty\}$
be an accumulation point
of the sequence $(Z_j)\sb{j\in\N}$.
Then:
\begin{enumerate}
\item\label{linear-b-theorem-2m-1}
$Z_0
\in
\{\frac{1}{2m}\}\cup
\sigma\sb{\mathrm{d}}(\eurl\sb{-})
$.
In particular,
$Z_0\ne\infty$.
\item\label{linear-b-theorem-2m-2}
If the edge of the essential
spectrum of $\eurl\sb{-}$ at $1/(2m)$
is a regular point of the essential spectrum of $\eurl\sb{-}$
(neither a virtual level nor an eigenvalue),
then
$Z_0\ne 1/(2m)$.
\item\label{linear-b-theorem-2m-1a}
If $Z_j\to Z_0$
and $Z_0=0$,
then
$\lambda_j=2\omega_j\jj$
for all but finitely many $j\in\N$.
\end{enumerate}
\end{theorem}

We note that,
according to the definition \eqref{linear-b-def-Lambda-j-0},
\[
\lambda_j=\jj(2\omega_j+\epsilon_j^2 Z_j),
\qquad
j\in\N.
\]

\begin{remark}
For definitions and
more details on virtual levels
(also known in this context as
\emph{threshold resonances}
and
\emph{zero-energy resonances}),
see e.g.
\cite{MR544248} (for the case $n=3$),
\cite[Sections 5, 6]{MR1841744} (for $n=1,\,2$),
and
\cite[Sections 5.2 and 7.4]{MR2598115}
(for $n=1$ and $n\ge 3$).
For the practical purposes, in the case $V\not\equiv 0$,
$n\le 2$,
the threshold point is a \emph{virtual level}
if it corresponds to a \emph{virtual state};
a virtual state, in turn,
can be characterized as an $L^\infty$-solution
to $(-\Delta+V)\Psi=0$ which does not belong to $L^2$
(see e.g. \cite[Theorem 5.2 and Theorem 6.2]{MR1841744});
for $n=3$,
the virtual levels
can be characterized as $H^1_{-s}$-solutions
with some (and hence any) $s\in(1/2,3/2]$
\cite[Lemma 3.2]{MR544248}.
\end{remark}

\begin{remark}
We do not need to study the case
$\lambda\sb j\to -2m\jj$
since the eigenvalues of the linearization
at solitary waves are symmetric with respect
to real and imaginary axes; see e.g. \cite{MR3530581}.
\end{remark}

In other words, as long as
$\eurl\sb{-}$ has no nonzero point spectrum
and the threshold $z=1/(2m)$ is a regular point of the essential spectrum,
there can be no nonzero-real-part eigenvalues near $\pm 2m\jj$
in the nonrelativistic limit $\omega\lessapprox m$.
We prove
Theorem~\ref{linear-b-theorem-2m} in Section~\ref{linear-b-sect-vb-thresholds-1}.

Let us focus on the most essential point of our work.
It is of no surprise that the behaviour of eigenvalues
of the linearized operator
near $\lambda=0$,
in the nonrelativistic limit $\omega\lessapprox m$,
follows closely the pattern which one finds in the
nonlinear Schr\"odinger equation with the same nonlinearity;
this is the content of Theorem~\ref{linear-b-theorem-scaled-lambda}.
In its proof
in Section~\ref{linear-b-sect-vb-zero},
we will make this rigorous
by applying the rescaling and the Schur complement
to the linearization of the nonlinear Dirac equation
and recovering in the nonrelativistic limit
$\omega\to m$
the linearization of the nonlinear Schr\"odinger equation.
Consequently, the absence of eigenvalues with nonzero real part
in the vicinity of $\lambda=0$
is controlled by the Kolokolov condition
$
\p\sb\omega Q(\phi\sb\omega)<0,
$
$
\omega\lessapprox m
$
(see \cite{kolokolov-1973}).
In other words, in the limit $\omega\lessapprox m$,
the eigenvalue families $\lambda\sb a(\omega)$
of the nonlinear Dirac equation
linearized at a solitary wave
which satisfy $\lambda\sb a(\omega)\to 0$ as $\omega\to m$
are merely deformations
of the eigenvalue families $\lambda\sb a\sp{\mathrm{NLS}}(\omega)$
of the nonlinear Schr\"odinger equation
with the same nonlinearity
(linearized at corresponding solitary waves).

On the other hand,
by \cite[Theorem 2.19]{MR3530581},
there could be eigenvalue families of the
linearization of the nonlinear Dirac operator
which satisfy
$\lim\sb{\omega\to m}\lambda\sb a(\omega)=\pm 2 m\jj$.
Could these eigenvalues go off the imaginary axis
into the complex plane?
Theorem~\ref{linear-b-theorem-2m}
states that in the Soler model,
under certain spectral assumptions,
this scenario could be excluded.
Rescaling and the Schur complement approach
will show that there could be at most $N/2$
families of eigenvalues with nonnegative real part
(with $N$ being the number of spinor components)
bifurcating from each of $\pm 2m\jj$
when $\omega=m$;
this essentially follows from
Section~\ref{linear-b-sect-nonlinear} below
(see Lemma~\ref{linear-b-lemma-m-ast} below).
At the same time,
the linearization at a solitary wave
has eigenvalues $\lambda=\pm 2\omega\jj$,
each of multiplicity (at least) $N/2$;
this follows from the existence of bi-frequency solitary waves
\cite{MR3842864} in the Soler model.
Namely,
if there is a solitary wave solution
of the form \eqref{dirac-existence-sol-forms}
to the nonlinear Dirac equation \eqref{nld},
then there are also bi-frequency solitary wave solutions of the form
\begin{eqnarray}\label{linear-b-def-psi-xi-eta}
\psi(t,x)
=
\phi\sb{\omega,\xi}(x)
e^{-\jj\omega t}
+
\chi\sb{\omega,\eta}(x)
e^{\jj\omega t},
\quad
\xi,\,\eta\in\C^{N/2},
\ \ \abs{\xi}^2-\abs{\eta}^2=1,
\end{eqnarray}
where
\begin{eqnarray}\label{linear-b-s1}
\phi\sb{\omega,\xi}(x)
=
\begin{bmatrix}
v(r,\omega)\xi
\\
\jj\frac{x}{r}\cdot\bm\upsigma\,u(r,\omega)\xi
\end{bmatrix},
\qquad
\chi\sb{\omega,\eta}
=
\begin{bmatrix}
-\jj\frac{x}{r}\cdot\bm\upsigma\sp\ast\,u(r,\omega)\eta
\\
v(r,\omega)\eta
\end{bmatrix},
\end{eqnarray}
with $v(r,\omega)$ and $u(r,\omega)$
from \eqref{dirac-existence-sol-forms}.
For more details
and the relation to $\mathbf{SU}(1,1)$ symmetry group
of the Soler model, see \cite{MR3842864}.
The form of these bi-frequency solitary waves
allows us to conclude that
$\pm 2\omega\jj$ are eigenvalues of the
linearization at a solitary wave
of multiplicities $N/2$
(see Lemma~\ref{linear-b-lemma-two-omega-n} below).
Thus, we know exactly what happens to the eigenvalues
which might bifurcate from $\pm 2m\jj$:
they all turn into $\pm 2\omega\jj$ and stay on the imaginary axis.

\begin{remark}
The spectral stability properties
of bi-frequency solitary waves \eqref{linear-b-def-psi-xi-eta}
can be related to the spectral stability of standard,
one-frequency solitary waves \eqref{dirac-existence-sol-forms};
see \cite{MR3842864}.
\end{remark}

As the matter of fact,
since the points $\pm 2 m\jj$ belong to the essential spectrum,
the perturbation theory can not be applied immediately
for the analysis of families of eigenvalues
which bifurcate from $\pm 2 m\jj$.
We use the limiting absorption principle
to rewrite the eigenvalue problem
in such a way that the eigenvalue
no longer
appears as embedded.
When doing so, we find out that the eigenvalues
$\pm 2\omega \jj$
become isolated solutions to the \emph{nonlinear eigenvalue problem},
known as the \emph{characteristic roots}
(or, informally, \emph{nonlinear eigenvalues}).
To make sure that we end up with \emph{isolated}
nonlinear eigenvalues,
we need to be able to vary the spectral parameter
to both sides of the imaginary axis.
To avoid the jump of the resolvent at the essential spectrum,
we use the analytic continuation
of the resolvent in the exponentially weighted spaces.
Finally, we show that
under the circumstances of the problem
the \emph{isolated nonlinear eigenvalues}
can not bifurcate off the imaginary axis.
This part is based
on the theory of the characteristic roots of
holomorphic operator-valued functions
\cite{MR0041353,MR0300125,MR0312306,MR0313856};
see also
\cite{MR971506} and \cite[Chapter I]{MR1995773}.
Unlike in the above references,
we have to deal with unbounded operators.
As a result, we find it easier to develop our own approach;
see
Lemma~\ref{linear-b-lemma-characteristic-roots}
in Section~\ref{linear-b-sect-nonlinear}.
It is of utmost importance to us that
we have the explicit description
of eigenvectors
corresponding to $\pm 2\omega\jj$ eigenvalues
(see Lemma \ref{linear-b-lemma-two-omega-n}).
Knowing that
$\pm 2\omega\jj$ are eigenvalues
of the linearization operator
of particular multiplicity,
we will be able to conclude
that there could be no other eigenvalue families
starting from $\pm 2m\jj$;
in particular, no families of eigenvalues with nonzero real part.


We use
Theorems~\ref{linear-b-theorem-2m},
and~\ref{linear-b-theorem-scaled-lambda}
to prove the
spectral stability of small amplitude solitary waves.

\begin{theorem}[Spectral stability of solitary waves
of the nonlinear Dirac equation]
\label{linear-b-theorem-dirac-stability}
Let $n\ge 1$.
Let $f\in C^1(\R\setminus\{0\})\cap C(\R)$
satisfy Assumption~\ref{linear-b-ass-f-c0-c1},
with
$\kappa,\,K$ such that
either
$0<\kappa<2/n$, $K>\kappa$
(\emph{charge-subcritical case}),
or
$\kappa=2/n$
and
$K>4/n$
(\emph{charge-critical case}).
Further, assume that
$
\sigma\sb{\mathrm{d}}(\eurl\sb{-})=\{0\},
$
and that the threshold
$z=1/(2m)$
of the operator $\eurl\sb{-}$
is a regular point of its essential spectrum.
Let $\phi\sb\omega(x) e\sp{- \jj \omega t}$,
$\phi\sb\omega\in H\sp 2(\R^n,\C\sp N)$,
$\omega\lessapprox m$,
be the family of solitary waves
constructed in Theorem~\ref{dirac-existence-theorem-solitary-waves-c1}.
Then there is $\omega\sb\ast\in(0,m)$
such that for each
$\omega\in(\omega\sb\ast,m)$
the corresponding solitary wave is spectrally stable.
\end{theorem}

\begin{remark}
We note that,
if either
$\kappa<2/n$, $K>\kappa$,
or
$\kappa=2/n$,
$K>4/n$,
then,
by Theorem~\ref{dirac-existence-theorem-solitary-waves-c1},
for $\omega\lessapprox m$
one has $\p\sb\omega Q(\phi\sb\omega)<0$,
which is formally
in agreement with
the Kolokolov stability criterion~\cite{kolokolov-1973}.
\end{remark}

\begin{proof}
We consider the family of solitary wave solutions
$\phi\sb\omega e\sp{- \jj \omega t}$,
$\omega\lessapprox m$,
described in Theorem~\ref{dirac-existence-theorem-solitary-waves-c1}.
Let us assume that there is a sequence
$\omega_j\to m$
and a family of eigenvalues
$\lambda\sb j$
of the linearization at solitary waves
$\phi\sb{\omega_j}e^{-\jj\omega_j t}$
such that
$
\Re\lambda\sb j\ne 0.
$

By \cite[Theorem 2.19]{MR3530581},
the only possible accumulation points of the sequence
$\big(\lambda\sb j\big)\sb{j\in\N}$
are $\lambda=\pm 2 m\jj$ and $\lambda=0$.
By Theorem~\ref{linear-b-theorem-2m},
as long as
$\sigma\sb{\mathrm{d}}(\eurl\sb{-})=\{0\}$
and the threshold of $\eurl\sb{-}$
is a regular point of the essential spectrum,
$\lambda=\pm 2 m \jj$
can not be an accumulation point
of nonzero-real-part eigenvalues;
it remains to consider the case $\lambda\sb j\to \lambda=0$.
By Theorem~\ref{linear-b-theorem-scaled-lambda}~\itref{linear-b-theorem-scaled-lambda-two},
if
$\Re\lambda\sb j\ne 0$
and
$\Lambda\sb 0$
is an accumulation point of the sequence
$\Lambda_j={\lambda\sb j}/(m\sp 2-\omega_j\sp 2)$,
then
\begin{eqnarray}\label{linear-b-in}
\Lambda\sb 0
\in
\sigma\sb{\mathrm{p}}\Big(
\begin{bmatrix}0&\eurl\sb{-}\\-\eurl\sb{+}&0\end{bmatrix}
\Big)
\cap
\R.
\end{eqnarray}
Above, $\eurl\sb\pm$
(see \eqref{linear-b-def-l-small-pm})
correspond to the linearization at a solitary wave
$u_\kappa(x)e^{-\jj\omega t}$
with $\omega=-1/(2m)$
of the nonlinear Schr\"{o}dinger equation \eqref{linear-b-nls-k}.
For $\kappa\le 2/n$,
the spectrum of the linearization
of the corresponding NLS at a solitary wave is purely imaginary:
\[
\sigma\sb{\mathrm{p}}\Big(
\begin{bmatrix}0&\eurl\sb{-}\\-\eurl\sb{+}&0\end{bmatrix}
\Big)\subset\jj\R.
\]
We conclude from \eqref{linear-b-in}
that one could only have $\Lambda\sb 0=0$;
by Theorem~\ref{linear-b-theorem-scaled-lambda}~\itref{linear-b-theorem-scaled-lambda-three},
this would require that
$\kappa=2/n$
and
$\p\sb\omega Q(\phi\sb\omega)>0$ for $\omega\lessapprox m$.
On the other hand,
as long as $\kappa=2/n$ and $K>4/n$,
Theorem~\ref{dirac-existence-theorem-solitary-waves-c1}
yields
$\p\sb\omega Q(\phi\sb\omega)<0$ for $\omega\lessapprox m$,
hence $\Lambda\sb 0=0$ would not be possible.
We conclude that
there is no family of eigenvalues $(\lambda\sb j)\sb{j\in\N}$
with $\Re\lambda\sb j\ne 0$.
\end{proof}

\begin{remark}\label{linear-b-remark-not-too-small}
We can not claim the spectral stability
for all subcritical values $\kappa\in(0,2/n)$:
We need to require that $\sigma\sb{\mathrm{d}}(\eurl\sb{-})$
does not have nonzero eigenvalues
(this requirement is known as the ``gap condition''),
while  for small values of $\kappa>0$ the operator $\eurl\sb{-}$
has a rich discrete spectrum,
and potentially any of its eigenvalues
could become a source of nonzero-real-part eigenvalues
of linearization of the nonlinear Dirac at a small amplitude solitary wave.
Such cases would require a more detailed analysis.
In particular, in one spatial dimension,
we only prove the spectral stability for
$1<\kappa\le 2$;
the critical, quintic case ($\kappa=2$)
is included, but our proof formally
does not cover the cubic case $\kappa=1$
because of the virtual level at the threshold
in the spectrum of the linearization operator
corresponding to the one-dimensional cubic NLS.
The analytic proof
of the ``gap condition''
for the one-dimensional case
is given in \cite[Theorem 3.1]{chang2007}.

In higher dimensions, the situation is more difficult: there,
we are not aware of the analytic results on the gap condition.
Our numerics show that
$
\sigma\sb{\mathrm{p}}(\eurl\sb{-})=\{0\}
$
and the threshold $1/(2m)$ is a regular point
of the essential spectrum of $\eurl\sb{-}$,
with $\eurl\sb{-}$ corresponding to
the nonlinear Schr\"odinger equation
in $\R^n$
(thus the spectral hypotheses of
Theorem~\ref{linear-b-theorem-scaled-lambda}
and
Theorem~\ref{linear-b-theorem-2m}
are satisfied)
as long as
\[
\kappa>\kappa_n,
\qquad
\mbox{where}
\quad
\kappa_1=
1,\quad
\kappa_2\approx 0.621,\quad
\kappa_3\approx 0.461\,.
\]
Let us mention the related  numerical results.
In two dimensions,
according to \cite[Fig. 4]{chang2007},
indeed $\eurl\sb{-}$ does not have nonzero eigenvalues
for $\kappa>\kappa_2$ with $\kappa_2\approx 0.6$.
In three dimensions,
according to \cite[Fig. 3]{demanet2006numerical},
$\eurl\sb{-}$ does not have nonzero eigenvalues
for $\kappa>\kappa_3$ with some $\kappa_3< 2/3$.
\end{remark}

\section{Technical results}
\label{linear-b-sect-technical}
\subsection{The linearization operator}


We assume that
$f\in C\sp 1(\R\setminus\{0\})\cap C(\R)$
satisfies Assumption~\ref{linear-b-ass-f-c0-c1}
(recall Remark~\ref{linear-b-Rem:NewAssumptionOnf}).
Let $\phi\sb\omega(x)e^{-\jj\omega t}$
be a solitary wave solution to equation \eqref{nld}
of the form \eqref{dirac-existence-sol-forms},
with $\omega\in(\upomega_1,m)$,
where $\upomega_1\in(m/2,m)$
is from Theorem~\ref{dirac-existence-theorem-solitary-waves-c1}.
Consider the solution to \eqref{nld}
in the form of the Ansatz
$\psi(t,x)=(\phi\sb\omega(x)+\rho(t,x))e\sp{-\jj \omega t}$,
so that $\rho(t,x)\in\C\sp N$
is a small perturbation of the solitary wave.
The linearization at
the solitary wave $\phi\sb\omega(x)e\sp{-\jj \omega t}$
(the linearized equation on $\rho$)
is given by
\begin{equation}\label{linear-b-def-cal-l}
\jj\p\sb t\rho=\mathcal{L}(\omega)\rho,
\qquad
\mathcal{L}(\omega)
=D\sb m
-\omega-f(\phi\sb\omega\sp\ast\beta\phi\sb\omega)\beta
-2f'(\phi\sb\omega\sp\ast\beta\phi\sb\omega)
\Re(\phi\sb\omega\sp\ast\beta\,\,\cdot\,)
\beta\phi\sb\omega.
\end{equation}

\begin{remark}
Even if $f'(\tau)$ is not continuous at $\tau=0$,
there are no singularities in \eqref{linear-b-def-cal-l}
for solitary waves
with $\omega\lessapprox m$
constructed in Theorem~\ref{dirac-existence-theorem-solitary-waves-c1}:
in view of the bound
$f'(\tau)=O(\abs{\tau}^{\kappa-1})$
(see \eqref{linear-b-ass-fp})
and the bound from below
$\phi\sb\omega\sp\ast\beta\phi\sb\omega
\ge\abs{\phi\sb\omega}^2/2$
(see Theorem~\ref{dirac-existence-theorem-solitary-waves-c1}),
the last term in \eqref{linear-b-def-cal-l}
could be estimated by $O(\abs{\phi\sb\omega}^{2\kappa})$.
\end{remark}

Since
$\mathcal{L}(\omega)$
is not $\C$-linear,
in order to work with $\C$-linear operators,
we introduce the following self-adjoint matrices of size $2N\times 2N$:
\begin{equation}\label{linear-b-upalpha-upbeta}
\begin{array}{l}
\bmupalpha\sp i
=
\begin{bmatrix}
\Re\alpha\sp i&-\Im\alpha\sp i\\\Im\alpha\sp i&\Re\alpha\sp i
\end{bmatrix},
\qquad
1\le i\le n;
\\[3ex]
\bmupbeta
=
\begin{bmatrix}\Re\beta&-\Im\beta\\\Im\beta&\Re\beta\end{bmatrix},
\qquad
\eubJ
=\left[\begin{matrix}0&I_{N}\\-I_{N}&0\end{matrix}\right],
\end{array}
\end{equation}
where the real part of a matrix is the matrix
made of the real parts of its entries
(and similarly for the imaginary part of a matrix).
We denote
\begin{eqnarray}\label{linear-b-def-bmupphi}
\bmupphi\sb\omega(x)
=\begin{bmatrix}\Re\phi\sb\omega(x)\\\Im\phi\sb\omega(x)\end{bmatrix}
\in\R\sp{2N}.
\end{eqnarray}
We mention that
$
\eubJ \bmupalpha\cdot\nabla_x+\bmupbeta m
$
is the operator which corresponds to $D\sb m$
acting on $\R\sp{2N}$-valued functions.
Introduce the operator
\begin{equation}\label{linear-a-def-ll}
\eubL(\omega)
=
\eubJ \bmupalpha\cdot\nabla_x+\bmupbeta m
-\omega
-f(\bmupphi\sb\omega\sp\ast\bmupbeta\bmupphi\sb\omega)\bmupbeta
-2(\bmupphi\sb\omega\sp\ast\bmupbeta\,\,\cdot\,)
f'(\bmupphi\sb\omega\sp\ast\bmupbeta\bmupphi\sb\omega)
\bmupbeta\bmupphi\sb\omega.
\end{equation}
We extend this operator
by $\C$-linearity
from $L\sp 2(\R\sp n,\R\sp{2N})$
onto
\[
L\sp 2(\R\sp n,\C\sp{2N})
=L\sp 2(\R\sp n,\C\otimes\sb{\R}\R\sp{2N}),
\]
with domain
$\dom(\eubL(\omega))=H\sp 1(\R\sp n,\C\sp{2N})$
where $\eubL(\omega)$ is self-adjoint.
The linearization at the solitary wave
in \eqref{linear-b-def-cal-l}
takes the form
\begin{equation}\label{linear-b-nld-linear}
\p\sb t\bmuprho
=\eubJ\eubL(\omega)\bmuprho,
\qquad
\bmuprho(t,x)=\left[\begin{matrix}\Re\rho(t,x)\\\Im\rho(t,x)\end{matrix}\right]
\in\R\sp{2N},
\end{equation}
with
$\eubJ$ from \eqref{linear-b-upalpha-upbeta}
and with $\eubL$ from \eqref{linear-a-def-ll}.
By Weyl's theorem on the essential spectrum,
the essential spectrum of $\eubJ\eubL(\omega)$
is purely imaginary,
with the edges at the thresholds $\pm (m-\abs{\omega})\jj$;
see \cite{MR3530581} for more details.
There are also embedded thresholds $\pm (m+\abs{\omega})\jj$.

For the reader's convenience,
we record the results on the spectral subspace
of $\eubJ\eubL(\omega)$
corresponding to the zero eigenvalue:

\begin{lemma}\label{linear-b-lemma-dim-ker-nld}
$\ \ \displaystyle
\Span\left\{
\eubJ\bmupphi\sb\omega,
\ \p_{x^i}\bmupphi\sb\omega
\sothat
1\le i\le n
\right\}
\subset
\ker(\eubJ\eubL(\omega)),
$
\[
\Span\left\{
\eubJ\bmupphi\sb\omega,
\ \p\sb\omega\bmupphi\sb\omega,
\ \p_{x^i}\bmupphi\sb\omega,
\ \omega x\sp i\eubJ\bmupphi\sb\omega
-\frac 1 2\bmupalpha\sp i\bmupphi\sb\omega
\sothat
1\le i\le n
\right\}
\subset
\frakL(\eubJ\eubL(\omega)).
\]
\end{lemma}

The proof is in~\cite{MR3530581}.
We mention the following relations
(see e.g. \cite{MR3530581}):
\begin{equation}\label{linear-b-l-phi-phi}
\eubJ\eubL\eubJ\bmupphi\sb\omega=0,
\qquad
\eubJ\eubL\p\sb\omega\bmupphi\sb\omega=\eubJ\bmupphi\sb\omega;
\qquad
\eubJ\eubL\p_{x^i}\bmupphi\sb\omega=0,
\qquad
\eubJ\eubL
\Big(
\omega x^i\eubJ\bmupphi\sb\omega
-\frac{1}{2}\alpha^i\bmupphi\sb\omega
\Big)
=\p_{x^i}\bmupphi\sb\omega.
\end{equation}

\begin{remark}
This lemma does not give the complete characterization
of the kernel of $\eubJ\eubL(\omega)$;
for example, there are also eigenvectors due to the
rotational invariance
and purely imaginary eigenvalues passing through $\lambda=0$
at some particular values of $\omega$ \cite{PhysRevLett.116.214101}.
We also refer to the proof of Proposition~\ref{linear-b-prop-jl-limit} below,
which gives the dimension of the generalized null space
for $\omega\lessapprox m$.
\end{remark}

\begin{lemma}\label{linear-b-lemma-two-omega-n}
The operator $\mathcal{L}(\omega)$
from \eqref{linear-b-def-cal-l}
corresponding to the linearization
at a (one-frequency) solitary wave
has the eigenvalue $-2\omega$
of geometric multiplicity larger than or equal to $N/2$,
with the eigenspace containing the subspace
$
\Span\big\{
\chi\sb{\omega,\eta}
\sothat
\eta\in\C^{N/2}
\big\},
$
with $\chi\sb{\omega,\eta}$
defined in \eqref{linear-b-s1}.
The operator
$\eubJ\eubL(\omega)$
of the linearization at the solitary wave
(see \eqref{linear-b-nld-linear})
has eigenvalues $\pm 2\omega\jj$
of geometric multiplicity larger than or equal to $N/2$.
\end{lemma}

\begin{proof}
This could be concluded from the expressions
for the bi-frequency solitary waves
\eqref{linear-b-def-psi-xi-eta}
or verified directly.
Indeed, one has
\[
-2\omega
\chi\sb{\omega,\eta}
=
(-\jj\bm\alpha\cdot\nabla_x
+(m-f)\beta-\omega)
\chi\sb{\omega,\eta},
\]
and then one takes into account that
$
\phi\sb{\omega}(x)\sp\ast
\beta
\chi\sb{\omega,\eta}(x)
=0,
$
so that the last term in the expression \eqref{linear-b-def-cal-l}
vanishes when applied to $\chi\sb{\omega,\eta}$.
\end{proof}


\subsection{Scaling, projections, and estimates on the effective potential}

Let
\begin{eqnarray}\label{linear-b-def-projectors}
\pi\sb{P}=(1+\bmupbeta)/2,
\qquad
\pi\sb{A}=(1-\bmupbeta)/2,
\qquad
\pi\sp\pm=(1\mp\jj\eubJ)/2
\end{eqnarray}
be the projectors
corresponding to $\pm 1\in\sigma(\bmupbeta)$
(``particle'' and ``antiparticle'' components)
and to $\pm\jj\in\sigma(\eubJ)$
($\C$-antilinear and $\C$-linear).
These projectors commute;
we denote their compositions by
\begin{eqnarray}\label{linear-b-def-projectors-2}
\pi\sp\pm\sb{P}=\pi\sp\pm\circ\pi\sb{P},
\qquad
\pi\sp\pm\sb{A}=\pi\sp\pm\circ\pi\sb{A}.
\end{eqnarray}
With $\xi\in\C^{N/2}$, $\abs{\xi}=1$,
from Theorem~\ref{dirac-existence-theorem-solitary-waves-c1} (see \eqref{dirac-existence-sol-forms}),
we denote
\begin{eqnarray}
\label{linear-b-def-Xi}
\bmXi=\begin{bmatrix}\Re\xi\\0\\\Im\xi\\0\end{bmatrix}
\in\R^{2N},
\qquad
\abs{\bmXi}=1.
\end{eqnarray}
For the future convenience, we introduce the orthogonal projection
onto $\bmXi$:
\begin{eqnarray}\label{linear-b-def-Pi}
\Pi_\bmXi=\bmXi\langle\bmXi,\,\cdot\,\rangle\sb{\C^{2N}}\in\End(\C^{2N}).
\end{eqnarray}
We note that, since $\bmupbeta\bmXi=\bmXi$,
one has
\begin{eqnarray}\label{linear-b-p-p-p}
\Pi_\bmXi\circ\pi\sb{P}=\pi\sb{P}\circ\Pi_\bmXi=\Pi_\bmXi,
\qquad
\Pi_\bmXi\circ\pi\sb{A}=\pi\sb{A}\circ\Pi_\bmXi=0.
\end{eqnarray}

Denote by $\bmuppsi\sb j\in H^1(\R^n,\C^{2N})$
eigenfunctions
which correspond to the eigenvalues
$\lambda\sb j\in\sigma\sb{\mathrm{p}}(\eubJ\eubL(\omega_j))$;
thus, one has
\begin{equation}\label{linear-b-j-d-m}
\eubL(\omega_j)\bmuppsi\sb j
=
\big(
\eubJ\bmupalpha\cdot\nabla_x
+\bmupbeta m-\omega_j+\eub{v}(\omega_j)
\big)\bmuppsi\sb j
=-\eubJ\lambda\sb j\bmuppsi\sb j,
\qquad
j\in\N,
\end{equation}
where the potential $\eub{v}(\omega)$ is defined by (cf. \eqref{linear-a-def-ll})
\begin{equation}\label{linear-b-def-eub-v}
\eub{v}(x,\omega)\bmuppsi(x)
=-f(\bmupphi\sb\omega\sp\ast\bmupbeta\bmupphi\sb\omega)
\bmupbeta\bmuppsi
-2\bmupphi\sb\omega\sp\ast\bmupbeta\bmuppsi\,
f'(\bmupphi\sb\omega\sp\ast\bmupbeta\bmupphi\sb\omega)
\bmupbeta\bmupphi\sb\omega.
\end{equation}

We will use the notations
\[
y=\epsilon_j x\in\R^n,
\]
where $\epsilon_j=\sqrt{m^2-\omega_j^2}$,
so that
$\eubJ\bmupalpha\cdot\nabla_x
=\epsilon_j\eubJ\bmupalpha\cdot\nabla_y
=:\epsilon_j\eubD_0$
and $\Delta=\epsilon_j^2\Delta\sb y$.
For $\eub{v}$ from \eqref{linear-b-def-eub-v},
define the potential
$\eubV(y,\epsilon)\in\End(\C\sp{2N})$
by
\begin{equation}\label{linear-b-def-eub-w}
\eubV(y,\epsilon)
=
\epsilon^{-2}\eub{v}(\epsilon^{-1}y,\omega),
\qquad
\omega=\sqrt{m^2-\epsilon^2},
\quad
y\in\R^n,
\quad
\epsilon\in(0,\upepsilon_1).
\end{equation}
We define
$
\bmPsi\sb j(y)
=\epsilon_j^{-n/2}\bmuppsi\sb j(\epsilon_j^{-1} y).
$
With $\eubV$ from \eqref{linear-b-def-eub-w},
$\eubL(\omega_j)$ is given by
\begin{equation}\label{linear-b-def-big-l}
\eubL(\omega_j)=
\epsilon_j
\eubD_0+\bmupbeta m-\omega_j
+\epsilon_j^2\eubV(\omega_j),
\qquad
j\in\N,
\end{equation}
and the relation \eqref{linear-b-j-d-m}
takes the form
\begin{equation}\label{linear-b-j-d-m-w}
\big(
\epsilon_j
\eubD_0+\bmupbeta m-\omega_j
+\eubJ\lambda\sb j
+\epsilon_j^2\eubV(\omega_j)
\big)\bmPsi\sb j
=0,
\qquad
j\in\N.
\end{equation}
We project \eqref{linear-b-j-d-m-w}
onto ``particle'' and ``antiparticle'' components
and
onto the $\mp\jj$ spectral subspaces of $\eubJ$
with the aid of projectors \eqref{linear-b-def-projectors}:
\begin{eqnarray}
\epsilon_j\eubD_0\pi\sp{-}\sb{A}\bmPsi\sb j
+(m-\omega_j-\jj\lambda\sb j)\pi\sp{-}\sb{P}\bmPsi\sb j
+\epsilon_j^2\pi\sp{-}\sb{P}\eubV\bmPsi\sb j
=0,
\label{linear-b-4p-minus}
\\
\epsilon_j\eubD_0\pi\sp{-}\sb{P}\bmPsi\sb j
-(m+\omega_j+\jj\lambda\sb j)\pi\sp{-}\sb{A}\bmPsi\sb j
+\epsilon_j^2\pi\sp{-}\sb{A}\eubV\bmPsi\sb j
=0,
\label{linear-b-4a-minus}
\\
\epsilon_j\eubD_0\pi\sp{+}\sb{A}\bmPsi\sb j
+(m-\omega_j+\jj\lambda\sb j)\pi\sp{+}\sb{P}\bmPsi\sb j
+\epsilon_j^2\pi\sp{+}\sb{P}\eubV\bmPsi\sb j
=0,
\label{linear-b-4p-plus}
\\
\epsilon_j\eubD_0\pi\sp{+}\sb{P}\bmPsi\sb j
-(m+\omega_j-\jj\lambda\sb j)\pi\sp{+}\sb{A}\bmPsi\sb j
+\epsilon_j^2\pi\sp{+}\sb{A}\eubV\bmPsi\sb j
=0.
\label{linear-b-4a-plus}
\end{eqnarray}

We need several estimates
on the potential $\eubV$.

\begin{lemma}\label{linear-b-lemma-w}
There is $C>0$ such that
for all $y\in\R^n$ and all $\epsilon\in(0,\upepsilon_1)$
one has the following pointwise bounds:
\begin{eqnarray*}
&&
\norm{\eubV(y,\epsilon)}\sb{\End(\C^{2N})}
\le C \abs{u_\kappa(y)}^{2\kappa},
\\[1ex]
&&
\norm{
\pi\sb{P}\circ\eubV(y,\epsilon)\circ\pi\sb{A}}\sb{\End(\C^{2N})}
+
\norm{
\pi\sb{A}\circ\eubV(y,\epsilon)\circ\pi\sb{P}}\sb{\End(\C^{2N})}
\le
C\epsilon
\abs{u_\kappa(y)}^{2\kappa},
\\[1ex]
&&
\norm{
\pi\sb{A}\circ\big(\eubV(y,\epsilon)+\abs{u_\kappa(y)}^{2\kappa}(1+2\kappa\Pi_\bmXi)\bmupbeta\big)\circ\pi\sb{A}
}\sb{\End(\C^{2N})}
\le
C\epsilon^{2\varkappa}
\abs{u_\kappa(y)}^{2\kappa},
\\[1ex]
&&
\norm{
\pi\sb{P}\circ\big(\eubV(y,\epsilon)+\abs{u_\kappa(y)}^{2\kappa}(1+2\kappa\Pi_\bmXi)\bmupbeta\big)\circ\pi\sb{P}
}\sb{\End(\C^{2N})}
\le
C\epsilon^{2\varkappa}
\abs{u_\kappa(y)}^{2\kappa}.
\end{eqnarray*}
\end{lemma}

Above,
$u_\kappa$ is the positive radially symmetric ground state
of the nonlinear Schr\"odinger equation \eqref{dirac-existence-def-uu};
$\varkappa=\min\big(1,K/\kappa-1\big)>0$
was defined in \eqref{dirac-existence-def-varkappa}.

\begin{proof}
The bound
on $\norm{\eubV(y,\epsilon)}\sb{\End(\C^{2N})}$
follows from
\eqref{linear-b-ass-f} and \eqref{linear-b-ass-fp}:
\[
\norm{\eubV(y,\epsilon)}\sb{\End(\C^{2N})}
\le
C\epsilon^{-2}
\big(\abs{f(v^2-u^2)}+v^2\abs{f'(v^2-u^2)}
\big)
\le
C\epsilon^{-2}
v^{2\kappa}
\le
C \abs{u_\kappa(y)}^{2\kappa},
\]
where $\omega=\sqrt{m^2-\epsilon^2}$,
$\epsilon\in(0,\upepsilon_1)$,
and
\[
\abs{v(\epsilon^{-1}\abs{y},\omega)}\le C\hat V(\abs{y})\epsilon^{\frac 1 \kappa},
\qquad
\abs{u(\epsilon^{-1}\abs{y},\omega)}\le C\hat V(\abs{y})\epsilon^{1+\frac 1 \kappa},
\]
in the notations from Theorem~\ref{dirac-existence-theorem-solitary-waves-c1},
with $\hat V(\abs{y})=u_\kappa(y)$ (see \eqref{dirac-existence-Vhatdef}).

The bounds on
$\norm{
\pi\sb{P}\circ\eubV(y,\epsilon)\circ\pi\sb{A}}\sb{\End(\C^{2N})}
$
and
$
\norm{
\pi\sb{A}\circ\eubV(y,\epsilon)\circ\pi\sb{P}}\sb{\End(\C^{2N})}
$
follow from
\[
\norm{\pi\sb{P}\circ\eubV\circ\pi\sb{A}}\sb{\End(\C^{2N})}
+
\norm{\pi\sb{A}\circ\eubV\circ\pi\sb{P}}\sb{\End(\C^{2N})}
\le
C\abs{
\epsilon^{-2}f'(\phi\sb\omega\sp\ast\beta\phi\sb\omega)v u},
\]
where
$\abs{f'(\phi\sb\omega\sp\ast\beta\phi\sb\omega)}
=\abs{f'(v^2-u^2)}
\le C \abs{v}^{2\kappa-2}$
by \eqref{dirac-existence-phi-beta-phi-large} 
and \eqref{linear-b-ass-fp},
with $v$, $u$ bounded as above.

Let us prove the bound on
$\norm{
\pi\sb{A}\circ\big(\eubV(y,\epsilon)+\abs{u_\kappa(y)}^{2\kappa}(1+2\kappa\Pi_\bmXi)\bmupbeta\big)\circ\pi\sb{A}
}\sb{\End(\C^{2N})}
$.
For any numbers
$\hat V>0$ and $\hat U,\,\tilde V,\,\tilde U\in\R$
which satisfy
\begin{equation}\label{linear-b-v-u-abstract-1}
\upepsilon_1\abs{U}\le \frac V 2,
\qquad
\abs{\tilde V}\le \min\Big(\frac 1 2,C\epsilon^{2\varkappa}\Big)\hat V,
\end{equation}
with $V=\hat V+\tilde V$ and $U=\hat U+\tilde U$,
there are the following bounds:
\begin{eqnarray}
\label{linear-b-f-minus-v}
&&
\hskip -20pt
\abs{f\big(\epsilon^{\frac{2}{\kappa}}(V^2-\epsilon^2 U^2)\big)-\epsilon^2\hat V^{2\kappa}}
\nonumber
\\
&&
\hskip -20pt
\le
\abs{f\big(\epsilon^{\frac{2}{\kappa}}(V^2\!\!-\!\epsilon^2 U^2)\big)
\!-\!\epsilon^2(V^2\!\!-\!\epsilon^2 U^2)^\kappa}
+
\epsilon^2\abs{(V^2\!\!-\!\epsilon^2 U^2)^\kappa\!-\!V^{2\kappa}}
+
\epsilon^2\abs{V^{2\kappa}\!\!-\!\hat V^{2\kappa}}
\nonumber
\\
&&
\hskip -20pt
\le
c
\epsilon^{2K/\kappa}(V^2-\epsilon^2 U^2)^K
+
O(\epsilon^2 V^{2(\kappa-1)}\epsilon^2 U^2)
+
O(\epsilon^2\hat V^{2\kappa-1}\tilde V)
\nonumber
\\
&&
\hskip -20pt
\le
C\epsilon^{2+2\varkappa}\hat V^{2\kappa},
\end{eqnarray}
where we used \eqref{linear-b-v-u-abstract-1}
and also applied \eqref{linear-b-ass-f};
similarly, using \eqref{linear-b-ass-fp},
\begin{eqnarray}
\label{linear-b-fp-minus-vp}
&&
\hskip -20pt
\abs{f'\big(\epsilon^{\frac{2}{\kappa}}(V^2-\epsilon^2 U^2)\big)\epsilon^{\frac{2}{\kappa}}V^2-\kappa\epsilon^2\hat V^{2\kappa}}
\nonumber
\\
&&
\hskip -20pt
\le
\big\vert
f'\!\big(\epsilon^{\frac{2}{\kappa}}(V^2\!\!-\!\epsilon^2 U^2)\!\big)
\!-\!\kappa\big(\epsilon^{\frac{2}{\kappa}}(V^2\!\!-\!\epsilon^2 U^2)\!\big)^{\kappa-1}
\big\vert
\epsilon^{\frac{2}{\kappa}}V^2
+
\kappa\epsilon^2\abs{(V^2\!\!-\!\epsilon^2 U^2)^{\kappa-1}V^2\!\!-\!\hat V^{2\kappa}}
\nonumber
\\
&&
\hskip -20pt
\le
C\abs{\epsilon^{\frac{2}{\kappa}}(V^2-\epsilon^2 U^2)}^{K-1}\epsilon^{\frac{2}{\kappa}}V^2
+
\kappa\epsilon^2\abs{(V^2\!\!-\!\epsilon^2 U^2)^{\kappa-1}V^2\!\!-\!V^{2\kappa}}
+
\kappa\epsilon^2\abs{V^{2\kappa}\!-\!\hat V^{2\kappa}}
\nonumber
\\[1ex]
&&
\hskip -20pt
\le C\epsilon^{2+2\varkappa}\hat V^{2\kappa};
\end{eqnarray}
\begin{eqnarray}
\label{linear-b-fp-minus-nothing}
&&
\hskip -50pt
\abs{f'\!\big(\epsilon^{\frac 2 \kappa}(V^2\!\!-\!\epsilon^2 U^2)\!\big)\epsilon^{1+ \frac 2 \kappa} U V}
\le
C\abs{\epsilon^{\frac 2 \kappa}(V^2\!\!-\!\epsilon^2 U^2)}^{\kappa-1}\epsilon^{1+\frac 2 \kappa}\abs{UV}
\le
C\epsilon^3\hat V^{2\kappa}.
\end{eqnarray}
By \eqref{linear-b-estimates-on-uu},
\eqref{dirac-existence-v-u-tilde-small-better},
and \eqref{dirac-existence-phi-asymptotics},
we may assume that $\upepsilon_1>0$
in Theorem~\ref{dirac-existence-theorem-solitary-waves-c1}
is sufficiently small so that
for $\epsilon\in(0,\upepsilon_1)$
the functions
$V(t,\epsilon)$, $U(t,\epsilon)$,
$\hat V(t)$, $\tilde V(t,\epsilon)$
satisfy \eqref{linear-b-v-u-abstract-1}, pointwise in $t\in\R$.
Then, by \eqref{linear-b-f-minus-v}, one has
$
\abs{
\epsilon^{-2}f-\hat V^{2\kappa}}
\le
C\epsilon^{2\varkappa}
\hat V^{2\kappa};
$
so,
\[
\norm{
\pi\sb{A}\circ\big(\eubV+\hat V^{2\kappa}(1+2\kappa\Pi_\bmXi)\bmupbeta\big)\circ\pi\sb{A}
}\sb{\End(\C^{2N})}
\le
\norm{
\pi\sb{A}\circ\big(\eubV-\hat V^{2\kappa}\big)\circ\pi\sb{A}
}\sb{\End(\C^{2N})}
\]
\[
\le C\abs{\epsilon^{-2}f-\hat V^{2\kappa}}
+
C\abs{
\epsilon^{-2}f' \epsilon^{2+\frac{2}{\kappa}}U^2
}
\le
C\epsilon^{2\varkappa}\hat V^{2\kappa};
\]
in the first inequality, we also took into account \eqref{linear-b-p-p-p}.
The above is understood pointwise in $y\in\R^n$;
$\hat V$ and $U$ are evaluated at $t=\abs{y}$,
$f$ and $f'$ are evaluated at
$\phi\sb\omega\sp\ast\beta\phi\sb\omega
=V^2-\epsilon^2 U^2$
and are estimated with the aid of \eqref{linear-b-ass-f} and \eqref{linear-b-ass-fp}.

The bound on
$\norm{
\pi\sb{P}\circ\big(\eubV(y,\epsilon)+\abs{u_\kappa(y)}^{2\kappa}(1+2\kappa\Pi_\bmXi)\bmupbeta\big)\circ\pi\sb{P}
}\sb{\End(\C^{2N})}$
is derived similarly.
We have:
\begin{eqnarray}
\nonumber
&&
\hskip -20pt
\norm{
\pi\sb{P}\circ\big(\eubV+\hat V^{2\kappa}(1+2\kappa\Pi_\bmXi)\bmupbeta\big)\circ\pi\sb{P}
}\sb{\End(\C^{2N})}
\\
\nonumber
&&
\hskip -20pt
\le
\norm{
-\epsilon^{-2}f
-\epsilon^{-2}2(\bmupphi\sb\omega\sp\ast\pi\sb{P}\,\cdot\,)
f'
\pi\sb{P}\bmupphi\sb\omega
+\hat V^{2\kappa}(1+2\kappa\Pi_\bmXi)
}\sb{\End(\C^{2N})}
\\
\nonumber
&&
\hskip -20pt
\le
\abs{
\epsilon^{-2}\!f-\hat V^{2\kappa}}
+
\abs{
\epsilon^{-2}2f' v^2-\hat V^{2\kappa}2\kappa
}
\le
C\abs{\epsilon^{-2}\!f-\hat V^{2\kappa}}
+
C\abs{\epsilon^{-2}\!f' u^2}
\le
C\epsilon^{2\varkappa}\hat V^{2\kappa}.
\end{eqnarray}
Above, we took into account that
\[
(\bmupphi\sb\omega\sp\ast\bmupbeta\pi\sb{P}\,\cdot\,)
\pi\sb{P}\bmupphi\sb\omega
=
((\pi\sb{P}\bmupbeta\bmupphi\sb\omega)\sp\ast\,\cdot\,)
\pi\sb{P}\bmupphi\sb\omega
=
((\pi\sb{P}\bmupphi\sb\omega)\sp\ast\,\cdot\,)
\pi\sb{P}\bmupphi\sb\omega
=v^2\Pi_\bmXi,
\]
with $\Pi_\bmXi$ from \eqref{linear-b-def-Pi}
and $v=v(\abs{x},\omega)$
from \eqref{dirac-existence-sol-forms},
since
\[
\frac 1 2 (1+\beta)\phi\sb\omega(x)
=v(\abs{x},\omega)\begin{bmatrix}\xi\\0\end{bmatrix},
\qquad
\mbox{hence}
\quad
\pi\sb{P}\bmupphi\sb\omega(x)=v(\abs{x},\omega)\bmXi.
\qedhere
\]
\end{proof}

\section{Bifurcations from the origin}
\label{linear-b-sect-vb-zero}

Here we prove Theorem~\ref{linear-b-theorem-scaled-lambda}.

\begin{lemma}\label{linear-b-lemma-lambda-small}
Let
$\omega_j\in(0,m)$,
$j\in\N$;
$\omega_j\to m$.
If there are eigenvalues
$
\lambda_j\in\sigma\sb{\mathrm{p}}\big(\eubJ\eubL(\omega_j))\big),
$
\ac{we need this:}
$\Re\lambda_j\ne 0$,
$
j\in\N,
$
such that
$\lim\limits\sb{j\to\infty}\lambda_j=0$,
then the sequence
\[
\Lambda_j:=\frac{\lambda_j}{\epsilon_j^2},
\qquad
j\in\N
\]
does not have the accumulation point at infinity.
\end{lemma}

\begin{proof}
Due to the exponential decay of solitary waves
stated in Theorem~\ref{dirac-existence-theorem-solitary-waves-c1},
there is $C>0$ and $s>1/2$ such that
\begin{equation}\label{linear-b-v-to-zero-two}
\norm{ \langle r\rangle^{2s}\eubV(\cdot,\omega)}\sb{L\sp\infty(\R\sp n,\End(\C\sp{2N}))}
\le C,
\qquad
\forall\omega\in(0,m).
\end{equation}
Let $\bmPsi_j\in L\sp 2(\R\sp n,\C\sp{2N})$, $j\in\N$,
be the eigenfunctions
of $\eubJ\eubL(\omega_j)$
corresponding to $\lambda_j$;
we then have
\[
(\epsilon_j\eubD_0+\bmupbeta m-\omega_j+ \eubJ\lambda_j)\bmPsi_j
=-\epsilon_j^2\eubV(\omega_j)\bmPsi_j
\]
(see \eqref{linear-b-j-d-m-w}).
Applying to this relation
$\pi\sp\pm=(1\mp\jj\eubJ)/2$
and denoting
$\bmPsi_j\sp\pm=\pi\sp\pm\bmPsi_j$,
we arrive at the system
\begin{equation}\label{linear-b-d-omega-pm}
\begin{array}{l}
(\epsilon_j\eubD_0+\bmupbeta m-\omega_j+ \jj \lambda_j)\bmPsi_j\sp{+}
=-\epsilon_j^2\pi\sp{+} \eubV(\omega_j)\bmPsi_j,
\\[2ex]
(\epsilon_j\eubD_0+\bmupbeta m-\omega_j-\jj \lambda_j)\bmPsi_j\sp{-}
=-\epsilon_j^2\pi\sp{-} \eubV(\omega_j)\bmPsi_j.
\end{array}
\end{equation}
Due to $\omega_j\to m$,
without loss of generality,
we can assume that
$\omega_j>m/2$
for all $j\in\N$.
Since
the spectrum $\sigma(\eubJ\eubL)$ is symmetric with respect to
real and imaginary axes,
we may assume, without loss of generality,
that
$
\Im\lambda_j\ge 0
$ for all
$j\in\N$,
so that $\Re(\jj\lambda_j)\le 0$ (see Figure~\ref{linear-b-fig-spectrum-dm}).
Since $\lambda_j\to 0$,
we can also assume that
$\abs{\lambda_j}\le m/2$ for all $j\in\N$.
With
$\epsilon_j\eubD_0+\bmupbeta m-\omega_j= D\sb m-\omega_j$
(considered in $L^2(\R^n,\C\otimes_{\R}\R^{2N})$)
being self-adjoint,
one has
\begin{equation}\label{linear-b-res-b}
\norm{(\epsilon_j\eubD\sb 0+\bmupbeta m-\omega_j-\jj \lambda_j)\sp{-1}}
=
\frac{1}{\mathop{\mathrm{dist}}( \jj \lambda_j,\sigma(D\sb m-\omega_j))}
=\frac{1}{\abs{m-\omega_j-\jj \lambda_j}},
\qquad
\forall j\in\N.
\end{equation}

\begin{figure}[ht]
\begin{center}
\setlength{\unitlength}{1pt}
\begin{picture}(-100,0)(0,10)
\font\gnuplot=cmr10 at 10pt
\gnuplot
\put(-45,8){$\scriptstyle \jj \lambda\sb j$}
\put(38,-8){$\scriptstyle -\jj \lambda\sb j$}
\put(-40,3){\circle*{4}}
\put(40,-3){\circle*{4}}
\put(-2,5){$\scriptstyle 0$}
\put( 12,5){$\scriptstyle m-\omega_j$}
\put(-177,5){$\scriptstyle -m-\omega_j$}
\put(-160, 0){\line(1,0){180}}
\linethickness{2pt}
\put(20,0){\line(1,0){60}}
\put(-160,0){\line(-1,0){30}}
\linethickness{4pt}
\put(20,0){\line(1,0){1}}
\put(0,0){\line(1,0){1}}
\put(-160,0){\line(-1,0){1}}
\end{picture}
\end{center}
\caption{
\small
$\ \sigma(D\sb m-\omega_j)\ $
and $\pm \jj \lambda\sb j\ $.
}
\label{linear-b-fig-spectrum-dm}
\end{figure}


From \eqref{linear-b-d-omega-pm} and \eqref{linear-b-res-b},
using the bound \eqref{linear-b-v-to-zero-two} on $\eubV$,
we obtain
\begin{eqnarray}\label{linear-b-psi-m-bound}
\norm{\bmPsi\sb j\sp{-}}\sb{L\sp 2}
\le
\frac{\epsilon_j^2\norm{\pi\sp{-} \eubV\bmPsi\sb j}}
{\abs{m-\omega_j-\jj \lambda\sb j}}
\le
C
\frac{\epsilon_j\sp 2}{\abs{m-\omega_j-\jj \lambda\sb j}}
\norm{\langle r\rangle^{-2s}\bmPsi\sb j},
\qquad
\forall j\in\N.
\end{eqnarray}
From \eqref{linear-b-d-omega-pm} we have:
\begin{eqnarray}\label{linear-b-m-omega-lambda}
\begin{bmatrix}
m-\omega_j+ \jj \lambda\sb j&\epsilon_j \eubD_0
\\
\epsilon_j \eubD_0&-m-\omega_j+ \jj \lambda\sb j
\end{bmatrix}
\begin{bmatrix}
\pi_P^{+}\bmPsi\sb j\\\pi_A^{+}\bmPsi\sb j
\end{bmatrix}
=-
\epsilon_j^2
\begin{bmatrix}
\pi_P\sp{+} \eubV\bmPsi\sb j
\\
\pi_A\sp{+} \eubV\bmPsi\sb j
\end{bmatrix}
,
\end{eqnarray}
hence
\begin{eqnarray*}
\begin{bmatrix}
\pi_P^{+}\bmPsi\sb j\\\pi_A^{+}\bmPsi\sb j
\end{bmatrix}
=
\begin{bmatrix}
m+\omega_j-\jj \lambda\sb j&\epsilon_j \eubD_0
\\
\epsilon_j \eubD_0&-(m-\omega_j+ \jj \lambda\sb j)
\end{bmatrix}
\big(
\mu_j
+\Delta_y
\big)^{-1}
\begin{bmatrix}
\pi_P\sp{+} \eubV\bmPsi\sb j
\\
\pi_A\sp{+} \eubV\bmPsi\sb j
\end{bmatrix}
,
\end{eqnarray*}
with
$\Delta_y=-\eubD_0^2$ and
\begin{eqnarray}\label{linear-b-def-mu-j}
\mu_j
=
\frac{
(-m-\omega_j+\jj \lambda\sb j)
(m-\omega_j+\jj \lambda\sb j)
}{\epsilon_j^2},
\qquad
j\in\N;
\end{eqnarray}
\ac{we need this:}
we note that
$\Im\mu_j=-2\epsilon_j^{-2}
(\omega_j+\Im\lambda_j)\Re\lambda_j\ne 0$
for all $j\in\N$,
so that $\mu_j+\Delta_y$ is invertible indeed.
We may assume that
$\inf\sb{j\in\N}\abs{\mu_j}>0$,
or else there would be nothing to prove:
if $\mu_j\to 0$,
we would have
$\abs{\lambda_j-\jj(m-\omega_j)}=o(\epsilon_j^2)$,
hence
$\abs{\lambda_j}=O(\epsilon_j^2)$.
Then, by the limiting absorption principle
(cf. Lemma~\ref{linear-b-lemma-lap-agmon}),
\[
\norm{
\langle r\rangle^{-s}
\pi^{+}\bmPsi\sb j
}
\le
C
\abs{\mu_j}^{-1/2}
\norm{
\langle r\rangle^s
\pi\sp{+} \eubV\bmPsi\sb j
}
+
C
\epsilon_j
\norm{
\langle r\rangle^s
\pi\sp{+} \eubV\bmPsi\sb j
}.
\]
The above,
together with \eqref{linear-b-psi-m-bound}
and the bound \eqref{linear-b-v-to-zero-two} on $\eubV$,
leads to
\[
\norm{
\langle r\rangle^{-s}
\bmPsi_j}
\le
\norm{\langle r\rangle^{-s}\bmPsi_j^{-}}+
\norm{\langle r\rangle^{-s}\bmPsi_j^{+}}
\le C
\Big(
\frac{\epsilon_j^2}
{\abs{m-\omega_j-\jj \lambda\sb j}}
+
\frac{1}{\abs{\mu_j}^{\frac{1}{2}}}
+\epsilon_j
\Big)
\norm{\langle r\rangle^{-s}\bmPsi_j}.
\]
If we had
$\abs{\lambda_j}/\epsilon_j^2\to\infty$,
then
$
{\abs{m-\omega_j-\jj \lambda\sb j}}
\ge {\abs{\lambda_j}}/2
$
for $j$ large enough,
hence
\[
\abs{\mu\sb j}
\ge m
\abs{m-\omega_j-\jj \lambda\sb j}/\epsilon_j^2
\ge m
\abs{\lambda\sb j}/(2\epsilon_j^2)
\]
for $j$ large enough (since $\omega_j\to m$ and $\lambda_j\to 0$
in \eqref{linear-b-def-mu-j}),
\[
\norm{
\langle r\rangle^{-s}
\bmPsi_j}
\le C
\Big(
\frac{\epsilon_j^2}{\abs{\lambda_j}}
+
\frac{\epsilon_j}{\abs{\lambda_j}^{\frac{1}{2}}}
+\epsilon_j
\Big)
\norm{\langle r\rangle^{-s}\bmPsi_j}.
\]
Due to $\abs{\lambda_j}/\epsilon_j^2\to\infty$,
the above relation would lead to a contradiction
since $\bmPsi_j\not\equiv 0$
for all $j\in\N$.
We conclude that the sequence
$
\Lambda_j={\lambda_j}/{\epsilon_j^2}
$
can not have an accumulation point at infinity.
\end{proof}



\begin{lemma}
\label{linear-b-lemma-laplace-s}
For any $\eta\in\C\setminus\overline{\R\sb{+}}$
there is $s_0(\eta)\in(0,1)$,
lower semicontinuous in $\eta$, such that
the resolvent
$(-\Delta-\eta)^{-1}$
defines a continuous mapping
\[
(-\Delta-\eta)^{-1}:\;
L^2_s(\R^n)\to H^2_s(\R^n),
\qquad
0\le s<s_0(\eta).
\]
\end{lemma}

\begin{proof}
Let $f\in L^2_s(\R^n)$;
define $u=(-\Delta-\eta)^{-1}f\in H^2(\R^n)$.
There is the identity
\begin{eqnarray}\label{linear-b-i-dent}
(-\Delta-\eta)(\langle r \rangle^s u)
+
[\langle r \rangle^s,-\Delta]u
=\langle r \rangle^s (-\Delta-\eta)u,
\end{eqnarray}
which holds in the sense of distributions.
Taking into account that
\[
\norm{[\langle r \rangle^s,-\Delta]u}
\le C\norm{u}_{H^1}O(s)
\le C\norm{f}_{L^2_s}O(s),
\]
one concludes from \eqref{linear-b-i-dent} that
$(-\Delta-\eta)(\langle r \rangle^s u)\in L^2(\R^n)$
and hence
$\langle r \rangle^s u\in L^2(\R^n)$,
both being bounded by $C\norm{f}_{L^2_s}$,
with some $C=C(\eta)<\infty$,
thus so is $\norm{u}_{H^2_s}$.
\end{proof}

It is convenient to introduce the following operator
acting in $L^2(\R^n,\C^{2N})$:
\begin{eqnarray}\label{linear-b-def-K}
\eubK
=
\frac{1}{2m}-
\frac{\Delta}{2m}
-u_\kappa\sp{2\kappa}(1+2\kappa\Pi_\bmXi),
\qquad
\dom(\eubK)=H^2(\R^n,\C^{2N}).
\end{eqnarray}
Above, $\Pi_\bmXi\in\End(\C^{2N})$
is the orthogonal projector onto $\bmXi\in\R^{2N}$;
see \eqref{linear-b-def-Xi}, \eqref{linear-b-def-Pi}.

\begin{lemma}\label{linear-b-lemma-limit-system}
\begin{enumerate}
\item
\label{linear-b-lemma-limit-system-1}
\[
\sigma(\eubJ\eubK\at{\range(\pi\sb{P})})
=
\begin{cases}
\sigma\Big(
\begin{bmatrix}0&\eurl\sb{-}\\-\eurl\sb{+}&0\end{bmatrix}
\Big)
,
&N=2;\\[3ex]
%
\sigma
\Big(
\begin{bmatrix}0&\eurl\sb{-}\\-\eurl\sb{+}&0\end{bmatrix}
\Big)
\cup
\sigma(\jj\eurl\sb{-})
\cup
\sigma(-\jj\eurl\sb{-}),
&
N\ge 4.
\end{cases}
\]
The equality also holds for the point spectra.

\item
\label{linear-b-lemma-limit-system-2}
One has:
\begin{eqnarray}\label{linear-b-ng-ng}
\displaystyle
\dim\frakL(\eubJ\eubK\at{\range(\pi\sb{P})})
=
\begin{cases}
2n+N,\quad \kappa\ne 2/n;
\\
2n+N+2,\quad \kappa=2/n.
\end{cases}
\end{eqnarray}
\end{enumerate}
\end{lemma}

Above,
$\eubK$ is from \eqref{linear-b-def-K}
and
$\eurl\sb\pm$
were introduced in
\eqref{linear-b-def-l-small-pm}.

\begin{proof}
We decompose
$L^2(\R^n,\range(\pi\sb{P}))$
into the direct sum
$\scrX_1\oplus\scrX_2$,
where
\begin{eqnarray}\label{linear-b-def-x1-x2}
\begin{array}{l}
\scrX_1
=
L^2\big(\R^n,\Span\big\{\bmXi,\eubJ\bmXi)\big\},
\\[1ex]
\scrX_2
=
L^2\big(\R^n,\big(\Span\big\{\bmXi,\eubJ\bmXi\big\}\big)\sp\perp
\cap\range(\pi\sb{P})\big).
\end{array}
\end{eqnarray}
Note that both $\bmXi$ and $\eubJ\bmXi$
belong to $\range(\pi\sb{P})$.
The proof
of Part~\ref{linear-b-lemma-limit-system-1}
follows once we notice that
$\eubJ\eubK$ is invariant in
the spaces $\scrX_1$ and $\scrX_2$,
and that
$\eubJ\eubK\at{\scrX_1}$ is represented in
$L^2\big(\R^n,\Span\big\{\bmXi,\eubJ\bmXi\big\}\big)$
by
$
\begin{bmatrix}0&\eurl\sb{-}\\-\eurl\sb{+}&0\end{bmatrix},
$
while
$\eubJ\eubK\at{\scrX_2}$
is represented in
$L^2\big(\R^n,\big(\Span\big\{\bmXi,\eubJ\bmXi\big\}\big)\sp\perp
\cap\range(\pi\sb{P})\big)$
by
$
I_{\frac N 2-1}\otimes\sb{\C}
\begin{bmatrix}0&\eurl\sb{-}\\-\eurl\sb{-}&0\end{bmatrix}.
$
We also notice that if $N=2$, then $\scrX_2=\{0\}$.

The proof of Part~\ref{linear-b-lemma-limit-system-2}
also follows from the above decomposition
and the relations
\[
\displaystyle
\dim\frakL(\eubJ\eubK\at{\scrX_1})
=
\dim\frakL
\Big(
\begin{bmatrix}0&\eurl\sb{-}\\-\eurl\sb{+}&0\end{bmatrix}
\Big)
=
\begin{cases}
2n+2,\quad \kappa\ne 2/n;
\\
2n+4,\quad \kappa=2/n
\end{cases}
\]
(cf. Lemma~\ref{linear-b-lemma-dim-ker})
and
\[
\displaystyle
\dim\frakL(\eubJ\eubK\at{\scrX_2})
=
(N-2)
\dim\frakL(\eurl\sb{-})
=
(N-2)
\dim\ker(\eurl\sb{-})
=
N-2.
\qedhere
\]
\end{proof}

\begin{remark}
$\eubJ\eubK\at{\range(\pi\sb{A})}$
is represented in
$L^2\big(\R^n,\range(\pi\sb{A})\big)$
by
$
I_{N/2}\otimes\sb\C
\begin{bmatrix}0&\eurl\sb{-}\\-\eurl\sb{-}&0\end{bmatrix}
$.
\end{remark}

Since
$\sigma(\eubJ\eubL)$ is symmetric with respect to
real and imaginary axes,
we assume without loss of generality
that
$\lambda\sb j$
satisfies
\begin{equation}\label{linear-b-im-lambda-positive}
\Im\lambda_j\ge 0,
\qquad
\forall
j\in\N.
\end{equation}
Passing to a subsequence, we assume that
\begin{eqnarray}
\Lambda\sb j=\frac{\lambda\sb j}{\epsilon_j^2}\to \Lambda_0\in\C.
\end{eqnarray}


\begin{lemma}
\label{linear-b-lemma-bounds}
\begin{enumerate}
\item
\label{linear-b-lemma-bounds-1}
If $\Lambda_0\not\in\sigma(\eubJ\eubK)$, then
there is $C>0$ such that
\[
\norm{\pi\sb{P}\bmPsi_j}
+
\epsilon_j^{-1}\norm{\pi\sb{A}\bmPsi_j}
\le
C\epsilon_j^{2\varkappa}\norm{u_\kappa^{\kappa}\bmPsi_j},
\qquad
\forall j\in\N.
\]
\item
\label{linear-b-lemma-bounds-2}
For $s>0$ sufficiently small
there is $C>0$ such that
\[
\norm{\pi\sb{P}\sp{-}\bmPsi_j}_{H^1_s}
+
\epsilon_j^{-1}\norm{\pi\sb{A}\sp{-}\bmPsi_j}_{H^1_s}
\le
C\norm{u_\kappa^{\kappa}\bmPsi_j},
\quad
\forall j\in\N.
\]
\item
\label{linear-b-lemma-bounds-3}
If $\Lambda_0\not\in\jj[1/(2m),+\infty)$, then
\[
\norm{\pi\sb{P}\sp{+}\bmPsi_j}_{H^1_s}
+
\epsilon_j^{-1}\norm{\pi\sb{A}\sp{+}\bmPsi_j}_{H^1_s}
\le
C\norm{u_\kappa^{\kappa}\bmPsi_j},
\quad
\forall j\in\N.
\]
\item
\label{linear-b-lemma-bounds-4}
If $\Lambda_0\in\jj[1/(2m),+\infty)$, then
\[
\norm{u_\kappa^{\kappa}\pi\sb{P}\sp{+}\bmPsi_j}
+
\epsilon_j^{-1}\norm{u_\kappa^{\kappa}\pi\sb{A}\sp{+}\bmPsi_j}
\le
C\norm{u_\kappa^{\kappa}\bmPsi_j},
\quad
\forall j\in\N.
\]
\end{enumerate}
\end{lemma}

\begin{proof}
Let us prove Part~\ref{linear-b-lemma-bounds-1}.
We divide
\eqref{linear-b-4p-minus}, \eqref{linear-b-4p-plus} by $\epsilon_j^2$
and
\eqref{linear-b-4a-minus}, \eqref{linear-b-4a-plus} by $\epsilon_j$,
arriving at
\[
\begin{bmatrix}
\frac{1}{m+\omega_j}+\eubJ\Lambda_j&\eubD_0
\\
\eubD_0&-m-\omega_j+\epsilon_j\sp 2\eubJ \Lambda_j
\end{bmatrix}
\begin{bmatrix}
\pi\sb{P}\bmPsi_j
\\
\epsilon_j^{-1}\pi\sb{A}\bmPsi_j
\end{bmatrix}
=
-\begin{bmatrix}
\pi\sb{P}\eubV(y,\epsilon_j)\bmPsi_j
\\
\epsilon_j\pi\sb{A}\eubV(y,\epsilon_j)\bmPsi_j
\end{bmatrix},
\]
which we rewrite as
\begin{eqnarray}\label{linear-b-mst-jl}
\begin{bmatrix}
\frac{1}{m+\omega_j}-u_\kappa^{2\kappa}(1+2\kappa\Pi_\bmXi)+\eubJ\Lambda_j&\eubD_0
\\
\eubD_0&-m-\omega_j
\end{bmatrix}
\begin{bmatrix}
\pi\sb{P}\bmPsi_j
\\
\epsilon_j^{-1}\pi\sb{A}\bmPsi_j
\end{bmatrix}
\nonumber
\\
=
-\begin{bmatrix}
\pi\sb{P}
(\eubV(y,\epsilon_j)+u_\kappa^{2\kappa}(1+2\kappa\Pi_\bmXi))
\bmPsi_j
\\
\epsilon_j\pi\sb{A}(\eubV(y,\epsilon_j)+\eubJ\Lambda_j)\bmPsi_j
\end{bmatrix}.
\end{eqnarray}
The Schur complement of the entry $-m-\omega_j$
is given by
\[
T_j=\frac{1}{m+\omega_j}+\Lambda_j\eubJ
-u_\kappa^{2\kappa}(1+2\kappa\Pi_\bmXi)-\frac{\Delta_y}{m+\omega_j}.
\]
If $\Lambda_j\to\Lambda_0\not\in\sigma(\eubJ\eubK)$,
then $T_j$ has a bounded inverse for $j$ sufficiently large,
and then the operator in the left-hand side of \eqref{linear-b-mst-jl}
has a bounded inverse (see \eqref{linear-b-Schur-22}).
Now the proof of Part~\ref{linear-b-lemma-bounds-1}
follows from \eqref{linear-b-mst-jl}
once we take into account the bounds from Lemma~\ref{linear-b-lemma-w}.


Let us prove Part~\ref{linear-b-lemma-bounds-2}.
We apply $\pi\sp\pm$ to \eqref{linear-b-mst-jl}
and rewrite the result as
\begin{equation}\label{linear-b-mst}
\begin{bmatrix}
\frac{1}{m+\omega_j}\pm\jj \Lambda_j&\eubD_0
\\
\eubD_0&-m-\omega_j
\end{bmatrix}
\begin{bmatrix}
\pi\sb{P}\sp{\pm}\bmPsi_j
\\
\epsilon_j^{-1}\pi\sb{A}\sp{\pm}\bmPsi_j
\end{bmatrix}
=
-\begin{bmatrix}
\pi\sb{P}\sp{\pm}\eubV(y,\epsilon_j)\bmPsi_j
\\
\epsilon_j\pi\sb{A}\sp{\pm}(\eubV(y,\epsilon_j)\pm\jj \Lambda_j)\bmPsi_j
\end{bmatrix}.
\end{equation}
Denote the matrix-valued operator
in the left-hand side of \eqref{linear-b-mst} by $A_j\sp\pm$.
The Schur complement of $(A_j\sp\pm)_{22}$
is given by
\begin{eqnarray}\label{linear-b-def-s-pm}
T_j\sp\pm
=(A_j\sp\pm)_{11}-(A_j\sp\pm)_{12}(A_j\sp\pm)_{22}^{-1}(A_j\sp\pm)_{21}
=
\frac{1}{m+\omega_j}
\pm
\jj\Lambda_j
-\frac{\Delta_y}{m+\omega_j}.
\end{eqnarray}
Since $\Im\Lambda_0\ge 0$
(cf. \eqref{linear-b-im-lambda-positive}),
$T_j\sp{-}$ is invertible in $L^2$
(except perhaps at finitely many values of $j$ which we disregard);
writing the inverse of $A_j\sp{-}$ in terms of $T_j\sp{-}$,
we conclude from \eqref{linear-b-mst} that
$
\norm{\pi\sb{P}\sp{-}\bmPsi\sb j}\sb{H^1}
+
\epsilon_j^{-1}
\norm{\pi\sb{A}\sp{-}\bmPsi\sb j}\sb{H^1}
\le
C\norm{\eubV\bmPsi_j}.
$
Moreover, by Lemma~\ref{linear-b-lemma-laplace-s},
for sufficiently small $s>0$,
\[
\norm{\pi\sb{P}\sp{-}\bmPsi\sb j}\sb{H^1_s}
+
\epsilon_j^{-1}
\norm{\pi\sb{A}\sp{-}\bmPsi\sb j}\sb{H^1_s}
\le
C\norm{\eubV\bmPsi_j}\sb{L^2_s}
\le
C\norm{u_\kappa^\kappa\bmPsi_j},
\qquad
j\in\N.
\]
This proves
Part~\ref{linear-b-lemma-bounds-2}.
As long as $\Lambda_0\not\in\jj[1/(2m),+\infty)$,
Part~\ref{linear-b-lemma-bounds-3} is proved in the same way
as Part~\ref{linear-b-lemma-bounds-2}.

To prove Part~\ref{linear-b-lemma-bounds-4}, we write
\begin{eqnarray}\label{linear-b-mst-3}
\begin{bmatrix}
\frac{1}{m+\omega_j}+\jj\Lambda_j
+\mu u_\kappa^{2\kappa}
&\eubD_0
\\
\eubD_0&-m-\omega_j
\end{bmatrix}
\begin{bmatrix}
\pi\sp{+}\sb{P}\bmPsi\sb j
\\
\epsilon_j^{-1}
\pi\sp{+}\sb{A}\bmPsi\sb j
\end{bmatrix}
=
-\begin{bmatrix}
\pi\sp{+}\sb{P}(\eubV-\mu u_\kappa^{2\kappa})\bmPsi\sb j
\\
\epsilon_j
\pi\sp{+}\sb{A}
(\eubV+\jj\Lambda_j)
\bmPsi_j
\end{bmatrix},
\end{eqnarray}
with some $\mu\ge 0$ to be specified.
The Schur complement of $-m-\omega_j$ is
\[
T_j=
\frac{1}{m+\omega_j}+\jj\Lambda_j
+\mu u_\kappa^{2\kappa}
-\frac{\Delta_y}{m+\omega_j}.
\]
We pick $\mu\ge 0$ such that
the threshold $z=1/(2m)$ is a regular point of the
essential spectrum of the operator
$\frac{1}{2m}+\mu u_\kappa^{2\kappa}-\frac{\Delta}{2m}$,
which is by \cite{MR1841744} a generic situation;
indeed, by \eqref{linear-b-def-m-kappa},
the virtual levels correspond to the situation when
$M=-1+\mu \mathcal{K}$ is not invertible, with $\mathcal{K}$
a compact operator.
(For $n\ge 3$, it is enough to take $\mu=0$ since $-\Delta$ has no
virtual level at $z=0$;
for $n\le 2$, it is enough to take $\mu>0$ small \cite{MR0404846}.)
Then, by Lemma~\ref{linear-b-lemma-j-n},
$u_\kappa^\kappa \circ T_j^{-1} \circ u_\kappa^\kappa$ (for $j$ large enough)
is bounded in $L^2$,
and the conclusion follows from \eqref{linear-b-mst-3}.
Above,
$u_\kappa^\kappa\circ\ldots\circ u_\kappa^\kappa$
denotes the compositions with
the operators of multiplication by $u_\kappa^\kappa$.
\end{proof}

Lemma~\ref{linear-b-lemma-bounds}~\itref{linear-b-lemma-bounds-1}
shows that if the sequence $\Lambda_j$ converges to a point
$\Lambda_0$ which were away from
$\sigma(\eubJ\eubK)$, then at most finitely many of
$\bmPsi_j$ could be different from zero.
Therefore, there is the inclusion
\[
\Lambda_0\in\sigma(\eubJ\eubK).
\]

\ac{REMOVED:
Together with the results on the spectrum of $\eubJ\eubK$
(cf. Lemma~\ref{linear-b-lemma-limit-system}),
this proves
Theorem~\ref{linear-b-theorem-scaled-lambda}~\itref{linear-b-theorem-scaled-lambda-1}.
}

\begin{proposition}\label{linear-b-prop-limit}
If $\Lambda_0=\lim\sb{j\to\infty}\fra{\lambda_j}{\epsilon_j^2}$
and $\Re\lambda\sb j\ne 0$ for all $j\in\N$, then
\[
\Lambda_0\in\sigma\sb{\mathrm{p}}(\eubJ\eubK)\cap\R.
\]
\end{proposition}

\begin{proof}
 From now on, we assume that the corresponding eigenfunctions $\bmPsi_j$
(cf. \eqref{linear-b-j-d-m-w})
are normalized:
\begin{equation}\label{linear-b-normalization}
\norm{\bmPsi_j}^2=1,
\qquad
j\in\N.
\end{equation}
By \eqref{linear-b-mst-jl},
$\epsilon_j^{-1}\eubD_0\pi\sb{A}\bmPsi\sb j$
is bounded in $L^2$ uniformly in $j\in\N$,
while by Lemma~\ref{linear-b-lemma-bounds}~\itref{linear-b-lemma-bounds-2},
\itref{linear-b-lemma-bounds-3}
and~\itref{linear-b-lemma-bounds-4}
so is
$u_\kappa^\kappa\epsilon_j^{-1}\pi\sb{A}\bmPsi\sb j$.
Again by \eqref{linear-b-mst-jl},
$u_\kappa^\kappa\eubD_0\pi\sb{P}\bmPsi_j$ is uniformly bounded in $L^2$.
It follows that both
$\epsilon_j^{-1}\pi\sb{P}\bmPsi\sb j$
and
$\epsilon_j^{-1}\pi\sb{A}\bmPsi\sb j$
belong to $H^1\sb{\mathrm{loc}}(\R^n,\C^{2N})$
and contain weakly convergent subsequences;
we denote their limits by
\begin{eqnarray}
\hat\eubP\in H^1\sb{\mathrm{loc}}(\R^n,\C^{2N}),
\qquad
\hat\eubA\in H^1\sb{\mathrm{loc}}(\R^n,\C^{2N}).
\end{eqnarray}
Passing to the limit in \eqref{linear-b-mst-jl}
and using the bounds from Lemma~\ref{linear-b-lemma-w},
we arrive at the following system
(valid in the sense of distributions):
\begin{eqnarray}\label{linear-b-mst-jl-limit}
\begin{bmatrix}
\frac{1}{2m}-u_\kappa^{2\kappa}(1+2\kappa\Pi_\bmXi)+\eubJ\Lambda_0&\eubD_0
\\
\eubD_0&-2m
\end{bmatrix}
\begin{bmatrix}
\hat\eubP
\\
\hat\eubA
\end{bmatrix}
=
0.
\end{eqnarray}



Let us argue that
if $\Re\lambda\sb j\ne 0$ for all $j\in\N$,
then
$\begin{bmatrix}\hat\eubP\\\hat\eubA\end{bmatrix}$ is not identically zero.
By
Lemma~\ref{linear-b-lemma-bounds}~\itref{linear-b-lemma-bounds-2},
using the compactness of the Sobolev embedding
$H^1_s\subset L^2$,
we conclude that there is an infinite subsequence (which we again enumerate by $j\in\N$)
such that
\begin{equation}\label{linear-b-p-to-p}
\pi\sb{P}\sp{-}\bmPsi\sb j\to\pi\sp{-}\hat\eubP\in H^1(\R^n,\C^{2N}),
\qquad
\epsilon_j^{-1}\pi\sb{A}\sp{-}\bmPsi\sb j\to\pi\sp{-}\hat\eubA\in H^1(\R^n,\C^{2N}),
\qquad
j\to\infty,
\end{equation}
with the strong convergence in $L^2$.

\begin{remark}\label{linear-b-remark-plus}
If additionally $\Lambda_0\not\in\jj[1/(2m),+\infty)$,
then $T\sp{+}_j$
from \eqref{linear-b-def-s-pm}
is also invertible; just like above, one concludes that
there is an infinite subsequence (which we again enumerate by $j\in\N$)
such that
\begin{eqnarray}\label{linear-b-p-to-p-plus}
\hskip -15pt
\pi\sb{P}\sp{+}\bmPsi\sb j\to\pi\sp{+}\hat\eubP\in H^1(\R^n,\C^{2N}),
\qquad
\epsilon_j^{-1}\pi\sb{A}\sp{+}\bmPsi\sb j\to\pi\sp{+}\hat\eubA\in H^1(\R^n,\C^{2N})
\end{eqnarray}
as $j\to\infty$,
with the strong convergence in $L^2$.
\end{remark}

\begin{lemma}
\label{linear-b-lemma-php2}
Let $J$ be skew-symmetric
and let $L$ be symmetric linear operators in a Hilbert space $\mathscr{X}$,
with $J^2=-I_{\mathscr{X}}$.
If $\lambda\in\sigma\sb{\mathrm{p}}(JL)\setminus \jj \R$
and $\bmPsi\in\mathscr{X}$ is a corresponding eigenvector,
then
$
\langle\bmPsi,L\bmPsi\rangle=0
$
and
$
\langle\bmPsi,J\bmPsi\rangle=0.
$
\end{lemma}

\begin{proof}
If $\Re\lambda\ne 0$,
then both sides of the identity
$
\langle\bmPsi,L\bmPsi\rangle
=-\lambda\langle\bmPsi,J\bmPsi\rangle
$
equal zero
since $\langle\bmPsi,L\bmPsi\rangle\in\R$,
$\langle\bmPsi,J\bmPsi\rangle\in \jj \R$.
\end{proof}

\begin{lemma}\label{linear-b-lemma-nonzero}
If $\Re\lambda\sb j\ne 0$ for all $j\in\N$,
then $\hat\eubP\sp{-}\ne 0$.
\end{lemma}

\begin{proof}
Lemma~\ref{linear-b-lemma-php2}
yields
$
0=\langle\bmPsi\sb j,\eubJ\bmPsi\sb j\rangle
=\jj \norm{\bmPsi\sb j\sp{+}}\sp 2
-\jj \norm{\bmPsi\sb j\sp{-}}\sp 2,
$
for all
$j\in\N$;
thus, by \eqref{linear-b-normalization},
\begin{equation}\label{linear-b-half}
\norm{\bmPsi\sb j\sp{+}}\sp 2=\norm{\bmPsi\sb j\sp{-}}\sp 2
=\norm{\bmPsi\sb j}\sp 2/2=1/2,
\qquad
\forall j\in\N.
\end{equation}
Therefore,\begin{equation}\label{linear-b-root-two}
\norm{\hat\eubP\sp{-}}^2
=
\lim\sb{j\to\infty}
\norm{\pi\sb{P}\sp{-}\bmPsi\sb j}^2
=
\lim\sb{j\to\infty}
\big(
\norm{\pi\sb{P}\sp{-}\bmPsi\sb j}^2
+
\norm{\pi\sb{A}\sp{-}\bmPsi\sb j}^2
\big)
=1/2,
\qquad
\forall j\in\N.
\end{equation}
Above, in the first two relations,
we took into account \eqref{linear-b-p-to-p}.
\end{proof}

Thus,
$\begin{bmatrix}\hat\eubP\\\hat\eubA\end{bmatrix}\in
L^2(\R^n,\C^{2N}\times \C^{2N})$ is not identically zero,
hence $\Lambda_0\in\sigma\sb{\mathrm{p}}(\eubJ\eubK)$.
It remains to prove that
\begin{eqnarray}\label{linear-b-lambda-0-il}
\Lambda\sb 0
\in\R.
\end{eqnarray}
Let us assume that, on the contrary,
$\Lambda\sb 0\in\sigma\sb{\mathrm{p}}(\eubJ\eubK)
\cap(\jj\R\setminus\{0\})$.
By \eqref{linear-b-im-lambda-positive},
it is enough to consider
\begin{equation}\label{linear-b-im-lambda-positive-1}
\Lambda\sb 0=\jj a,
\qquad
a>0.
\end{equation}
By \eqref{linear-b-root-two},
$
\norm{\hat\eubP\sp{-}}^2
=1/2$.
Since
\[
\norm{\hat\eubP\sp{+}}^2
\le
\lim\sb{j\to\infty}
\norm{\pi\sb{P}\sp{+}\bmPsi\sb j}^2
\le
\lim\sb{j\to\infty}
\big(\norm{\pi\sb{P}\sp{+}\bmPsi\sb j}^2+\norm{\pi\sb{A}\sp{+}\bmPsi\sb j}^2\big)
=1/2,
\]
we arrive at the inequality
\begin{equation}\label{linear-b-pm-zero}
\norm{\hat\eubP\sp{+}}\sp 2-\norm{\hat\eubP\sp{-}}\sp 2\le 0.
\end{equation}
 From the above and from
$\eubJ\eubK\hat\eubP=\jj a\hat\eubP$
it follows that
\begin{eqnarray}\label{linear-b-z-l-z-negative}
\langle\hat\eubP,\eubK\hat\eubP\rangle
=
\langle\hat\eubP,-\jj a\eubJ\hat\eubP\rangle
=a\langle\hat\eubP\sp{+},\hat\eubP\sp{+}\rangle
-a\langle\hat\eubP\sp{-},\hat\eubP\sp{-}\rangle
\le 0.
\end{eqnarray}

\begin{remark}
If $\Lambda\sb 0$ belongs to the spectral gap
of $\eubJ\eubK$
($\Lambda_0\in\jj\R$, $\abs{\Lambda_0}<1/(2m)$),
then
both $\pi\sb{P}\sp{\pm}\bmPsi_j$ and $\epsilon_j^{-1}\pi\sb{A}\sp{\pm}\bmPsi_j$
(up to choosing a subsequence)
converge to
$\pi\sp{\pm}\hat\eubP\in H^1(\R^n,\C^{2N})$
and
$\pi\sp{\pm}\hat\eubA\in H^1(\R^n,\C^{2N})$
strongly in $L^2$ (cf. Remark~\ref{linear-b-remark-plus}).
Then, by the above arguments,
$\norm{\hat\eubP\sp{\pm}}^2=1/2$
and hence
$\langle\hat\eubP,\eubJ\hat\eubP\rangle
=0=\langle\hat\eubP,\eubK\hat\eubP\rangle$.
\end{remark}

\begin{lemma}\label{linear-b-lemma-zlz}
If
$\Lambda\sb 0\in\jj\R\setminus\{0\}$,
$\Lambda\sb 0\in\sigma\sb{\mathrm{p}}
\Big(
\begin{bmatrix}0&\eurl\sb{-}\\-\eurl\sb{+}&0\end{bmatrix}\Big)$,
and $\eubz\in L^2(\R^n,\C^2)$ is a corresponding eigenfunction,
then
\[
\Big\langle
\eubz,\begin{bmatrix}\eurl\sb{+}&0\\0&\eurl\sb{-}\end{bmatrix}\eubz
\Big\rangle>0.
\]
\end{lemma}

\begin{proof}
Let $\eubz$ be an eigenfunction
which corresponds to
$\Lambda\sb 0
\in\sigma\sb{\mathrm{p}}\Big(
\begin{bmatrix}0&\eurl\sb{-}\\-\eurl\sb{+}&0\end{bmatrix}
\Big)\cap\jj\R$,
$\Lambda\sb 0\ne 0$.
Let $p,\,q\in L^2(\R^n,\C)$
be such that
$
\eubz=\begin{bmatrix}p\\ \jj q\end{bmatrix}$
and let
$\Lambda\sb 0= \jj a$ with $a\in\R\setminus\{0\}$.
Then
$
 \jj a
\begin{bmatrix}p\\ \jj q\end{bmatrix}
=
\begin{bmatrix}0&\eurl\sb{-}\\-\eurl\sb{+}&0\end{bmatrix}
\begin{bmatrix}p\\ \jj q\end{bmatrix}
$
results in
$
a p=\eurl\sb{-}q
$
and
$
a q=\eurl\sb{+}p
$
(note that $q\not\in\ker(\eurl\sb{-})$;
otherwise one would conclude
that $p\equiv 0$ and then also $q\equiv 0$,
so that $\eubz\equiv 0$,
hence not an eigenvector).
These relations lead to
\[
\langle p,\eurl\sb{+}p\rangle
=
a\langle p,q\rangle
=
a\overline{\langle q,p\rangle}
=
\overline{\langle q,a p\rangle}
=
\overline{\langle q,\eurl\sb{-}q\rangle}
=
\langle q,\eurl\sb{-}q\rangle,
\]
hence
\begin{equation*}
\Big\langle
\begin{bmatrix}p\\ \jj q\end{bmatrix},
\begin{bmatrix}\eurl\sb{+}&0\\0&\eurl\sb{-}\end{bmatrix}
\begin{bmatrix}p\\ \jj q\end{bmatrix}
\Big\rangle
=
\langle p,\eurl\sb{+}p\rangle
+
\langle q,\eurl\sb{-}q\rangle
=2\langle q,\eurl\sb{-}q\rangle>0,
\end{equation*}
where we took into account that $\eurl\sb{-}$
is semi-positive-definite
and that $q\not\in\ker(\eurl\sb{-})$.
\end{proof}

Since $\eubK$ is invariant in the subspaces
$\scrX_1$ and $\scrX_2$
defined in \eqref{linear-b-def-x1-x2},
where it is represented by
$\begin{bmatrix}\eurl\sb{+}&0\\0&\eurl\sb{-}\end{bmatrix}$
and by
a positive-definite operator
$I_{N/2-1}\otimes \sb{\C}
\begin{bmatrix}\eurl\sb{-}&0\\0&\eurl\sb{-}\end{bmatrix}$,
respectively
(see the proof of Lemma~\ref{linear-b-lemma-limit-system}),
it follows from Lemma~\ref{linear-b-lemma-zlz}
that the quadratic form
\[
\langle\,\cdot\,,\eubK\,\cdot\,\rangle
=
\langle\,\cdot\,,\eubK\,\cdot\,\rangle\at{\scrX_1}
+
\langle\,\cdot\,,\eubK\,\cdot\,\rangle\at{\scrX_2}
=
\langle\,\cdot\,,\eubK\,\cdot\,\rangle\at{\scrX_1}
+
\langle\,\cdot\,,
(I_{N-2}\otimes\eurl\sb{-})\,\cdot\,\rangle\at{\scrX_2}
\]
is strictly positive-definite on any eigenspace of $\eubJ\eubK$
corresponding to
$\Lambda_0=\jj a\in\sigma\sb{\mathrm{p}}(\eubJ\eubK)$,
$a>0$.
Therefore,
\begin{eqnarray}\label{linear-b-z-l-z-positive}
\langle\hat\eubP,\eubK\hat\eubP\rangle>0.
\end{eqnarray}
The relations \eqref{linear-b-z-l-z-negative}
and \eqref{linear-b-z-l-z-positive}
lead to a contradiction;
we conclude that \eqref{linear-b-lambda-0-il} is satisfied.
\end{proof}

Proposition~\ref{linear-b-prop-limit}
concludes the proof of
Theorem~\ref{linear-b-theorem-scaled-lambda}~\itref{linear-b-theorem-scaled-lambda-two}.

\medskip

\noindent
{\it The case $\Lambda\sb 0=0$.}
Now we turn to
Theorem~\ref{linear-b-theorem-scaled-lambda}~\itref{linear-b-theorem-scaled-lambda-three},
which treats the case $\Lambda\sb 0=0$.
Let us find the dimension of the spectral subspace
of $\eubJ\eubL(\omega)$
corresponding to all eigenvalues
which satisfy $\abs{\lambda}=o(\epsilon^2)$.

\begin{proposition}\label{linear-b-prop-jl-limit}
There is $\delta>0$ sufficiently small
and $\upepsilon_2>0$
such that
$
\p\mathbb{D}\sb{\delta\epsilon^2}\subset\rho(\eubJ\eubL)
$
for all $\epsilon\in(0,\upepsilon_2)$,
and for the Riesz projector
\begin{equation}\label{linear-b-def-p-delta-0}
P\sb{\delta,\epsilon}
=
-\frac{1}{2\pi\jj}
\oint\sb{\abs{\eta}=\delta\epsilon^2}
\big(\eubJ\eubL(\omega)-\eta\big)^{-1}
\,d\eta,
\qquad
\omega=\sqrt{m^2-\epsilon^2}
\end{equation}
one has
$
\rank P\sb{\delta,\epsilon}=2n+N
$
if $\kappa\ne 2/n$,
and $2n+N+2$ otherwise.
One also has
$
\dim\ker(\eubJ\eubL(\omega))
=
n+N-1.
$
\end{proposition}

\begin{remark}\label{linear-b-remark-informal}
Let us first give an informal calculation
of $\rank P\sb{\delta,\epsilon}$,
which is the dimension of the generalized null space of $\eubJ\eubL$.
By Lemma~\ref{linear-b-lemma-dim-ker-nld},
due to the unitary and translational invariance,
the null space is of dimension (at least)
$n+1$,
and there is (at least) a $2\times 2$ Jordan block
corresponding to each of these null vectors,
resulting in
$
\dim\frakL(\eubJ\eubL(\omega))\ge 2n+2.
$
Moreover,
the ground states
of the nonlinear Dirac equation
from Theorem~\ref{dirac-existence-theorem-solitary-waves-c1}
have additional degeneracy due to the choice
of the direction $\bm{\xi}\in\C\sp{N/2}$,
$\abs{\bm{\xi}}=1$
(cf. \eqref{dirac-existence-sol-forms}).
The tangent space
to the sphere on which $\bm{\xi}$ lives
is of complex dimension $N/2-1$.
(Let us point out that the real dimension
is $N-2$, as it should be;
we did not expect to have
the real dimension $N-1$
since we have already factored out the action of the unitary group.)
Thus,
\begin{eqnarray}\label{linear-b-informal}
\dim\frakL(\eubJ\eubL(\omega))
\ge
2(n+1)+2\big(N/2-1\big)
=
2n+N,
\qquad
\omega\lessapprox m.
\end{eqnarray}
Whether this is a strict inequality,
depends on the Kolokolov condition
$\p\sb\omega Q(\phi\sb\omega)=0$
which indicates the jump by $2$ in size of the Jordan block
corresponding to the unitary symmetry,
and on the energy vanishing $E(\phi\sb\omega)=0$,
which indicates jumps in size of Jordan blocks
corresponding to the translation symmetry \cite{MR3311594}.
\end{remark}

\begin{proof}[Proof of Proposition~\ref{linear-b-prop-jl-limit}]
Let $\delta>0$ be such that
\[
\overline{\mathbb{D}}\sb{\delta}\cap\sigma\Big(
\begin{bmatrix}0&\eurl\sb{-}\\-\eurl\sb{+}&0\end{bmatrix}
\Big)
=\{0\}.
\]
Let us define the operator
\begin{equation}\label{linear-b-def-l-script}
\mathscr{L}(\omega)
=\epsilon^{-2}\eubL(\omega)
=\epsilon^{-1}\eubD_0
+\epsilon^{-2}(\bmupbeta m-\omega)
+\eubV(y,\omega),
\qquad
\omega\in(\upomega_1,m)
\end{equation}
(cf. \eqref{linear-b-def-big-l}),
where
$y=\epsilon x$,
$\epsilon=\sqrt{m^2-\omega^2}$,
and
$\eubD_0$
is the Dirac operator in the variables $y=\epsilon x$
(we recall that
$\epsilon\eubD_0=\epsilon\eubJ\bmupalpha\cdot\nabla_y
=\eubJ\bmupalpha\cdot\nabla_x$).
We rewrite \eqref{linear-b-def-p-delta-0}
as follows:
\[
P\sb{\delta,\epsilon}
=
-\frac{1}{2\pi\jj}
\oint\sb{\abs{\eta}=\delta}
(\eubJ\mathscr{L}(\omega)-\eta)^{-1}
\,d\eta,
\qquad
\omega=\sqrt{m^2-\epsilon^2}.
\]

\begin{lemma}\label{linear-b-lemma-same-rank}
Let
\[
p\sb{\delta}
=
-\frac{1}{2\pi\jj}
\oint\sb{\abs{\eta}=\delta}
(\eubJ\eubK-\eta)^{-1}\pi\sb{P}\,d\eta
\]
be the Riesz projector onto the generalized null space
of $\eubJ\eubK\at{\range(\pi\sb{P})}$.
Then:
\begin{enumerate}
\item
\label{linear-b-lemma-same-rank-i}
\[
\Norm{
\begin{bmatrix}
\pi\sb{P}P\sb{\delta,\epsilon}\pi\sb{P}
&
\pi\sb{P}P\sb{\delta,\epsilon}\pi\sb{A}
\\
\pi\sb{A}P\sb{\delta,\epsilon}\pi\sb{P}
&
\pi\sb{A}P\sb{\delta,\epsilon}\pi\sb{A}
\end{bmatrix}
-
\begin{bmatrix}p\sb{\delta}&0\\0&0\end{bmatrix}
}\sb{L^2(\R^n,\C^{4N})\to L^2(\R^n,\C^{4N})}\to 0
\quad
\mbox{
as
\ $\epsilon\to 0$};
\]
\item
\label{linear-b-lemma-same-rank-ii}
There is $\upepsilon_2>0$ such that,
for any $\epsilon\in(0,\upepsilon_2)$,
one has
$\rank P\sb{\delta,\epsilon}=\rank p\sb{\delta}$.
\end{enumerate}
\end{lemma}

\begin{proof}
By Lemma~\ref{linear-b-lemma-limit-system},
$\sigma(\eubJ\eubK)\subset\sigma
\Big(
\begin{bmatrix}0&\eurl\sb{-}\\-\eurl\sb{+}&0\end{bmatrix}
\Big)
\cup\sigma(\jj\eurl\sb{-})\cup\sigma(-\jj\eurl\sb{-})$,
hence $(\eubJ\eubK-\eta)\at{\range(\pi\sb{P})}$ has a bounded inverse
\[
(\eubJ\eubK-\eta)^{-1}:\;
H^{-1}(\R^n,\range(\pi\sb{P}))\to H^1(\R^n,\range(\pi\sb{P}))
\]
on the circle $\abs{\eta}=\delta$, $\eta\in\C$,
with $\delta>0$ sufficiently small
(cf. Lemma~\ref{linear-b-lemma-lm-inverse-h1}).

On the direct sum
$(\range(\pi\sb{P}))\oplus(\range(\pi\sb{A}))$,
the operator $\eubJ\mathscr{L}(\omega)-\eta$
is represented by the matrix
\[
\begin{bmatrix}
A_{11}(\epsilon,\eta)&A_{12}(\epsilon)
\\
A_{21}(\epsilon)&A_{22}(\epsilon,\eta)
\end{bmatrix}
:=
\begin{bmatrix}
\pi\sb{P}\eubJ\mathscr{L}\pi\sb{P}-\eta
&\pi\sb{P}\eubJ\mathscr{L}\pi\sb{A}
\\
\pi\sb{A}\eubJ\mathscr{L}\pi\sb{P}
&\pi\sb{A}\eubJ\mathscr{L}\pi\sb{A}-\eta
\end{bmatrix}.
\]
According to
\eqref{linear-b-def-l-script},
\begin{eqnarray}\label{linear-b-b-c-d-small}
&
\norm{A_{12}(\epsilon)}\sb{H^1\to L^2}
+\norm{A_{12}(\epsilon)}\sb{L^2\to H^{-1}}
+\norm{A_{21}(\epsilon)}\sb{H^1\to L^2}=O(\epsilon^{-1}),
\nonumber
\\
&
A_{22}(\epsilon,\eta)^{-1}\at{\range(\pi\sb{A})}
=
-\frac{\epsilon^2}{2m}\eubJ^{-1}+O\sb{L^2\to L^2}(\epsilon^4).
\end{eqnarray}
In the last equality, we used the following relation (cf. \eqref{linear-b-def-l-script}):
\[
A_{22}(\epsilon,\eta)
=
\pi\sb{A}\eubJ\mathscr{L}\pi\sb{A}-\eta
=
-\epsilon^{-2}(m+\omega)\eubJ
+
\pi\sb{A}\eubJ\eubV(y,\omega)\pi\sb{A}-\eta.
\]
The Schur complement of $A_{22}(\epsilon,\eta)$
(see \eqref{linear-b-Schur-22})
is given by
\begin{eqnarray}\label{linear-b-def-Schur-0}
&&T(\epsilon,\eta)
=
A_{11}(\epsilon,\eta)-A_{12}(\epsilon)A_{22}(\epsilon,\eta)^{-1}A_{21}(\epsilon)
\\
\nonumber
&&=
\pi\sb{P}
\Big(\frac{\eubJ}{m+\omega}+\eubJ\eubV-\eta\Big)
\pi\sb{P}
-
\pi\sb{P}(\epsilon^{-1}\eubJ\eubD_0+\eubJ\eubV)
\pi\sb{A}
A_{22}(\epsilon,\eta)^{-1}
\pi\sb{A}
(\epsilon^{-1}\eubJ\eubD_0+\eubJ\eubV)
\pi\sb{P},
\end{eqnarray}
which we consider as an operator
$
T(\epsilon,\eta):\;
H^1(\R^n,\C^{2N})\to H^{-1}(\R^n,\C^{2N})
$.
With the expression \eqref{linear-b-b-c-d-small}
for $A_{22}(\epsilon,\eta)^{-1}$,
the Schur complement \eqref{linear-b-def-Schur-0} takes the form
\begin{equation}\label{linear-b-def-Schur}
T(\epsilon,\eta)
=
\pi\sb{P}
\Big(
\frac{\eubJ}{m+\omega}+\eubJ\eubV-\eta
-\frac{\eubJ\Delta_y}{2m}
+O\sb{H^1\to H^{-1}}(\epsilon^2)
\Big)\pi\sb{P}.
\end{equation}
Using the expression \eqref{linear-b-def-Schur},
we can write
the inverse of $\eubJ\mathscr{L}(\omega)-\eta$,
considered as a map
\[
(\eubJ\mathscr{L}(\omega)-\eta)^{-1}:\;
L^2(\R^n,\range(\pi\sb{P})\oplus\range(\pi\sb{A}))
\to
L^2(\R^n,\range(\pi\sb{P})\oplus\range(\pi\sb{A}))
,
\]
as follows (see \eqref{linear-b-Schur-22}):
\begin{equation}\label{linear-b-m-z-inverse}
(\eubJ\mathscr{L}-\eta)^{-1}
=
\begin{bmatrix}
T^{-1}&T^{-1}A_{12} A_{22}^{-1}
\\
-A_{22}^{-1}A_{21}T^{-1}
&A_{22}^{-1}+A_{22}^{-1}A_{21}
T^{-1}A_{12}A_{22}^{-1}
\end{bmatrix}
.
\end{equation}
Above, $T=T(\epsilon,\eta)$,
$A_{12}=A_{12}(\epsilon)$,
$A_{21}=A_{21}(\epsilon)$,
$A_{22}=A_{22}(\epsilon,\eta)$.
Since
\begin{equation}\label{linear-b-s-m-z}
\norm{
\big(T(\epsilon,\eta)-(\eubJ\eubK-\eta)\big)\at{\range(\pi\sb{P})}}\sb{H^1\to H^{-1}}=O(\epsilon),
\end{equation}
uniformly in $\abs{\eta}=\delta$,
while
$\eubJ\eubK-\eta:\, H^{1}(\R^n,\C^{2N})\to H^{-1}(\R^n,\C^{2N})$
has a bounded inverse
for $\abs{\eta}=\delta$,
the operator $T(\epsilon,\eta)\at{\range(\pi\sb{P})}$ is also invertible
for $\abs{\eta}=\delta$
as long as $\epsilon>0$ is sufficiently small,
with its inverse being a bounded map from
$H^{-1}(\R^n,\range(\pi\sb{P}))$ to $H^1(\R^n,\range(\pi\sb{P}))$.
Using \eqref{linear-b-b-c-d-small}, we conclude that
the matrix \eqref{linear-b-m-z-inverse}
has all its entries, except the top left one,
of order $O(\epsilon)$
(when considered in the $L^2\to L^2$ operator norm).
Hence, it follows from \eqref{linear-b-m-z-inverse}
and \eqref{linear-b-s-m-z}
that,
considering $P\sb{\delta,\epsilon}$
as an operator on $\range(\pi\sb{P})\oplus\range(\pi\sb{A})$,
\[
\Norm{
P\sb{\delta,\epsilon}
-
\begin{bmatrix}p\sb\delta&0\\0&0\end{bmatrix}
}
=
\Norm{
\frac{1}{2\pi\jj}\oint\sb{\abs{\eta}=\delta}
\begin{bmatrix}T(\epsilon,\eta)^{-1}-(\eubJ\eubK-\eta)^{-1}&0\\0&0\end{bmatrix}
\,d\eta
}
+O(\epsilon)
=O(\epsilon),
\]
where the norms refer to $L^2\to L^2$ operator norm.
This proves
Lemma~\ref{linear-b-lemma-same-rank}~\itref{linear-b-lemma-same-rank-i}.
The statement~\itref{linear-b-lemma-same-rank-ii} follows
since both $P\sb{\delta,\epsilon}$ and $p\sb{\delta}$
are projectors.
\end{proof}

The statement
of Proposition~\ref{linear-b-prop-jl-limit}
on the rank of $P\sb{\delta,\epsilon}$
follows from Lemma~\ref{linear-b-lemma-same-rank}
and Lemma~\ref{linear-b-lemma-limit-system}~\itref{linear-b-lemma-limit-system-2}.
The dimension of the kernel of $\eubJ\eubL(\omega)$
follows from
considering the rank of the projection
onto the neighborhood of the eigenvalue $\lambda=0$
of the self-adjoint operator $\mathscr{L}$:
\[
\hat P\sb{\delta,\epsilon}
=-\frac{1}{2\pi\jj}\oint\sb{\abs{\eta}=\delta}
(\mathscr{L}(\omega)-\eta)^{-1}\,d\eta,
\qquad
\omega=\sqrt{m^2-\epsilon^2},
\]
similarly to how it was done for $P\sb{\delta,\epsilon}$,
and from the relation
\[
\ker(\eubJ\mathscr{L}(\omega))
=
\ker(\mathscr{L}(\omega))
=\range(\hat P\sb{\delta,\epsilon}),
\qquad
\epsilon\in(0,\upepsilon_2).
\]
Above, $\delta>0$ is small enough so that
\[
\overline{\mathbb{D}}\sb\delta
\cap\sigma\Big(
\begin{bmatrix}\eurl\sb{+}&0\\0&\eurl\sb{-}\end{bmatrix}
\Big)=\{0\}.
\]
This finishes the proof of Proposition~\ref{linear-b-prop-jl-limit}.
\end{proof}

Now we return to the proof of
Theorem~\ref{linear-b-theorem-scaled-lambda}~\itref{linear-b-theorem-scaled-lambda-three}.
If there is an eigenvalue family
$(\lambda\sb j)\sb{j\in\N}$,
$\lambda\sb j\in\sigma\sb{\mathrm{p}}(\eubJ\eubL(\omega_j))$,
such that
$\lambda_j\ne 0$ for all $j\in\N$,
$\Lambda_j=\frac{\lambda\sb j}{m\sp 2-\omega_j\sp 2}\to 0$
as $\omega_j\to m$,
then the dimension of the generalized kernel of the
nonrelativistic limit of the rescaled system
jumps up,
so that
$
\dim\frakL\big(\eubJ\eubL(\omega)\big)\at{\omega<m}+1
\ge 2n+N+1,
$
or, taking into account the symmetry of $\sigma(\eubJ\eubL(\omega))$
with respect to reflections relative to the axes $\R$ and $\jj \R$,
we see that
there is at least one more eigenvalue family,
hence the dimension
of the generalized kernel of the nonrelativistic limit
jumps up by at least two:
\[
\dim\frakL\big(\eubJ\eubL(\omega)\big)\at{\omega<m}+2
\ge 2n+N+2.
\]
Comparing this inequality to
Lemma~\ref{linear-b-lemma-limit-system}~\itref{linear-b-lemma-limit-system-2}
shows that
the assumption $\Lambda_j\ne 0$ for $j\in\N$,
$\Lambda_j\to 0$
leads to
$
\dim\frakL\Big(
\begin{bmatrix}0&\eurl\sb{-}\\-\eurl\sb{+}&0\end{bmatrix}
\Big)
\ge 2n+4.
$
By Lemma~\ref{linear-b-lemma-dim-ker}
this is only possible in the charge-critical case $\kappa=2/n$.


Thus, we know that $\kappa=2/n$.
The remaining part of the argument
further develops the approach from \cite{MR1995870}
to show that there could be no subsequence
$\Lambda_j\to 0$
with $\Re\Lambda_j\ne 0$
in the case when
$\p\sb\omega Q(\phi\sb\omega)<0$
for $\omega\lessapprox m$,
in a formal agreement with the Kolokolov
stability condition \cite{kolokolov-1973}.
We define
\begin{eqnarray}\label{linear-b-def-e1-e2}
\bm\Phi(y,\omega)
&=&
\epsilon^{-\frac 1 \kappa}\bmupphi\sb\omega(\epsilon^{-1}y),
\nonumber
\\
\e_1(y,\omega)
&=&
\epsilon^{-\frac 1 \kappa}\eubJ\bmupphi\sb\omega(\epsilon^{-1}y),
\\
\nonumber
\e_2(y,\omega)
&=&
\epsilon^{2-\frac 1 \kappa}(\p\sb\omega\bmupphi\sb\omega)(\epsilon^{-1}y);
\end{eqnarray}
here and below, $\epsilon=\sqrt{m^2-\omega^2}$.
Noting the factor $\epsilon^{-2}$
in the definition of $\mathscr{L}$ in \eqref{linear-b-def-l-script},
we deduce from
\eqref{linear-b-l-phi-phi}
the relations
\begin{equation}\label{linear-b-e1-e2}
\eubJ\mathscr{L}(\omega)
\e_1(\omega)=0,
\qquad
\eubJ\mathscr{L}(\omega)
\e_2(\omega)=\e_1(\omega),
\qquad
\omega\in(\upomega_1,m).
\end{equation}
With
\begin{eqnarray}\label{linear-b-def-theta-0}
\theta(y)=-\frac{m}{\kappa}u_\kappa(y)-m\,y\!\cdot\!\nabla u_\kappa(y)
\qquad
\theta\in H^1(\R^n)
\end{eqnarray}
and real-valued $\alpha,\,\beta\in H^2(\R^n)$
such that
\begin{eqnarray}\label{linear-b-theta-alpha-beta}
\eurl\sb{+}\theta(y)=u_\kappa(y),
\qquad
\eurl\sb{-}\alpha(y)=\theta(y),
\qquad
\eurl\sb{+}\beta(y)=\alpha(y)
\end{eqnarray}
(see
\eqref{linear-b-def-theta}, \eqref{linear-b-def-alpha}, and \eqref{linear-b-def-beta}
in the proof of Lemma~\ref{linear-b-lemma-dim-ker}),
we define
\[
E_3(y)=-\eubJ\bmXi
\alpha(y),
\qquad
E_4(y)=
-\bmXi
\beta(y),
\]
with $\bmXi\in\R^{2N}$ from
\eqref{linear-b-def-Xi},
so that $E_3,\,E_4\in H^2(\R^n,\R^{2N})$
satisfy
\begin{equation}\label{linear-b-E3-E4}
\eubJ\eubK E_3(y)
=
\bmXi
\theta(y),
\qquad
\eubJ\eubK E_4(y)=E_3(y).
\end{equation}

\begin{lemma}\label{linear-b-lemma-e1-e2}
Let $\upomega_2=\sqrt{m^2-\upepsilon_2^2}$,
with $\upepsilon_2$ from Proposition~\ref{linear-b-prop-jl-limit}.
The functions
\[
\e_a(\omega),
\ \ 1\le a\le 2,
\qquad
\e_4(\omega)=P\sb{\delta,\epsilon}E_4,
\qquad
\e_3(\omega)=\eubJ\eubL(\omega)\e_4(\omega),
\qquad
\omega\in(\upomega_2,m),
\]
can be extended to continuous maps
$\e_a:\,(\upomega_2,m]\to L^2(\R^n,\R^{2N})$,
$1\le a\le 4$,
with
$\e_1(m)=
\eubJ\bmXi
u_\kappa$,
$\e_2(m)=
\bmXi
\theta$,
and $\e_a(m)=\lim\sb{\omega\to m}\e_a(\omega)=E_a$,
$3\le a\le 4$,
so that
\begin{eqnarray}\label{linear-b-jordan}
\eubJ\eubK\e_1(\omega)=0,
\quad
\eubJ\eubK\e_2(\omega)=\e_1(\omega),
\quad
\omega\in(\upomega_2,m];
\nonumber
\\[2ex]
\eubJ\eubK\e_3(m)=\e_2(m),
\quad
\eubJ\eubK\e_4(m)=\e_3(m).
\end{eqnarray}
\end{lemma}

\begin{proof}
By Theorem~\ref{dirac-existence-theorem-solitary-waves-c1},
\[
\e_1(y,\omega)
=
\epsilon^{-\frac 1 \kappa}\eubJ\bmupphi\sb\omega(\epsilon^{-1} y)
=\eubJ\bmXi
u_\kappa(y)
+O\sb{H^1(\R^n,\C^{2N})}(\epsilon^{2\varkappa}),
\]
so
$
\lim\limits\sb{\omega\to m}
\e_1(y,\omega)
$
is defined in $H^1(\R^n,\C^{2N})$.
Since
\[
v(r,\omega)=\epsilon^{1/\kappa}
(\hat V(\epsilon r)+\tilde V(\epsilon r,\epsilon))
\quad
\mbox{and}
\quad
u(r,\omega)=\epsilon^{1+1/\kappa}
(\hat U(\epsilon r)+\tilde U(\epsilon r,\epsilon)),
\]
with $\hat V$, $\hat U$ from \eqref{dirac-existence-Vhatdef},
one has
$
\p\sb\omega v(x,\omega)
=
\frac{\p\epsilon}{\p\omega}
\p\sb\epsilon
\big(\epsilon^{\frac{1}{\kappa}}\hat V(\epsilon x)
+\epsilon^{\frac{1}{\kappa}}\tilde V(\epsilon x,\epsilon)\big)
,
$
so that
\begin{eqnarray}\label{linear-b-st}
&&
\epsilon^{2-\frac 1 \kappa}
\p\sb\omega v(\epsilon^{-1}y,\omega)
\\
\nonumber
&&
=
-\omega
\Big(
\frac{\hat V(y)}{\kappa}
+y\cdot\nabla\hat V(y)
+\frac{\tilde V(y,\epsilon)}{\kappa}
+y\cdot\nabla\tilde V(y,\epsilon)
+\epsilon\p\sb\epsilon\tilde V(y,\epsilon)
\Big).
\end{eqnarray}
Using \eqref{dirac-existence-v-u-tilde-small-better}
from Theorem~\ref{dirac-existence-theorem-solitary-waves-c1}
to bound the $y\cdot\nabla\tilde V$-term,
one has
\[
\norm{\abs{y}\nabla\sb y\tilde V(\abs{y},\epsilon)}\sb{L^2(\R^n)}
=O(\epsilon^{2\varkappa});
\]
due to \eqref{linear-b-norm-p-epsilon-w}
from Theorem~\ref{dirac-existence-theorem-solitary-waves-c1},
$
\norm{\p\sb\epsilon \tilde V(\cdot,\epsilon)}
\sb{H^1(\R^n,\R^2)}
=O(\epsilon^{2\varkappa-1}).
$
Taking into account these estimates
in \eqref{linear-b-st}, we arrive at
\[
\epsilon^{2-\frac 1 \kappa}
(\p\sb\omega v)(\epsilon^{-1}y,\epsilon)
=
-\omega
\Big(
\frac{1}{\kappa}\hat V(y)
+y\!\cdot\!\nabla\hat V(y)
\Big)
+O\sb{L^2(\R^n)}(\epsilon\sp{2\varkappa}),
\]
with a similar expression for $\epsilon^{2-\frac 1 \kappa}\p\sb\omega u$.
This leads to
\begin{eqnarray}\label{linear-b-almost-theta}
\epsilon^{2-\frac 1 \kappa}
(\p\sb\omega\bmupphi\sb\omega)(\epsilon^{-1}y)
=
-\omega
\Big(
\frac{1}{\kappa}\hat V(y)
+y\!\cdot\!\nabla\hat V(y)
\Big)
\bmXi
+O\sb{L^2(\R^n)}(\epsilon\sp{2\varkappa}).
\end{eqnarray}
Taking into account that
$
\e_2(y,\omega)
=
\epsilon^{2-\frac 1 \kappa}(\p\sb\omega\bmupphi\sb\omega)(\epsilon^{-1}y)
$
(cf. \eqref{linear-b-def-e1-e2}),
the relation \eqref{linear-b-almost-theta} allows us to define
\begin{eqnarray}\label{linear-b-def-e2-m}
\e_2(m):=
\lim\sb{\omega\to m}\e_2(\omega)
=
\lim\sb{\omega\to m}
\epsilon^{2-\frac 1 \kappa}(\p\sb\omega\bmupphi\sb\omega)(\epsilon^{-1}\,\cdot\,)
=\bmXi\theta,
\end{eqnarray}
with
$
\theta
$
from \eqref{linear-b-def-theta-0}.
By \eqref{linear-b-almost-theta},
the convergence in \eqref{linear-b-def-e2-m}
is in $L^2(\R^n,\C^{2N})$.

For $a=4$, one has:
\[
\lim\sb{\omega\to m}\e_4(\omega)
=
\lim\sb{\omega\to m}P\sb{\delta,\epsilon}E\sb 4
=
E\sb 4
+
\lim\sb{\omega\to m}
(P\sb{\delta,\epsilon}-p\sb\delta)E\sb 4
=E\sb 4,
\]
with the limit holding in $L^2$ norm.
In the last relation, we used the relation
$p\sb{\delta}E_a=E_a$, $1\le a\le 4$,
and
Lemma~\ref{linear-b-lemma-same-rank}.

For $a=3$,
the result follows from
$\e_3(\omega)
=\eubJ\eubL(\omega)\e_4(\omega)
=\eubJ\eubL(\omega)P_{\delta,\omega}\e_4(\omega)$
since
$\eubJ\eubL(\omega)P_{\delta,\omega}$ is a bounded operator.

We also point out that not only
$\e_1(\omega)$ and $\e_2(\omega)$,
but also
$\e_3(\omega)$ and $\e_4(\omega)$
are real-valued;
this follows from the observation that
$E_4\in L^2(\R^n,\R^{2N})$ is real-valued, while
$P_{\delta,\omega}$ commutes with the operator
$\bmK:\,\C^{2N}\to \C^{2N}$ of complex conjugation
since
$\eubJ\eubL$ has real coefficients.
\end{proof}

In the vector space
$\range(P\sb{\delta,\epsilon})$
we may choose the basis
(cf. Lemma~\ref{linear-b-lemma-dim-ker-nld})
\ac{BELOW, $y^i$ or $x^i$????}
\begin{equation}\label{linear-b-n-g-basis}
\big\{
\e_a(\omega),
\ 1\le a\le 4;
\ \ \ \p_{y^i}\bm\Phi,
\ \ \omega y^i\eubJ\bm\Phi
-\frac{\epsilon}{2}\bmupalpha\sp i\bm\Phi,
\ \ 1\le i\le n;
\ \ \ \bm\Theta_k,
\ 1\le k\le N-2
\big\},
\end{equation}
where $\bm\Phi$ is from \eqref{linear-b-def-e1-e2}
and
$\bm\Theta_k(\omega)$ are certain vectors from
$\ker(\eubJ\mathscr{L}(\omega))$,
with $1\le k\le N-2$
due to Proposition~\ref{linear-b-prop-jl-limit}
(which states that $\rank P\sb{\delta,\epsilon}=2n+N+2$,
$\dim\ker(\eubJ\mathscr{L}(\omega)\at{P\sb{\delta,\epsilon}})=n+N-1$).

\begin{remark}
When $n=3$ and $N=4$,
there are three
vectors $\bm\Theta_k(\omega)$
corresponding to infinitesimal rotations
around three coordinate axes,
but, as it was mentioned in \cite{MR3311594},
the span of these vectors,
$\mathop{\rm span}\{\bm\Theta_k;\,1\le k\le 3\}$,
turns out to contain
the null eigenvector $\e_1(\omega)$.
\end{remark}

In the basis
\eqref{linear-b-n-g-basis}
of the space
$\range(P\sb{\delta,\epsilon})$,
the operator
$(\eubJ\mathscr{L}(\omega)-\lambda I_{N})
\at{\range(P\sb{\delta,\epsilon})}$
is represented by
\begin{equation}\label{linear-b-dv-rbm-new}
M\sb\omega-\lambda I_{N}
=
\begin{bmatrix}
-\lambda&1&\sigma_1(\omega)&0&0&0&0
\\
0&-\lambda&\sigma_2(\omega)&0&0&0&0
\\
0&0&\sigma_3(\omega)-\lambda&1&0&0&0
\\
0&0&\sigma_4(\omega)&-\lambda&0&0&0
\\
\vdots&\vdots&\vdots&\vdots&-\lambda I_{n}&I_{n}&0
\\
\vdots&\vdots&\vdots&\vdots&0&-\lambda I_{n}&0
\\
\vdots&\vdots&\vdots&\vdots&0&0&-\lambda I_{{N-2}}
\end{bmatrix},
\end{equation}
where vertical dots denote columns of irrelevant coefficients
and
\[
\sigma_a(\omega),
\qquad
1\le a\le 4,
\]
are some continuous functions of $\omega$.
We used Lemma~\ref{linear-b-lemma-e1-e2}
and the relations
\[
\eubJ\mathscr{L}\Big(
\omega y^i\eubJ\bm\Phi
-\frac{\epsilon}{2}\bmupalpha\sp{i}\bm\Phi
\Big)
=\p_{y^i}\bm\Phi,
\qquad
1\le i\le n,
\]
which follow from \eqref{linear-b-l-phi-phi},
\eqref{linear-b-def-l-script},
and \eqref{linear-b-def-e1-e2}.
Considering \eqref{linear-b-dv-rbm-new}
at $\lambda=0$ and $\epsilon=0$,
one concludes from \eqref{linear-b-jordan} that
\begin{equation}\label{linear-b-dv-sigma0100}
\sigma_1(m)=\sigma_3(m)=\sigma_4(m)=0,
\qquad
\sigma_2(m)=1.
\end{equation}
 From \eqref{linear-b-dv-rbm-new}, we also have
\begin{equation}\label{linear-b-det-m}
\det(M\sb\omega-\lambda)
=(-\lambda)^{2n+N}(\lambda^2-\lambda\sigma_3(\omega)-\sigma_4(\omega)).
\end{equation}

\begin{lemma}\label{linear-b-lemma-e-alpha-e}
For any solitary wave $\phi(x)e^{-\jj\omega t}$
with $\phi\in H^1_{1/2}(\R^n)$
and any $1\le i\le n$,
one has
$
\langle\phi,\alpha\sp i\phi\rangle=0$.
\end{lemma}

\begin{proof}
The local version of the charge conservation,
$\p\sb\mu\mathscr{J}\sp\mu=0$,
with $\mathscr{J}\sp\mu(t,x)=\bar\psi(t,x)\gamma\sp\mu\psi(t,x)$,
when applied to a solitary wave
with stationary charge and current densities,
$\mathscr{J}\sp\mu(t,x)=\bar\phi(x)\gamma\sp\mu\phi(x)$,
yields the desired identity:
for any $1\le i\le n$,
\[
0=\p\sb t\int\sb{\R^n}\mathscr{J}\sp 0(x) x^i\,dx
=-\sum\sb{j=1}^n\int\sb{\R^n}\big(\p\sb j\mathscr{J}\sp j(x)\big)x^i\,dx
=\int\sb{\R^n}\mathscr{J}\sp i(x)\,dx.
\qedhere
\]
\end{proof}

Expanding
$\eubJ\mathscr{L}\e_3(\omega)$
over the basis in $\range(P\sb{\delta,\epsilon})$
(see \eqref{linear-b-n-g-basis}),
we conclude that
for some continuous functions
$\gamma\sb i(\omega)$ and $\rho\sb i(\omega)$,
$1\le i\le n$,
and $\tau\sb k(\omega)$, $1\le k\le N-2$,
there is a relation
\ac{BELOW, $y^i$ or $x^i$????}
\begin{equation}\label{linear-b-dv-jh2-e2-e4-new}
\eubJ\mathscr{L}
\e_3
=
\sum\sb{a=1}\sp{4}\sigma_a
\e_a
+\sum\sb{i=1}\sp{n}
\Big(
\gamma\sb i
\p_{y^i}\bm\Phi
+
\rho\sb i
\Big(\omega y\sp i\eubJ\bm\Phi
-\frac{\epsilon}{2}\bm\upalpha\sp i\bm\Phi
\Big)
\Big)
+\sum\sb{k=1}\sp{N-2}\tau\sb k
\bm\Theta\sb k,
\end{equation}
valid for $\omega\in(\upomega_2,m]$.
Above, $\sigma_a(\omega)$, $1\le a\le 4$, are continuous functions from \eqref{linear-b-dv-rbm-new}.
Pairing \eqref{linear-b-dv-jh2-e2-e4-new}
with $\bm\Phi=\bm\Phi(\omega)=\eubJ^{-1}\e_1(\omega)$,
we get:
\begin{equation}\label{linear-b-sigma2-sigma4}
0
=
\sigma_2(\omega)
\left\langle
\eubJ^{-1}\e_1(\omega),\e_2(\omega)
\right\rangle
+
\sigma_4(\omega)
\left\langle
\eubJ^{-1}\e_1(\omega),\e_4(\omega)
\right\rangle,
\quad
\omega\in(\upomega_2,m].
\end{equation}
We took into account that
one has
$\langle\bm\Phi,\bm{v}\rangle
=\langle\mathscr{L}\e_2,\bm{v}\rangle
=\langle\e_2,\mathscr{L}\bm{v}\rangle=0$
for any
$
\bm{v}
\in
\ker(\eubJ\mathscr{L}),
$
the identities
\[
\left\langle
\eubJ^{-1}\e_1,\eubJ\mathscr{L}\e_3(\omega)
\right\rangle
=
-\left\langle
\mathscr{L}\e_1,\e_3(\omega)
\right\rangle=0,
\]
\[
\langle \eubJ^{-1}\e_1,\e_3\rangle
=\langle \eubJ^{-1}\e_1,\eubJ\mathscr{L}\e_4\rangle
=-\langle\mathscr{L}\e_1,\e_4\rangle
=0,
\]
and also the identity
$\langle\bm\Phi,
\omega y^i\eubJ\bm\Phi
-\frac{\epsilon}{2}\bmupalpha\sp i\bm\Phi
\rangle=0$
which holds due to Lemma~\ref{linear-b-lemma-e-alpha-e}
and due to
$
\bm\Phi\sp\ast
\eubJ\bm\Phi
\equiv 0
$
(the left-hand side is skew-adjoint while all the quantities
are real-valued).
Since
\begin{eqnarray}\label{linear-b-e1-e2-q}
\langle\eubJ^{-1}\e_1(\omega),\e_2(\omega)\rangle
&=&
\epsilon^{-\frac 2 \kappa}
\langle\bmupphi\sb\omega(\epsilon^{-1}\cdot),
(\p\sb\omega\bmupphi\sb\omega)(\epsilon^{-1}\cdot)\rangle
\nonumber
\\
&=&\epsilon^{n-\frac 2 \kappa}
\langle\bmupphi\sb\omega,\p\sb\omega\bmupphi\sb\omega\rangle
=\fra{\p\sb\omega Q(\phi\sb\omega)}{2}
\end{eqnarray}
(we took into account that $n-\frac 2 \kappa=0$),
the relation
\eqref{linear-b-sigma2-sigma4} takes the form
\begin{equation}\label{linear-b-sigma2-sigma4-better}
\sigma_2(\omega)\p\sb\omega Q(\phi\sb\omega)/2
=\mu(\omega)\sigma_4(\omega),
\qquad
\omega\in(\upomega_2, m],
\end{equation}
where
$
\mu(\omega):=-\langle \eubJ^{-1}\e_1(\omega),\e_4(\omega)\rangle
$
is a continuous function
of $\omega\in(\upomega_2,m]$.

\begin{remark}
By \eqref{linear-b-e1-e2-q} and Lemma~\ref{linear-b-lemma-e1-e2},
$\p\sb\omega Q(\phi\sb\omega)$
is a continuous function of $\omega\in(\upomega_2,m]$.
\end{remark}

\begin{lemma}\label{linear-b-lemma-mu-positive}
There is $\upomega_3\in(\upomega_2,m)$ such that
$
\mu(\omega)>0
$
for $\upomega_3<\omega\le m$.
\end{lemma}

\begin{proof}
We have
\[
\mu(\omega)
=-\langle\bm\Phi,\e_4\rangle
=-\langle\bm\Phi,P\sb{\delta,\epsilon}(\e_4)\rangle
=-\langle\eubJ^{-1}\e_1(m),\e_4(m)\rangle
+O(\epsilon),
\]
while \eqref{linear-b-jordan} yields
\[
-\langle
\eubJ^{-1}\e_1,\e_4\rangle
\at{\omega=m}
=-\langle\eubK \e_2,\e_4\rangle
\at{\omega=m}
=-\langle\eubJ\eubK \e_3,\eubJ^{-1}\e_3\rangle
\at{\omega=m}
=
\big\langle
\bmXi\alpha,
\eubK
\bmXi\alpha
\big\rangle
>0.
\]
Above, we used \eqref{linear-b-E3-E4}
and the explicit form of $E_2$ and $E_3$.
\end{proof}

\begin{lemma}
There is $\upomega_4\in[\upomega_3,m)$ such that
$\sigma_3(\omega)\equiv 0$
for $\omega\in[\upomega_4,m]$.
\end{lemma}

\begin{proof}
Applying $(\eubJ\mathscr{L}(\omega))^2$ to \eqref{linear-b-dv-jh2-e2-e4-new},
we get
\[
(\eubJ\mathscr{L})^3
\e_3(\omega)
=
\sigma_3(\omega)(\eubJ\mathscr{L})^2\e_3(\omega)
+
\sigma_4(\omega)(\eubJ\mathscr{L})^2\e_4(\omega).
\]
Coupling this relation with $\eubJ^{-1}\e_4$
and using the identities
\[
\langle
\eubJ^{-1}\e_4,(\eubJ\mathscr{L})^3\e_3\rangle
=
\langle\e_3,\mathscr{L}\eubJ\mathscr{L}\e_3\rangle=0
\]
and
\[
\langle
\eubJ^{-1}\e_4,(\eubJ\mathscr{L})^2\e_4\rangle
=-\langle\e_4,\mathscr{L}\eubJ\mathscr{L}\e_4\rangle=0
\]
(both of these due to skew-adjointness of $\mathscr{L}\eubJ\mathscr{L}$,
taking into account that
$\e_a(\omega)$, $1\le a\le 4$,
are real-valued by Lemma~\ref{linear-b-lemma-e1-e2},
while $\eubJ$ and $\mathscr{L}$ have real coefficients),
we have
\begin{equation}\label{linear-b-dv-s3asdf}
\sigma_3(\omega)
\langle \eubJ^{-1}\e_4,
(\eubJ\mathscr{L})^2\e_3\rangle=0.
\end{equation}
The factor at $\sigma_3(\omega)$
is nonzero
for $\omega<m$ sufficiently close to $m$.
Indeed, using \eqref{linear-b-dv-sigma0100},
\[
\langle
\eubJ^{-1}\e_4,(\eubJ\mathscr{L})^2\e_3\rangle
\at{\omega=m}
=
\langle
\eubJ^{-1}\e_4,\,
\sigma_2\e_1
+
\sigma_3 \eubJ\mathscr{L}\e_3
+
\sigma_4\e_3
\rangle
\at{\omega=m}
\]
\[
=
\langle
\eubJ^{-1}\e_4,
\e_1
\rangle
\at{\omega=m}
=
-\langle
\e_4,\bmupphi
\rangle
\at{\omega=m},
\]
which is positive due to Lemma~\ref{linear-b-lemma-mu-positive}.
Due to continuity in $\omega$ of the coefficient at $\sigma_3(\omega)$
in \eqref{linear-b-dv-s3asdf}, we conclude that
$\sigma_3(\omega)$ is identically zero
for $\omega\in[\upomega_4,m]$,
with some $\upomega_4<m$.
\end{proof}

Since $\sigma_3(\omega)$ is identically zero
for $\omega\in[\upomega_4,m]$,
we conclude from \eqref{linear-b-det-m}
that the nonzero eigenvalues of $\eubJ\mathscr{L}(\omega)$
satisfy
$
\lambda^2-\sigma_4(\omega)=0,
$
$
\omega\in[\upomega_4,m].
$
By \eqref{linear-b-dv-sigma0100} and Lemma~\ref{linear-b-lemma-mu-positive},
the relation \eqref{linear-b-sigma2-sigma4-better}
shows that
$\sigma_4(\omega)$ is of the same sign as $\p\sb\omega Q(\phi\sb\omega)$.
Thus, if $\p\sb\omega Q(\phi\sb\omega)>0$
for $\omega\lessapprox m$,
then for these values of $\omega$
there are two nonzero real eigenvalues of $\eubJ\mathscr{L}(\omega)$,
one positive
(indicating the linear instability)
and one negative, both of magnitude
$\sim\sqrt{\p\sb\omega Q(\phi\sb\omega)}$
for $\omega\lessapprox m$;
hence, there are two real eigenvalues of $\eubJ\eubL$,
of magnitude $\sim\epsilon^2\sqrt{\p\sb\omega Q(\phi\sb\omega)}$.
This completes the proof of
Theorem~\ref{linear-b-theorem-scaled-lambda}.

\section{Bifurcations from embedded thresholds}
\label{linear-b-sect-vb-thresholds-1}

\subsection{Absence of bifurcations from the essential spectrum}
Now we proceed to the proof of
Theorem~\ref{linear-b-theorem-2m}~\itref{linear-b-theorem-2m-1}
and ~\itref{linear-b-theorem-2m-2}:
we need to prove that the sequence
\eqref{linear-b-def-Lambda-j-0},
\begin{equation}\label{linear-b-def-Lambda-j}
Z_j=-\frac{2\omega_j+\jj\lambda_j}{\epsilon_j^2}
\in\C,
\qquad
j\in\N,
\end{equation}
can only accumulate to either the
discrete spectrum
or the threshold
of the operator $\eurl\sb{-}$
from \eqref{linear-b-def-l-small-pm}.
Moreover, we will see that the accumulation to the threshold
is not possible when it is a regular point of the essential spectrum of $\eurl\sb{-}$
(neither an $L^2$-eigenvalue nor a virtual level).


\begin{lemma}\label{linear-b-lemma-psi-a-m}
There is $C>0$ such that
\[
\norm{u_\kappa^\kappa\pi\sp{+}\bmPsi\sb{j}}\sb{L^2}
\le
C
\epsilon_j
\norm{u_\kappa^\kappa\bmPsi\sb j}\sb{L^2},
\qquad
\forall j\in\N.
\]
\end{lemma}

\begin{proof}
The relations
\eqref{linear-b-4p-plus} and \eqref{linear-b-4a-plus}
yield
\[
\begin{bmatrix}
m-\omega_j+\jj\lambda\sb j
&
\epsilon_j\eubD_0
\\
\epsilon_j\eubD_0
&
-(m+\omega_j-\jj\lambda\sb j)
\end{bmatrix}
\begin{bmatrix}
\pi\sp{+}\sb{P}\bmPsi\sb j
\\
\pi\sp{+}\sb{A}\bmPsi\sb j
\end{bmatrix}
=
-
\epsilon_j^2
\begin{bmatrix}
\pi\sp{+}\sb{P}\eubV\bmPsi\sb j
\\
\pi\sp{+}\sb{A}\eubV\bmPsi\sb j
\end{bmatrix},
\]
hence
\begin{equation}\label{linear-b-delta-tilde-z}
\begin{bmatrix}
\pi\sp{+}\sb{P}\bmPsi\sb j
\\
\pi\sp{+}\sb{A}\bmPsi\sb j
\end{bmatrix}
=
\begin{bmatrix}
m+\omega_j-\jj\lambda\sb j
&
\epsilon_j\eubD_0
\\
\epsilon_j\eubD_0
&
-(m-\omega_j+\jj\lambda\sb j)
\end{bmatrix}
\big(
\Delta_y
+\zeta_j^2
\big)^{-1}
\begin{bmatrix}
\pi\sp{+}\sb{P}\eubV\bmPsi\sb j
\\
\pi\sp{+}\sb{A}\eubV\bmPsi\sb j
\end{bmatrix},
\end{equation}
with
\[
\zeta_j^2
:=
\big((\omega_j-\jj\lambda\sb j)^2-m^2\big)/{\epsilon_j^2}
=(8m^2+o(1))/\epsilon_j^{2},
\qquad
\Re\zeta_j\ge 0;
\qquad
j\in\N.
\]
(Note that since
$\omega_j\to m$,
$\Re\lambda\sb j\ne 0$,
and
$\lambda\sb j\to 2 m \jj$,
one has $\Im\zeta_j^2\ne 0$ for all but finitely many
$j\in\N$ which we discard.)
The limiting absorption principle
from Lemma~\ref{linear-b-lemma-lap-agmon}
with $\nu=0,\,1$ and with $z=\zeta_j^2$
shows that there is $C>0$ such that
\begin{eqnarray}\label{linear-b-lap-large-1}
\norm{u_\kappa^\kappa\circ (\Delta_y+\zeta_j^2)^{-1}\circ u_\kappa^\kappa}
&\le&
C\abs{\zeta_j}^{-1},
\qquad\forall j\in\N,
\\[1ex]
\label{linear-b-lap-large-2}
\norm{u_\kappa^\kappa\circ \epsilon_j\eubD_0(\Delta_y+\zeta_j^2)^{-1}\circ u_\kappa^\kappa}
&\le&
C\epsilon_j,
\qquad\forall j\in\N,
\end{eqnarray}
where
$u_\kappa^\kappa\circ\ldots\circ u_\kappa^\kappa$
denotes the compositions with
the operators of multiplication by $u_\kappa^\kappa$.
Applying
\eqref{linear-b-lap-large-1}
and
\eqref{linear-b-lap-large-2}
to
\eqref{linear-b-delta-tilde-z}
leads to
\begin{eqnarray}
&&
\hskip -30pt
\norm{u_\kappa^\kappa\pi\sp{+}\bmPsi\sb j}
\le
C
\big(
\norm{u_\kappa^\kappa\epsilon\eubD_0(\Delta_y+\zeta_j^2)^{-1}
\pi\sp{+}\eubV\bmPsi\sb j}
+
\norm{u_\kappa^\kappa(\Delta_y+\zeta_j^2)^{-1}
\pi\sp{+}\eubV\bmPsi\sb j}
\big)
\nonumber
\\
\nonumber
&&
\hskip 20pt
\le
C
\big(
\norm{u_\kappa^\kappa\circ\epsilon\eubD_0(\Delta_y+\zeta_j^2)^{-1}\circ u_\kappa^\kappa}
+
\norm{u_\kappa^\kappa\circ(\Delta_y+\zeta_j^2)^{-1}\circ u_\kappa^\kappa}
\big)
\norm{u_\kappa^\kappa\bmPsi\sb j}
\\
\nonumber
&&
\hskip 20pt
\le
C\epsilon_j\norm{u_\kappa^\kappa\bmPsi\sb j},
\qquad\forall j\in\N.
\end{eqnarray}
We used the bound
$\norm{\eubV(y,\epsilon)}\sb{\End(\C^{2N})}\le C \abs{u_\kappa(y)}^{2\kappa}$
from Lemma~\ref{linear-b-lemma-w}.
\end{proof}

\begin{lemma}\label{linear-b-lemma-epsilon-square}
The values $Z_j$, $j\in\N$ from \eqref{linear-b-def-Lambda-j} are uniformly bounded:
There is $C>0$ such that
\[
\abs{Z_j}\le C, \qquad \forall j\in\N.
\]
\end{lemma}

In other words, the sequence $\big(Z_j\big)_{j\in\N}$
has no accumulation point at infinity.

\begin{proof}
Let us consider the values of $j\in\N$ for which
the following inequality is satisfied:
\begin{eqnarray}\label{linear-b-ge-kappa-0}
\abs{(\omega_j+\jj\lambda\sb j)^2-m^2}
<\epsilon_j^2.
\end{eqnarray}
This implies that
$
(\omega_j+\jj\lambda_j)^2=m^2+O(\epsilon_j^2),
$
hence
\[
-\omega_j-\jj\lambda_j=m+O(\epsilon_j^2)
\]
(since $\lambda_j\to -2m\jj$ as $\omega_j\to m$);
taking into account the relation $\epsilon_j^2 Z_j=-2\omega_j-\jj\lambda_j$
(see \eqref{linear-b-def-Lambda-j}),
we arrive at
\[
\epsilon_j^2 Z_j=m-\omega_j+O(\epsilon_j^2)
=O(\epsilon_j^2),
\]
which shows that
$\abs{Z_j}$ are uniformly bounded
for the values of $j$ for which \eqref{linear-b-ge-kappa-0} takes place.
Now we discard these values of  $j\in\N$;
assuming that there are infinitely many left (or else there is nothing to prove),
we have:
\begin{eqnarray}\label{linear-b-ge-kappa}
\abs{(\omega_j+\jj\lambda\sb j)^2-m^2}
\ge
\epsilon_j^2,
\qquad \forall j\in\N.
\end{eqnarray}
We write
\eqref{linear-b-4p-minus}, \eqref{linear-b-4a-minus}
as the following system:
\begin{equation}\label{linear-b-d-s-p-a}
\begin{bmatrix}
m-\omega_j-\jj\lambda\sb j
&
\epsilon_j\eubD_0
\\
\epsilon_j\eubD_0
&
-(m+\omega_j+\jj\lambda\sb j)
\end{bmatrix}
\begin{bmatrix}
\pi\sp{-}\sb{P}\bmPsi\sb j
\\
\pi\sp{-}\sb{A}\bmPsi\sb j
\end{bmatrix}
=
-
\epsilon_j^2
\begin{bmatrix}
\pi\sp{-}\sb{P}\eubV\bmPsi\sb j
\\
\pi\sp{-}\sb{A}\eubV\bmPsi\sb j
\end{bmatrix},
\qquad j\in\N,
\end{equation}
which can then be rewritten as follows:
\begin{eqnarray}\label{linear-b-psi-minus-v-psi}
\hskip -10pt
\begin{bmatrix}
\pi\sp{-}\sb{P}\bmPsi\sb j
\\
\pi\sp{-}\sb{A}\bmPsi\sb j
\end{bmatrix}
=
\begin{bmatrix}
m+\omega_j+\jj\lambda\sb j
&
\epsilon_j\eubD_0
\\
\epsilon_j\eubD_0
&
-(m-\omega_j-\jj\lambda\sb j)
\end{bmatrix}
\Big(
\Delta_y
+\upnu\sb j
\Big)^{-1}
\begin{bmatrix}
\pi\sp{-}\sb{P}\eubV\bmPsi\sb j
\\
\pi\sp{-}\sb{A}\eubV\bmPsi\sb j
\end{bmatrix},
\quad
\end{eqnarray}
with
\begin{eqnarray}\label{linear-b-def-mu-tilde}
\upnu\sb j
:=
\frac{(\omega_j+\jj\lambda\sb j)^2-m^2}
{\epsilon_j^2},
\qquad
j\in\N.
\end{eqnarray}
We notice that
$\abs{\upnu\sb j}\ge 1
$
by \eqref{linear-b-ge-kappa}
and that $\Im\upnu\sb j\ne 0$
except perhaps for finitely many values of $j$,
which we discard.
Applying
\eqref{linear-b-lap-large-1}
and
\eqref{linear-b-lap-large-2}
to \eqref{linear-b-psi-minus-v-psi},
we derive:
\begin{eqnarray}
\label{linear-b-lemma-psi-p-0}
\norm{u_\kappa^\kappa\pi\sp{-}\bmPsi\sb j}
\le
C\big(
\epsilon_j+\abs{\upnu\sb j}^{-{1}/{2}}
\big)
\norm{u_\kappa^\kappa\bmPsi\sb j},
\qquad
\forall j\in\N.
\end{eqnarray}
By Lemma~\ref{linear-b-lemma-psi-a-m}
and \eqref{linear-b-lemma-psi-p-0},
there is $C>0$ such that
\begin{eqnarray}\label{linear-b-z-j}
\hskip -16pt
\norm{u_\kappa^\kappa\bmPsi\sb j}
\le
\norm{u_\kappa^\kappa\pi\sp{+}\bmPsi\sb j}
+\norm{u_\kappa^\kappa\pi\sp{-}\bmPsi\sb j}
\le
C
\big(
\epsilon_j
+\abs{\upnu\sb j}^{-{1}/{2}}
\big)
\norm{u_\kappa^\kappa\bmPsi\sb j},
\ \ \forall j\in\N.
\quad
\end{eqnarray}
Assume that
$
\mathop{\lim\sup}\limits\sb{j\to\infty}
\abs{\upnu_j}
=+\infty.
$
Then the coefficient at $\norm{u_\kappa^\kappa\bmPsi\sb j}$
in the right-hand side of \eqref{linear-b-z-j}
would go to zero for an infinite subsequence of $j\to\infty$;
since $\bmPsi\sb j\not\equiv 0$,
we arrive at the contradiction.

Thus, $\abs{\upnu_j}$ are uniformly bounded. From \eqref{linear-b-def-mu-tilde},
we derive:
\[
\omega_j+\jj\lambda_j=-\sqrt{m^2+\epsilon_j^2\upnu_j}
=-m+O(\epsilon_j^2),
\]
where we chose the ``positive'' branch of the square root
(for all but finitely many $j\in\N$)
since
$\omega\sb j\to m$
and
$\lambda_j\to 2m\jj$
as $j\to\infty$.
This yields
\[
Z_j
=-\frac{2\omega_j+\jj\lambda_j}{\epsilon_j^2}
=-\frac{\omega_j+(\omega_j+\jj\lambda_j)}{\epsilon_j^2}
=-\frac{\sqrt{m^2-\epsilon_j^2}-\sqrt{m^2+\epsilon_j^2\upnu_j}}{\epsilon_j^2},
\qquad j\in\N;
\]
therefore, $Z_j=O(1)$
are uniformly bounded and could not accumulate at infinity.
\end{proof}

Substituting
$
-\frac{m+\omega_j+\jj\lambda\sb j}{\epsilon_j^2}
=
-\frac{2\omega_j+\jj\lambda\sb j}{\epsilon_j^2}
-\frac{m-\omega_j}{\epsilon_j^2}
=
Z_j
-\frac{1}{m+\omega_j},
$
we rewrite \eqref{linear-b-d-s-p-a} as
\begin{equation}\label{linear-b-afa}
\begin{bmatrix}
m-\omega_j-\jj\lambda\sb j&\eubD_0
\\
\eubD_0&
Z_j-\frac{1}{m+\omega_j}
+u_\kappa^{2\kappa}
\end{bmatrix}
\begin{bmatrix}
\pi\sp{-}\sb{P}\bmPsi\sb j
\\
\epsilon_j\pi\sp{-}\sb{A}\bmPsi\sb j
\end{bmatrix}
=-
\begin{bmatrix}
\epsilon_j^2
\pi\sb{P}\sp{-}
\eubV\bmPsi\sb j
\\
\epsilon_j
\pi\sb{A}\sp{-}
\eubV\bmPsi\sb j
-\epsilon_j
u_\kappa^{2\kappa}\pi\sb{A}\sp{-}\bmPsi\sb j
\end{bmatrix}.
\end{equation}
We denote the matrix-valued operator in the left-hand side by
\begin{eqnarray}\label{linear-b-def-a-j}
A_j:=
\begin{bmatrix}
m-\omega_j-\jj\lambda\sb j&\eubD_0
\\
\eubD_0&
Z_j-\frac{1}{m+\omega_j}
+u_\kappa^{2\kappa}
\end{bmatrix},
\end{eqnarray}
\[
A_j:\;
L^2(\R^n,\C^{4N})\to L^2(\R^n,\C^{4N}),
\qquad
\dom(A_j)=H^1(\R^n,\C^{4N});
\qquad
j\in\N.
\]
We note that $(A_j)_{11}\to 2m$ as $j\to\infty$,
hence, for $j$ large enough, $(A_j)_{11}$ is invertible.
By \eqref{linear-b-Schur},
the Schur complement of $(A_j)_{11}$ is given by
\begin{eqnarray}\label{def-s-j}
S_j=(A_j)_{22}-(A_j)_{21}(A_j)_{11}^{-1}(A_j)_{12}
=
\Big(
Z_j-\frac{1}{m+\omega_j}+u_\kappa^{2\kappa}
+\frac{\Delta_y}{m-\omega_j-\jj\lambda_j}
\Big)I_{2N},
\qquad
j\in\N,
\quad
\end{eqnarray}
with $\ S_j:\,L^2(\R^n,\C^{2N})\to L^2(\R^n,\C^{2N})$,
$\ \dom(S_j)=H^2(\R^n,\C^{2N})$
\ for $j\in\N$.

\medskip

Let $Z_0$ be the limit point
of the sequence $\big(Z_j\big)_{j\in\N}$;
by Lemma~\ref{linear-b-lemma-epsilon-square}, $Z_0\ne\infty$.

\begin{lemma}\label{linear-b-lemma-eigenvalue-or-resonance-away}
$Z_0\in\sigma\sb{\mathrm{d}}(\eurl\sb{-})\cup\{1/(2m)\}$.
If, moreover, the threshold $z=1/(2m)$ is a regular point
of the essential spectrum of $\eurl\sb{-}$,
then
$Z_0\in\sigma\sb{\mathrm{d}}(\eurl\sb{-})$.
\end{lemma}

\begin{proof}
We write the inverse of the matrix-valued operator
in the left-hand side of \eqref{linear-b-afa}
using \eqref{linear-b-Schur}:
\begin{eqnarray}\label{linear-b-afa-inv-away}
\hskip -12pt
\begin{bmatrix}
\pi\sp{-}\sb{P}\bmPsi\sb j
\\[1ex]
\epsilon_j\pi\sp{-}\sb{A}\bmPsi\sb j
\end{bmatrix}
\!=-
\frac{1}{(A_j)_{11}}
\begin{bmatrix}
\displaystyle
1
+(A_j)_{11}^{-1}\eubD_0 S_j^{-1}\eubD_0
&
-\eubD_0 S_j^{-1}
\\[1ex]
-S_j^{-1}\eubD_0
&
(A_j)_{11} S_j^{-1}
\end{bmatrix}
\!
\begin{bmatrix}
\epsilon_j^2
\pi\sb{P}\sp{-}
\eubV\bmPsi\sb j
\\[1ex]
\epsilon_j
\pi\sb{A}\sp{-}
\big(
\eubV
\!-\!
u_\kappa^{2\kappa}
\big)
\bmPsi\sb j
\end{bmatrix}\!.
\quad
\end{eqnarray}
Above,
the Schur complement of $(A_j)_{11}=m-\omega_j-\jj\lambda\sb j$
is given by the expression \eqref{def-s-j}:
\[
S_j
=
h_j I_{2N},
\qquad
j\in\N,
\]
with
$h_j:\,L^2(\R^n)\to L^2(\R^n)$
(with domain $\dom(h_j)=H^2(\R^n)$)
given by
\begin{eqnarray*}
&&
h_j
=Z_j-\frac{1}{m+\omega_j}
+u_\kappa^{2\kappa}
+
\frac{\Delta_y}{m-\omega_j-\jj\lambda_j}
\\
&&
=
Z_j
+\frac{1}{m-\omega_j-\jj\lambda_j}
-\frac{1}{m+\omega_j}
-\frac{1}{m-\omega_j-\jj\lambda_j}
\\
&&
\qquad
+\Big(1-\frac{2m}{m-\omega_j-\jj\lambda_j}\Big)u_\kappa^{2\kappa}
+\frac{2m u_\kappa^{2\kappa}}{m-\omega_j-\jj\lambda_j}
+
\frac{\Delta_y}{m-\omega_j-\jj\lambda_j}
\\
&&
=
-\frac{2m}{m-\omega_j-\jj\lambda_j}(\eurl\sb{-}-\hat Z_j)
+\Big(1-\frac{2m}{m-\omega_j-\jj\lambda_j}\Big)u_\kappa^{2\kappa},
\end{eqnarray*}
where the sequence
\[
\hat Z_j
=\frac{m-\omega_j-\jj\lambda_j}{2m}
\Big(
Z_j
+\frac{1}{m-\omega_j-\jj\lambda_j}
-\frac{1}{m+\omega_j}\Big),
\qquad
j\in\N,
\]
has the same limit as the sequence $\big(Z_j\big)\sb{j\in\N}$.

Let $\varOmega\subset\C$ be a compact set as in Lemma~\ref{linear-b-lemma-j-n}:
$\varOmega\cap\sigma\sb{\mathrm{d}}(\eurl\sb{-})=\emptyset$,
and if $z=1/(2m)$ is an eigenvalue or a virtual level of $\eurl\sb{-}$,
then we also assume that $1/(2m)\not\in\varOmega$.
Then, by Lemma~\ref{linear-b-lemma-j-n},
there is $C>0$ such that
\[
\norm{u_\kappa^\kappa\circ(\eurl\sb{-}-z)^{-1}\circ u_\kappa^\kappa}_{L^2\to H^2}
\le C,
\qquad
\forall z\in\varOmega\setminus[1/(2m),+\infty).
\]
Since $Z_0\in\varOmega$,
one has $\hat Z_j\in\varOmega$
and also $\hat Z_j\not\in \sigma(\eurl\sb{-})$
for $j\in\N$ large enough.
We note that
\[
\hat Z_j
=\frac{m-\omega_j-\jj\lambda_j}{2m}
\Big(
\frac{2\omega_j+\jj\lambda_j}{m^2-\omega_j^2}
+\frac{1}{m-\omega-\jj\lambda_j}
-\frac{1}{m-\omega_j}
\Big)
=\frac{1}{2m}
-\frac{(m-\omega_j-\jj\lambda_j)^2}{2m(m^2-\omega_j^2)};
\]
since $\Im\lambda_j\to 2m$ and $\Re\lambda_j>0$,
we conclude that
$\Im\hat Z_j=\frac{(m-\omega_j+\Im\lambda_j)\Re\lambda_j}{2m(m^2-\omega_j^2)}\ne 0$
(hence $\hat Z_j\not\in\sigma(\eurl\sb{-})$)
for all $j\in\N$
except at most
for finitely many values of $j$ which we discard.
Then the mapping
\[
u_\kappa^\kappa\circ S_j^{-1}\circ u_\kappa^\kappa:\;L^2(\R^n)\to H^2(\R^n),
\qquad
j\in\N,
\]
is bounded uniformly in
$j\in\N$.
By duality, so is the mapping
$
H^{-2}(\R^n)\to L^2(\R^n),
$
and, by complex interpolation, so is
\[
u_\kappa^\kappa\circ S_j^{-1}\circ u_\kappa^\kappa:\;H^{-1}(\R^n)\to H^1(\R^n).
\]
It follows that the following maps are also continuous:
\begin{eqnarray*}
u_\kappa^\kappa\circ S_j^{-1}\circ \eubD_0\circ u_\kappa^\kappa:
&
\;L^2(\R^n)\to H^1(\R^n),
\\
u_\kappa^\kappa\circ\eubD_0\circ S_j^{-1}\circ u_\kappa^\kappa:
&
\;H^{-1}(\R^n)\to L^2(\R^n),
\\
u_\kappa^\kappa\circ\eubD_0\circ S_j^{-1}\circ \eubD_0\circ u_\kappa^\kappa:
&\;L^2(\R^n)\to L^2(\R^n),
\end{eqnarray*}
with the bounds which are uniform in $j\in\N$.

We recall that there is the inclusion
$u_\kappa\in C^2(\R^n)$
(see \eqref{linear-b-def-uk}),
with both $u_\kappa$ and its derivatives
exponentially decaying with $\abs{x}$
(\cite[Lemma A.1]{MR3670258});
moreover, $\inf\sb{x\in\R^n}\abs{\p_r u(x)/u(x)}>0$
(\cite[Lemma A.2]{MR3670258}).
As a consequence, for any $\kappa>0$
(with $\kappa<2/(n-2)$ if $n\ge 3$),
the multiplication by $u_\kappa^\kappa$
extends to a continuous mapping in $L^2(\R^n)$,
in $H^1(\R^n)$,
and by duality in $H^{-1}(\R^n)$.
Similarly,
the matrix multiplication by
$[\eubD_0,u_\kappa^\kappa]=\kappa u_\kappa^{\kappa-1}\eubJ\bmupalpha^i\p_{x^i}u_\kappa$
extends to a continuous mapping
in $L^2(\R^n,\C^{2N})$.

Therefore,
if either $Z_j\to Z_0\in\C\setminus\sigma(\eurl\sb{-})$
or
$Z_j\to Z_0\in(1/(2m),+\infty)$
(with $Z_0\ne 1/(2m)$ if $z=1/(2m)$ either an eigenvalue or a virtual level
of $\eurl\sb{-}$),
with $\Im Z_j\ne 0$,
then, taking into account that
$\lambda\sb j\to 2m\jj$
as $\omega_j\to m$,
we conclude that
the relation \eqref{linear-b-afa-inv-away} yields
the following inequality:
\begin{eqnarray}\label{linear-b-psi-psi-psi-psi-v}
\nonumber
&&
\hskip -20pt
\norm{u_\kappa^\kappa\pi\sb{P}\sp{-}\bmPsi\sb j}
+
\epsilon_j\norm{u_\kappa^\kappa\pi\sb{A}\sp{-}\bmPsi\sb j}
\le
C
\epsilon_j^2\norm{
\frac{1}{u_\kappa^\kappa}
\pi\sb{P}\sp{-}\eubV\bmPsi\sb j}
+
C\epsilon_j\bignorm{\frac{1}{u_\kappa^\kappa}
\big(\pi\sb{A}\sp{-}\eubV\bmPsi\sb j
-
u_\kappa^{2\kappa}\pi\sb{A}\sp{-}\bmPsi\sb j\big)}
\\
&&
\hskip -20pt
\le
C
\epsilon_j^2\bignorm{\frac{1}{u_\kappa^\kappa}\pi\sb{P}\sp{-}\eubV\bmPsi\sb j}
+
C\epsilon_j
\Bignorm{
\frac{1}{u_\kappa^\kappa}
\pi\sb{A}\sp{-}\eubV\pi\sb{P}\sp{-}\bmPsi\sb j
+
\frac{1}{u_\kappa^\kappa}
\big(\pi\sb{A}\sp{-}(\eubV-u_\kappa^{2\kappa})\pi\sb{A}\sp{-}\bmPsi\sb j\big)},
\end{eqnarray}
with $C=C(Z_0)>0$,
and then,
using the bounds
\[
\norm{\eubV}_{\End(\C^{2N})}
\le C\abs{u_\kappa(y)}^{2\kappa},
\qquad
\norm{\pi\sb{A}\circ\eubV\circ\pi\sb{P}}_{\End(\C^{2N})}
\le C\epsilon\abs{u_\kappa(y)}^{2\kappa},
\]
and
\[
\norm{\pi\sb{A}\sp{-}(\eubV+u_\kappa^{2\kappa}
(1+2\kappa\Pi_\bmXi)\bmupbeta)\pi\sb{A}\sp{-}}_{\End(\C^{2N})}
\le C\epsilon^{2\varkappa}\abs{u_\kappa(y)}^{2\kappa}
\]
from Lemma~\ref{linear-b-lemma-w}
and taking into account the identity
\[
\pi\sb{A}\sp{-}(\eubV-u_\kappa^{2\kappa})\pi\sb{A}\sp{-}
=
\pi\sb{A}\sp{-}(\eubV+u_\kappa^{2\kappa}
(1+2\kappa\Pi_\bmXi)\bmupbeta)\pi\sb{A}\sp{-},
\]
we obtain the estimate
\begin{equation}\label{linear-b-z-minus-small-away-v}
\norm{u_\kappa^\kappa\pi\sp{-}\bmPsi\sb j}
\le
C\epsilon_j^{\min(1,2\varkappa)}\norm{u_\kappa^\kappa\bmPsi\sb j},
\qquad
\forall j\in\N.
\end{equation}
The inequality \eqref{linear-b-z-minus-small-away-v}
and Lemma~\ref{linear-b-lemma-psi-a-m}
lead to
$\norm{u_\kappa^\kappa\bmPsi\sb j}
=O\big(\epsilon_j^{\min(1,2\varkappa)}\big)
\norm{u_\kappa^\kappa\bmPsi\sb j}$,
in contradiction to $\bmPsi\sb j\not\equiv 0$, $j\in\N$.
Thus, the assumption
$Z_0\in\C\setminus\sigma(\eurl\sb{-})$
leads to a contradiction,
and so does the assumption
$Z_0\in(1/(2m),+\infty)$,
and so does the assumption
$Z_0=1/(2m)$
when the threshold
$z=1/(2m)$
is a regular point of the essential spectrum of $\eurl\sb{-}$.
\end{proof}

This finishes the proof of
Theorem~\ref{linear-b-theorem-2m}~\itref{linear-b-theorem-2m-1}
and~\itref{linear-b-theorem-2m-2}.

\subsection{Characteristic roots of the nonlinear eigenvalue problem}
\label{linear-b-sect-nonlinear}

Let us prove
Theorem~\ref{linear-b-theorem-2m}~\itref{linear-b-theorem-2m-1a},
showing that $Z_0=0$
is only possible when $\lambda_j=2\omega_j\jj$
for all but finitely many $j\in\N$.
First, we claim that the relations
\eqref{linear-b-4p-plus} and \eqref{linear-b-4a-plus}
allow one to express
$\eubY:=\pi\sp{+}\bmPsi\sb j$
in terms of
$\eubX:=\pi\sp{-}\bmPsi\sb j$.

\begin{lemma}\label{linear-b-lemma-zp-zm}
There is
$\upepsilon_2\in(0,\upepsilon_1)$
such that for any $\epsilon\in(0,\upepsilon_2)$
and any $z\in\mathbb{D}\sb 1$
the relations
\begin{eqnarray}
\nonumber
\epsilon\eubD_0\pi\sb{A}\eubY
-(\omega-\jj\lambda-m)\pi\sb{P}\eubY
+
\epsilon^2\pi\sp{+}\sb{P}\eubV(\eubX+\eubY)
=0,
\\
\nonumber
\epsilon\eubD_0\pi\sb{P}\eubY
-(\omega-\jj\lambda+m)\pi\sb{A}\eubY
+\epsilon^2\pi\sp{+}\sb{A}\eubV(\eubX+\eubY)
=0,
\end{eqnarray}
where $\omega=\sqrt{m^2-\epsilon^2}$
and
\begin{eqnarray}\label{linear-b-def-mu}
\lambda=\lambda(z)=(2\omega+\epsilon^2 z)\jj
\end{eqnarray}
(cf. \eqref{linear-b-def-Lambda-j-0}),
define a linear map
$
\vartheta(\cdot,\epsilon,z):
L^{2,-\kappa}(\R^n,\range(\pi\sp{-}))
\to
L^{2,-\kappa}(\R^n,\range(\pi\sp{+}))
,
$
$
\ \vartheta(\cdot,\epsilon,z):\,\eubX\mapsto\eubY,
$
which is analytic in $z$,
where for $\mu\in\R$
the exponentially weighted spaces are defined by
\[
L^{2,\mu}(\R^n)
=\big\{u\in L^2_{\mathrm{loc}}(\R^n);\;
e^{\mu\langle r\rangle}u\in L^2(\R^n)
\big\},
\qquad
\norm{u}_{L^{2,\mu}}
:=\norm{e^{\mu\langle r\rangle}u}_{L^2}.
\]
Moreover,
there is $C>0$
such that
\begin{eqnarray}
\nonumber
\norm{\vartheta(\cdot,\epsilon,z)}
\sb{L^{2,-\kappa}(\R^n,\C^{2N})
\to L^{2,-\kappa}(\R^n,\C^{2N})}
\le C\epsilon,
\qquad
\epsilon\in(0,\upepsilon_2),
\quad
z\in\mathbb{D}\sb 1,
\\
\nonumber
\norm{\p\sb z\vartheta(\cdot,\epsilon,z)}
\sb{L^{2,-\kappa}(\R^n,\C^{2N})\to L^{2,-\kappa}(\R^n,\C^{2N})}
\le C\epsilon^3,
\qquad
\epsilon\in(0,\upepsilon_2),
\quad
z\in\mathbb{D}\sb 1.
\end{eqnarray}
\end{lemma}

\begin{proof}
By \eqref{linear-b-delta-tilde-z},
\[
\begin{bmatrix}
\pi\sb{P}\eubY
\\
\pi\sb{A}\eubY
\end{bmatrix}
=
\begin{bmatrix}
\omega-\jj\lambda+m&\epsilon\eubD_0
\\
\epsilon\eubD_0&\omega-\jj\lambda-m
\end{bmatrix}
\left(
\Delta_y
+
\frac{(\omega-\jj\lambda)^2-m^2}{\epsilon^2} I
\right)^{-1}
\begin{bmatrix}
\pi\sp{+}\sb{P}\eubV(\eubX+\eubY)
\\
\pi\sp{+}\sb{A}\eubV(\eubX+\eubY)
\end{bmatrix}.
\]
This leads to
\begin{equation}\label{linear-b-p-ap-x}
\eubY
=
\pi\sp{+}
\left\{
(\omega-\jj\lambda)
+
m
(\pi\sb{P}-\pi\sb{A})
+
\epsilon\eubD_0
\right\}
\Big(\Delta_y+\frac{(\omega-\jj\lambda)^2-m^2}{\epsilon^2}I\Big)^{-1}
\eubV(\eubX+\eubY).
\end{equation}
For $\epsilon>0$ and $z\in\C$,
$\Im z<0$
(so that $\Re\lambda(z)>0$ by \eqref{linear-b-def-mu}),
we define the linear map
\begin{eqnarray}\label{linear-b-def-Phi}
&
\Phi(\cdot,\epsilon,z):\;L^{2,-\kappa}(\R^n,\C^{2N})
\to L^{2,-\kappa}(\R^n,\range(\pi\sp{+})),
\\
\nonumber
&
\Phi(\bmPsi,\epsilon,z)
=
\pi\sp{+}
\left\{
(\omega-\jj\lambda)
+
m
(\pi\sb{P}-\pi\sb{A})
+
\epsilon\eubD_0
\right\}
\Big(\Delta_y+\frac{(\omega-\jj\lambda)^2-m^2}{\epsilon^2}I\Big)^{-1}
\eubV\bmPsi,
\end{eqnarray}
where $\lambda=\lambda(z)=(2\omega+\epsilon^2 z)\jj$.
Using the definition
\eqref{linear-b-def-Phi},
the relation \eqref{linear-b-p-ap-x} takes the form
\begin{eqnarray}\label{linear-b-ttf}
\eubY=\Phi(\eubX+\eubY,\epsilon,z).
\end{eqnarray}
Since the norm of
$
\Phi(\cdot,\epsilon,z)
\in
\mathscr{B}\big(
L^{2,-\kappa}(\R^n,\C^{2N}),L^{2,-\kappa}(\R^n,\C^{2N})
\big),
$
$
z\in\mathbb{D}\sb 1,
$
is small
(as long as $\epsilon>0$ is small enough),
we will be able to use the above relation
to express $\eubY$ as a function of
$\eubX\in L^{2,-\kappa}(\R^n,\range(\pi\sp{-}))$.

\begin{remark}
The
inverse of
$
\Delta_y+\frac{(\omega-\jj\lambda)^2-m^2}{\epsilon^2}
$
is not continuous
in $\lambda$
at $\Re\lambda=0$;
this discontinuity
could result in two different families of
eigenvalues bifurcating from an embedded eigenvalue
even when its algebraic multiplicity is one.
We will use the resolvent
corresponding to $\Re\lambda>0$ and then
use its analytic continuation
through $\Re\lambda=0$.
\end{remark}

In view of Proposition~\ref{linear-b-prop-rauch-prop3},
it is convenient to change the variables
so that in \eqref{linear-b-def-Phi}
we deal with $(-\Delta_y-\zeta^2)$;
let $\zeta\in\C$
be defined by
\begin{eqnarray}\label{linear-b-def-zeta}
\zeta^2={((\omega-\jj\lambda)^2-m^2)}/{\epsilon^2},
\qquad
\Re\zeta\ge 0;
\end{eqnarray}
the values of $\lambda$ with $\Re\lambda>0$ correspond to
the values of $\zeta$ with $\Im\zeta<0$
(since
$\lambda\in \mathbb{D}_{\upepsilon_2}(2\omega\jj)$
and $\omega\in(\upomega_2,m)$,
where
$\upomega_2=\sqrt{m^2-\upepsilon_2^2}$,
with $\upepsilon_2\in(0,\upepsilon_1)$ small enough).

Due to Lemma~\ref{linear-b-lemma-w},
$\norm{\eubV(y,\epsilon)}\sb{\End(\C^{2N})}\le C e^{-2\kappa\abs{y}}$;
using the analytic continuation
of the resolvent from Proposition~\ref{linear-b-prop-rauch-prop3},
the mapping \eqref{linear-b-def-Phi}
could be extended
from $\{\Im\zeta<0\}$
to
$
\{\zeta\in\C;\, \Im\zeta<\kappa\}
\setminus\overline{\jj\R_{+}}.
$
For the uniformity,
we require that
\begin{eqnarray}\label{linear-b-condition-on}
\abs{\Im\zeta}
<\kappa,
\end{eqnarray}
considering the resolvent
$(-\Delta_y-\zeta^2)^{-1}$
for $\Im\zeta<0$
(this corresponds to $\Re\lambda>0$)
and its analytic continuation
into the strip $0\le\Im\zeta<\kappa$
(this corresponds to $\Re\lambda\le 0$).
Due to our assumptions that
$\omega\to m$ and $\lambda\to 2 m \jj$,
one has
\begin{eqnarray}\label{linear-b-re-im}
\Re((\omega-\jj\lambda)^2-m^2)
=8m^2+O(\Re\lambda)+O(\Im\lambda-2m)+O(\epsilon^2).
\end{eqnarray}
Therefore, by \eqref{linear-b-def-zeta},
\begin{eqnarray}\label{linear-b-re-zeta-large}
\Re\zeta
=
\epsilon^{-1}
\Re\sqrt{(\omega-\jj\lambda)^2-m^2}
=O(\epsilon^{-1}),
\end{eqnarray}
showing that
for $\epsilon$ sufficiently small
one has $\zeta\in\C\setminus\mathbb{D}\sb 1$
(we take $\upepsilon_2>0$ smaller if necessary).
Since we only consider $z\in\mathbb{D}\sb 1$,
the relation \eqref{linear-b-def-mu}
yields
$\abs{\Re\lambda}\le\abs{z}\epsilon^2\le\epsilon^2$,
and then
\eqref{linear-b-def-zeta}
and \eqref{linear-b-re-zeta-large}
lead to
\[
\Im\zeta
=\frac{1}{2\Re\zeta}\Im\zeta^2
=O(\epsilon)
\frac{\Im((\omega-\jj\lambda)^2-m^2)}{\epsilon^2}
=O(\epsilon)
\frac{O(\Re\lambda)}{\epsilon^2}
=O(\epsilon),
\]
showing that the condition \eqref{linear-b-condition-on}
holds true for $\epsilon$ sufficiently small,
satisfying assumptions of Proposition~\ref{linear-b-prop-rauch-prop3}.
Then, by Proposition~\ref{linear-b-prop-rauch-prop3},
there is $C>0$ such that
for any $\bmPsi\in L^{2,-\kappa}(\R^n,\C^{2N})$
the map \eqref{linear-b-def-Phi} satisfies
\begin{equation}\label{linear-b-phi-x}
\norm{\Phi(\bmPsi,\epsilon,z))}\sb{L^{2,-\kappa}(\R^n,\C^{2N})}
\le C\epsilon\norm{\bmPsi}\sb{L^{2,-\kappa}(\R^n,\C^{2N})},
\quad
\epsilon\in(0,\upepsilon_2),
\quad
z\in\mathbb{D}\sb 1.
\end{equation}
We take $\upepsilon_2$ smaller if necessary so that
$\upepsilon_2\le 1/(2C)$ (with $C>0$ from \eqref{linear-b-phi-x});
then the linear map
\[
I-\Phi(\cdot,\epsilon,z):\,\eubY\mapsto\eubY-\Phi(\eubY,\epsilon,z)
\]
is invertible, with
\begin{eqnarray}\label{linear-b-y-smaller-x}
\norm{(I-\Phi(\cdot,\epsilon,z))^{-1}}\sb{L^{2,-\kappa}\to L^{2,-\kappa}}
\le 2,
\qquad
\epsilon\in(0,\upepsilon_2),
\qquad
z\in\mathbb{D}\sb 1.
\end{eqnarray}
Since $\Phi(\cdot,\epsilon,z)$ is linear,
writing \eqref{linear-b-ttf} in the form
$\eubY-\Phi(\eubY)=\Phi(\eubX)$,
we can express $\eubY=(I-\Phi)^{-1}\Phi(\eubX)$.
Thus, for each
$\epsilon\in(0,\upepsilon_2)$
and $z\in\mathbb{D}_1$,
we may define the mapping
$(\eubX,\epsilon,z)
\mapsto \eubY$,
which we denote by $\vartheta$:
\begin{eqnarray}\label{linear-b-def-vartheta}
&&
\vartheta(\cdot,\epsilon,z):\;
L^{2,-\kappa}(\R^n,\range(\pi\sp{-}))
\to
L^{2,-\kappa}(\R^n,\range(\pi\sp{+})),
\nonumber
\\[1ex]
&&
\vartheta(\cdot,\epsilon,z):\;
\eubX\mapsto\eubY=(I-\Phi(\cdot,\epsilon,z))^{-1}\Phi(\eubX,\epsilon,z).
\end{eqnarray}
By \eqref{linear-b-phi-x} and \eqref{linear-b-y-smaller-x},
one has
$
\norm{\vartheta(\cdot,\epsilon,z)}\sb{L^{2,-\kappa}\to L^{2,-\kappa}}
\le 2C\epsilon,
$
for
$\epsilon\in(0,\upepsilon_2)$
and $z\in\mathbb{D}\sb 1$.

Finally, let us discuss the differentiability
of $\vartheta$ with respect to $z$.
The map $\Phi$
can be differentiated in the strong sense
with respect to $z$.
First, we notice that, by
\eqref{linear-b-def-zeta} and then by \eqref{linear-b-def-mu},
\begin{eqnarray}\label{linear-b-d-zeta-d-z}
2\abs{\zeta \p_z\zeta}
=\Abs{
\p_z
\frac{(\omega-\jj\lambda)^2-m^2
}{\epsilon^2}
}
=
\Abs{
\frac{2(\omega-\jj\lambda)}{\epsilon^2}
\p_z(2\omega+\epsilon^2 z)}
=
2\abs{\omega-\jj\lambda},
\end{eqnarray}
with the right-hand side bounded uniformly in
$\epsilon\in(0,\upepsilon_2)$ and $z\in\mathbb{D}\sb 1$.
Therefore, using the bound for the derivative
of the analytic continuation of the resolvent
(cf. Proposition~\ref{linear-b-prop-rauch-prop3},
which we apply with $\nu=0$ and also with $\nu=1$ to accommodate
the operator $\epsilon\eubD_0$ from the definition of $\Phi$
in \eqref{linear-b-def-Phi}),
we conclude that there is $C>0$ such that
\[
\norm{\p\sb z\Phi(\cdot,\epsilon,z)}\sb{L^{2,-\kappa}\to L^{2,-\kappa}}
\le
\norm{\p\sb\zeta\Phi}\sb{L^{2,-\kappa}\to L^{2,-\kappa}}
\abs{\p\sb z\zeta}
\le
\frac{C}{\langle\zeta\rangle^2}
\frac{\abs{\omega-\jj\lambda}}{\abs{\zeta}}
\le
C\epsilon^3
\]
for all
$
\epsilon\in(0,\upepsilon_2)
$
and
$
z\in\mathbb{D}\sb 1.
$
Above,
$\p\sb z$ is considered as a gradient in $\R^2\cong\C$;
we used \eqref{linear-b-d-zeta-d-z}
and the estimate \eqref{linear-b-re-zeta-large}.
Then it follows from \eqref{linear-b-def-vartheta} that
there is $C>0$ such that one also has
\[
\norm{\p\sb z\vartheta(\cdot,\epsilon,z)}\sb{L^{2,-\kappa}\to L^{2,-\kappa}}
\le
C\epsilon^3,
\]
for all
$
\epsilon\in(0,\upepsilon_2)
$
and
$
z\in\mathbb{D}\sb 1.
$
One can see from
\eqref{linear-b-def-Phi}
that $\Phi$ is analytic in the complex parameter $z$,
hence so is $\vartheta$.
\end{proof}

As long as $j\in\N$ is sufficiently large
so that $\epsilon_j\in(0,\upepsilon_2)$ and $z\sb j\in\mathbb{D}\sb 1$,
then, by Lemma~\ref{linear-b-lemma-zp-zm},
the relations \eqref{linear-b-4p-plus} and \eqref{linear-b-4a-plus}
allow us to express
\[
\pi\sp{+}\bmPsi\sb j
=\vartheta(\pi\sp{-}\bmPsi\sb j,\epsilon_j,z\sb j),
\]
with $\vartheta(\cdot,\epsilon,z)$
a linear map from $L^{2,-\kappa}(\R^n,\C^{2N})$
into itself,
with
\[
\norm{\vartheta(\cdot,\epsilon,z)}
\sb{L^{2,-\kappa}(\R^n,\C^{2N})\to L^{2,-\kappa}(\R^n,\C^{2N})}
\le C\epsilon.
\]
Now \eqref{linear-b-4p-minus} and \eqref{linear-b-4a-minus}
can be written as
\begin{eqnarray}
\label{linear-b-4p-minus-1}
&&
\epsilon_j\eubD_0\pi\sp{-}\sb{A}\bmPsi\sb j
+(m-\omega_j-\jj\lambda\sb j)\pi\sp{-}\sb{P}\bmPsi\sb j
+\epsilon_j^2\pi\sp{-}\sb{P}
\eubV\sp{\vartheta}\pi\sp{-}\bmPsi\sb j
=0,
\\
\label{linear-b-4a-minus-1}
&&
\epsilon_j\eubD_0\pi\sp{-}\sb{P}\bmPsi\sb j
-(m+\omega_j+\jj\lambda\sb j)\pi\sp{-}\sb{A}\bmPsi\sb j
+\epsilon_j^2\pi\sp{-}\sb{P}
\eubV\sp{\vartheta}\pi\sp{-}\bmPsi\sb j
=0,
\end{eqnarray}
where
$\eubV\sp\vartheta=\eubV\sp\vartheta(y,\epsilon_j,\lambda\sb j)$,
with
$
\eubV\sp{\vartheta}(y,\epsilon,z)
:=\eubV(y,\epsilon)\circ(1+\vartheta(\cdot,\epsilon,z))
$
satisfying
\[
\norm{\eubV\sp{\vartheta}(\epsilon,\lambda)}\sb{L^2\to L^2}
\le
\norm{\eubV(\epsilon)}\sb{L^{2,-\kappa}\to L^2}
\big(
1+\norm{\vartheta(\cdot,\epsilon,z)}\sb{L^{2,-\kappa}\to L^{2,-\kappa}}
\big)
\le C,
\]
so that
\[
\eubV(\epsilon_j)\bmPsi\sb j
=\eubV(\epsilon_j)
(\pi\sp{-}\bmPsi\sb j+\pi\sp{+}\bmPsi\sb j)
=\eubV(\epsilon_j)
\big(
\pi\sp{-}\bmPsi\sb j+\vartheta(\pi\sp{-}\bmPsi\sb j,\epsilon_j,z\sb j)
\big)
=\eubV\sp{\vartheta}
\pi\sp{-}\bmPsi\sb j.
\]
We recall the definition
$Z_j
=-(2\omega_j+\jj\lambda_j)/\epsilon_j^2$
(cf. \eqref{linear-b-def-Lambda-j})
and rewrite
\eqref{linear-b-4p-minus-1}, \eqref{linear-b-4a-minus-1},
as the following system:
\begin{eqnarray}\label{linear-b-m-z-z}
\begin{bmatrix}
\pi\sp{-}\sb{P}
(\frac{m+\omega_j}{\epsilon_j^2}+Z_j+\eubV\sp{\vartheta})
\pi\sp{-}\sb{P}
&
\pi\sp{-}\sb{P}
(\epsilon_j^{-1}\eubD_0+\eubV\sp{\vartheta})
\pi\sp{-}\sb{A}
\\
\pi\sp{-}\sb{A}
(\epsilon_j^{-1}\eubD_0+\eubV\sp{\vartheta})
\pi\sp{-}\sb{P}
&
\pi\sp{-}\sb{A}
(
Z_j
-\frac{1}{m+\omega_j}
+\eubV\sp{\vartheta})
\pi\sp{-}\sb{A}
\end{bmatrix}
\begin{bmatrix}
\pi\sp{-}\sb{P}\bmPsi\sb j
\\
\pi\sp{-}\sb{A}\bmPsi\sb j
\end{bmatrix}
=0,
\end{eqnarray}
with $j\in\N$.
We rewrite the above as the
\emph{nonlinear eigenvalue problem}
\begin{eqnarray}\label{linear-b-def-M}
&
A(\epsilon,z)
\begin{bmatrix}
\pi\sb{P}\sp{-}\bmPsi
\\
\pi\sb{A}\sp{-}\bmPsi
\end{bmatrix}
=0,
\nonumber
\\
&
A(\epsilon,z)
:=
\begin{bmatrix}
\pi\sp{-}\sb{P}
(\frac{m+\omega}{\epsilon^2}+z+\eubV\sp{\vartheta})
\pi\sp{-}\sb{P}
&
\pi\sp{-}\sb{P}
(\epsilon^{-1}\eubD_0+\eubV\sp{\vartheta})
\pi\sp{-}\sb{A}
\\
\pi\sp{-}\sb{A}
(\epsilon^{-1}\eubD_0+\eubV\sp{\vartheta})
\pi\sp{-}\sb{P}
&
\pi\sp{-}\sb{A}
(
z
-\frac{1}{m+\omega}
+\eubV\sp{\vartheta})
\pi\sp{-}\sb{A}
\end{bmatrix},
\end{eqnarray}
where the operator
$A(\epsilon,z):\;
H^1(\R^n,\range(\pi\sp{-}))
\to
L^2(\R^n,\range(\pi\sp{-}))
$
is considered as
\[
A(\epsilon,z):\;
L^2\big(\R^n,\range(\pi\sb{P}\sp{-})\times \range(\pi\sb{A}\sp{-})\big)
\to
L^2\big(\R^n,\range(\pi\sb{P}\sp{-})\times\range(\pi\sb{A}\sp{-})\big),
\]
with domain
$\dom(A(\epsilon,z))
=H^1(\R^n,\range(\pi\sb{P}\sp{-})\times \range(\pi\sb{A}\sp{-}))$.
Note that the operator $A(\epsilon,z)$
depends on $z$
analytically via $\vartheta$
(cf. Lemma~\ref{linear-b-lemma-zp-zm}).
By Weyl's theorem,
\begin{eqnarray}\label{linear-b-zero-not-in}
\sigma\sb{\mathrm{ess}}(A(\epsilon,z))
=
\big(-\infty,-(m+\omega)^{-1}+z\big]
\cup
\big[(m-\omega)^{-1}+z,+\infty\big),
\end{eqnarray}
so that $0\not\in\sigma\sb{\mathrm{ess}}(A(\epsilon,z))$.
Thus, the values $Z_j$ defined in \eqref{linear-b-def-Lambda-j}
are such that
the kernel of $A(\epsilon_j,z)$ is nontrivial at $z=Z_j$;
such values of $z$
are called the \emph{characteristic roots}
(or, informally, \emph{nonlinear eigenvalues}) of $A(\epsilon,z)$.

\medskip
\noindent
{\bf Nonlinear eigenvalue problem.}
We will study the location
of characteristic roots of $A(\epsilon,z)$
using the theory developed by
M. Keldysh \cite{MR0041353,MR0300125};
see also \cite{MR0312306,MR0313856}.
Let $X$ be a Hilbert space
and $\Omega\subset\C$ be an open neighborhood.
We assume that the family of closed operators
$A(z):\,X\to X$, $z\in\Omega$,
is a \emph{holomorphic family of type (A)} \cite[Section VII-2]{MR0407617}:
that is, we assume that
operators have the same domain $\dom(A(z))=\dom$,
$\forall z\in\Omega$,
with $\dom$ dense in $X$,
and that for any $u\in \dom$
the vector-function $A(z)u$ is holomorphic in $\Omega$.
It follows that
for any $u,\,v\in\dom$, each $z_0\in\Omega$,
and each $\eta$ in the resolvent set of $A(z_0)$,
the function
$\langle u,(A(z)-\eta)^{-1}v\rangle$
is analytic in $z$ in an open neighborhood of $z_0$
and that
$A(z)$ is resolvent-continuous in $z$ (in the norm topology).


\begin{definition}
The point $z_0\in\Omega$
is said to be \emph{regular} for the operator-valued analytic function $A(z)$
if the operator $A(z_0)$ has
a bounded inverse.
If the equation $A(z_0)\varphi=0$
has a non-trivial solution $\varphi_0\in \dom$, then $z_0$ is
said to be a \emph{characteristic root} of $A$ and $\varphi_0$
an eigenvector of $A$ corresponding to $z_0$.
\end{definition}

\begin{assumption}\label{linear-b-ass-characteristic}
$z_0\in\Omega$ is normal characteristic root of $A(z)$.
\end{assumption}

This means that
$A(z)$ has a bounded inverse for $0<\abs{z-z_0}<R$,
with some $R>0$,
and $0\in\sigma\sb{\mathrm{d}}(A(z_0))$;
see \cite{MR0041353,MR0300125}.

Due to Assumption~\ref{linear-b-ass-characteristic},
there is $\delta>0$
such that
$\overline{\mathbb{D}_\delta}\cap\sigma(A(z_0))=\{0\}$.
Due to the resolvent continuity of $A$ in $z$,
there is an open neighborhood $U\subset\Omega$ of $z_0$
such that
$\p\mathbb{D}_\delta\subset\rho(A(z))$
for all $z\in U$.
Let
\[
P_{\delta,z}=-\frac{1}{2\pi\jj}\oint\sb{\abs{\eta}=\delta}(A(z)-\eta)^{-1}\,d\eta,
\qquad
z\in U.
\]
One has
$\rank P_{\delta,z}=\rank P_{\delta,z_0}<\infty$
(the last inequality due to the assumption
that $0\in\sigma\sb{\mathrm{d}}(A(z_0))$).

\begin{definition}
The multiplicity $\alpha\in\N$
of the characteristic root $z_0$ of $A(z)$
is the order of vanishing
of $\det A(z)\at{\range(P_{\delta,z})}$
at $z_0$.
\end{definition}

\begin{remark}
The above definition
does not depend on a particular choice of
$\delta>0$.
\end{remark}

\begin{lemma}
\label{linear-b-lemma-c-roots-simple-0}
Let $z_0$ be a normal characteristic root of $A(z)$
of multiplicity $\alpha\in\N$.
The geometric multiplicity of $0\in\sigma\sb{\mathrm{d}}(A(z_0))$
satisfies
$g\le \alpha$.
\end{lemma}

\begin{proof}
We denote $\nu=\dim\range(P_{\delta,z_0})<\infty$
and choose the basis
$\{\psi_i\}\sb{1\le i\le\nu}$ in $\range(P_{\delta,z_0})$
so that $\psi_i$, $1\le i\le g$,
are eigenvectors corresponding to
zero eigenvalue of $A(z_0)$.
Let $\mathbf{A}(z)$
be the matrix representation of $A(z)$
in the basis $\{P_{\delta,z}\psi_i\}\sb{1\le i\le\nu}$.
Then the first $g$ columns
of $\mathbf{A}(z)$ vanish at $z=z_0$,
hence $\det\mathbf{A}(z)=O((z-z_0)^{g})$.
\end{proof}

The sum of multiplicities of characteristic roots
is stable under perturbations
(cf. \cite[Theorem 2.1]{MR0313856}):

\begin{theorem}
\label{linear-b-theorem-c-roots-stability}
Let $A(\epsilon,z):\,\bfX\to \bfX$, $\epsilon\in\R$, $z\in\Omega$,
be a continuous family of operators with domains $\dom(A(\epsilon,z))=\dom$,
$\ \forall (\epsilon,z)\in\R\times\Omega$,
which is resolvent-continuous in $\epsilon$ and $z$ and analytic in $z\in\Omega$.
Let $z_0$ be a normal characteristic root of $A(0,z)$
of multiplicity $\alpha\in\N$.
There is an open neighborhood $\Omega\subset\C$ of $z_0$
such that
if $\epsilon\ge 0$ is sufficiently small,
then
the sum of multiplicities of all characteristic values of $A(\epsilon,z)$
inside $\Omega$
equals $\alpha$.
\end{theorem}


\begin{proof}
Fix $\delta>0$ sufficiently small  so that
$\sigma(A(0,z_0))\cap\overline{\mathbb{D}_\delta}=\{0\}$.
Due to the resolvent continuity,
there is $\epsilon_1>0$
and an open neighborhood $\Omega\subset\C$ of $z_0$
such that for all $\epsilon\in[0,\epsilon_1]$ and all $z\in\Omega$
one has
$\sigma(A(\epsilon,z))\cap\overline{\mathbb{D}_\delta}=\{0\}$.
Let $\Gamma=\p\mathbb{D}_\delta$ and denote
\[
P_\Gamma(\epsilon,z)
=-\frac{1}{2\pi\jj}\oint\sb\Gamma
(A(\epsilon,z)-\eta I)^{-1}\,d\eta,
\qquad
\epsilon\in[0,\epsilon_1],
\quad
z\in\Omega.
\]
Due to the resolvent continuity of $A(\epsilon,z)$ in $\epsilon$ and $z$,
one has
\[
\rank P_\Gamma(\epsilon,z)
=
\rank P_\Gamma(0,z_0)=:\nu,
\qquad
\forall\epsilon\in[0,1],
\quad
\forall z\in\Omega.
\]
Let $\big(\psi_i\big)\sb{1\le i\le\nu}$ be the basis in
$\range(P_\Gamma(0,z_0))$.
Denote
\[
\psi_i(\epsilon,z)
=P_\Gamma(\epsilon,z)\psi_i,
\qquad
1\le i\le\nu,
\quad
\epsilon\in[0,1],
\quad
z\in \Omega;
\]
it is a basis in $\range(P_\Gamma(\epsilon,z))$
(we take $\Omega$ and $\epsilon_1>0$ smaller if necessary).
Let $ \mathbf{M}(\epsilon,z)$
be the matrix representation of
$A(\epsilon,z)$,
restricted onto $\range(P_\Gamma(\epsilon,z))$,
in this basis.
If $\epsilon\in[0,\epsilon_1]$
and
$z_j\in\Omega$ is a characteristic root of $A(\epsilon,z)$,
then its multiplicity equals the order of vanishing of $\det\mathbf{M}(\epsilon,z)$
at $z_j$.
Now the proof follows from Cauchy's argument principle
applied to $\det\mathbf{M}(\epsilon,z)$ in $\Omega$
when we take $\Omega$ and $\epsilon_1>0$ small enough
so that zeros of $\det\mathbf{M}(\epsilon,z)$ do not intersect $\p\Omega$.
\end{proof}

Now we apply the above theory to
the operator $A(\epsilon,z)$
defined in \eqref{linear-b-def-M};
first, we will do the reduction of $A$
using the Schur complement of its invertible block.
By \eqref{linear-b-zero-not-in},
we may assume that
there is a sufficiently small open neighborhood $U$ of $z=0$
and that $\omega\sb\ast\in(0,m)$ is sufficiently large
so that
\[
\sigma\sb{\mathrm{ess}}(A(\epsilon,z))\cap\mathbb{D}_{1/(4m)}
=\emptyset
\qquad
\forall\omega\in(\omega\sb\ast,m),
\quad
\forall\epsilon\in(0,\epsilon\sb\ast),
\quad
\forall
z\in U,
\]
where
$\epsilon\sb\ast=\sqrt{m^2-\omega\sb\ast^2}$.
Since $0\in\sigma\sb{\mathrm{d}}(\eurl\sb{-})$,
with
$\eurl\sb{-}$ from \eqref{linear-b-def-l-small-pm},
we may assume that the open neighborhood $U\ni\{0\}$
is small enough so that
\begin{eqnarray}\label{linear-b-U-small}
U\cap\sigma(\eurl\sb{-})=\{0\}.
\end{eqnarray}
Let
$A_{i j}(\epsilon,z)$,
$i,j=1,\,2$,
denote the operators
which are the entries of $A(\epsilon,z)$
in \eqref{linear-b-def-M}, so that
$
A(\epsilon,z)
=
\begin{bmatrix}
A_{11}(\epsilon,z)&A_{12}(\epsilon,z)
\\
A_{21}(\epsilon,z)&A_{22}(\epsilon,z)\end{bmatrix}.
$
We then have:
\begin{eqnarray}\label{linear-b-b-c-d-small-2w}
\begin{array}{c}
\norm{A_{12}}\sb{H^1\to L^2}
+\norm{A_{21}}\sb{H^1\to L^2}
+\norm{A_{21}}\sb{L^2\to H^{-1}}
=O(\epsilon^{-1}),
\\[2ex]
\norm{A_{11}}\sb{L^2\to L^2}=O(\epsilon^{-2}),
\end{array}
\end{eqnarray}
and, taking $\epsilon\sb\ast>0$ smaller if necessary,
one has
\begin{eqnarray}\label{linear-b-d-small-2w}
A_{11}^{-1}
=
\frac{\epsilon^2}{2m}+O\sb{L^2\to L^2}(\epsilon^4),
\qquad
\forall\epsilon\in(0,\epsilon\sb\ast),
\qquad
\forall z\in U,
\end{eqnarray}
with the estimate $O\sb{L^2\to L^2}(\epsilon^4)$
uniform in $z\in U$.
This suggests that we study the invertibility of $A(\epsilon,z)$
in terms of the Schur complement of $A_{11}(\epsilon,z)$,
which is defined by
\begin{eqnarray}\label{linear-b-def-s}
S(\epsilon,z)
=A_{22}-A_{21}A_{11}^{-1}A_{12}
:\;
H^1(\R^n,\range(\pi\sb{A}\sp{-}))
\to H^{-1}(\R^n,\range(\pi\sb{A}\sp{-})),
\end{eqnarray}
where $A_{11},\,A_{12},\,A_{21}$, and $A_{22}$
are evaluated at $(\epsilon,z)\in(0,\epsilon\sb\ast)\times U$.
Let us derive the explicit expression for $S(\epsilon,z)$:
\begin{eqnarray}
S(\epsilon,z)
\!\!\!\!&=&\!\!\!\!
\pi\sp{-}\sb{A}
\big(
z-\frac{1}{m+\omega}
+\eubV\sp\vartheta
\big)
\pi\sp{-}\sb{A}
\nonumber
\\
&&
-
\pi\sp{-}\sb{A}
(\eubD_0+\epsilon\eubV\sp\vartheta)
\pi\sp{-}\sb{P}
\big(
m+\omega+\epsilon^2(\eubV\sp\vartheta+z)
\big)^{-1}
\pi\sp{-}\sb{P}
(\eubD_0+\epsilon\eubV\sp\vartheta)
\pi\sp{-}\sb{A}
\nonumber
\\
\!\!\!\!&=&\!\!\!\!
\pi\sp{-}\sb{A}
\big(
z
-\frac1{m+\omega}
+u_\kappa^{2\kappa}
+\frac{\Delta_y}{m+\omega}
\big)
\pi\sp{-}\sb{A}
+
\pi\sp{-}\sb{A}
\big(
\eubV\sp\vartheta-u_\kappa^{2\kappa}
\big)
\pi\sp{-}\sb{A}
\nonumber
\\
&&
\!\!\!\!
+\pi\sp{-}\sb{A}
\frac{\eubD_0^2}{m+\omega}
\pi\sp{-}\sb{A}
-
\pi\sp{-}\sb{A}
(\eubD_0+\epsilon\eubV\sp\vartheta)
\pi\sp{-}\sb{P}
\frac{1}{m+\omega}
\pi\sp{-}\sb{P}
(\eubD_0+\epsilon\eubV\sp\vartheta)
\pi\sp{-}\sb{A}
\nonumber
\\
&&
\!\!\!\!
-
\pi\sp{-}\sb{A}
(\eubD_0+\epsilon\eubV\sp\vartheta)
\pi\sp{-}\sb{P}
\frac{1}{m+\omega}
\Big(
\Big(
1+
\frac{\epsilon^2(\eubV\sp\vartheta+z)}{m+\omega}
\Big)^{-1}
\!-1
\Big)
\pi\sp{-}\sb{P}
(\eubD_0+\epsilon\eubV\sp\vartheta)
\pi\sp{-}\sb{A}.
\nonumber
\end{eqnarray}
Above, $u_\kappa=u_\kappa(x)$ is the ground state
of the nonlinear Schr\"odinger equation \eqref{dirac-existence-def-uu}.
We note that, by
\eqref{linear-b-def-eub-v},
\[
\pi\sp{-}\sb{A}
\eubV\sp\vartheta
\pi\sp{-}\sb{A}
=
\pi\sp{-}\sb{A}
\eubV\circ(1+\vartheta)
\pi\sp{-}\sb{A}
=
\pi\sp{-}\sb{A}u_\kappa^{2\kappa}
+
O_{L^2\to L^2}(\epsilon);
\]
we used the bounds
$\norm{\pi\sb{A}\bmupphi\sb\omega}
\sb{L^\infty}=O(\epsilon^{1+\frac 1 \kappa})$
(cf. Theorem~\ref{dirac-existence-theorem-solitary-waves-c1})
and
$\norm{\vartheta}\sb{L^2_{-s}\to L^2_{-s}}
=O(\epsilon)$
(cf. Lemma~\ref{linear-b-lemma-zp-zm})
which yield
$\norm{\eubV\circ\vartheta}\sb{L^2\to L^2}
\le
\norm{\eubV}\sb{L^2_{-s}\to L^2_s}\norm{\vartheta}\sb{L^2_{-s}\to L^2_{-s}}
=O(\epsilon)$
for $s>1/2$.
Thus,
taking into account Lemma~\ref{linear-b-lemma-lm-inverse-h1},
the operator $S(\epsilon,z)$ defined in \eqref{linear-b-def-s}
takes the form
\begin{eqnarray}
\label{linear-b-s-z-eta}
S(\epsilon,z)
=
\pi\sp{-}\sb{A}
\Big(
z
-\frac{1}{2m}
+u_\kappa^{2\kappa}
+\frac{\Delta_y}{2m}
+
O\sb{H^1\to H^{-1}}(\epsilon)\Big)
\pi\sp{-}\sb{A},
\qquad
\epsilon\in[0,\epsilon\sb\ast),
\end{eqnarray}
with the estimate
$O\sb{H^{1}\to H^{-1}}(\epsilon)$
uniform in
$z\in U$.
Above, we extended $S(\epsilon,z)$
in \eqref{linear-b-s-z-eta}
from $\epsilon\in(0,\epsilon\sb\ast)$
to $\epsilon\in[0,\epsilon\sb\ast)$
by continuity.


The following lemma
allows us to reduce the problem
of studying the characteristic roots of $A(\epsilon,z)$
(see \eqref{linear-b-def-M})
to the characteristic roots of
$S(\epsilon,z)$ (cf. \eqref{linear-b-def-s}).

\begin{lemma}\label{linear-b-lemma-t-is-s}
If $\epsilon\sb\ast>0$ is sufficiently small,
then for all
$\epsilon\in[0,\epsilon\sb\ast)$
the point $z_0\in\mathbb{D}_{1/(2m)}$
is a characteristic root of $A(\epsilon,z)$
if and only if
it is a characteristic root of $S(\epsilon,z)$.
\end{lemma}

\begin{proof}
Since the operator $A_{11}(\epsilon,z):\,L^2(\R^n,\range(\pi\sb{P}\sp{-}))
\to L^2(\R^n,\range(\pi\sb{P}\sp{-}))$
is invertible for $\epsilon>0$ small enough
(see \eqref{linear-b-d-small-2w}),
the expression for the Schur complement \eqref{linear-b-Schur}
and the relation
\eqref{linear-b-t-3}
shows that the operator
$A(\epsilon,z)$
from \eqref{linear-b-def-M}
is invertible if and only if so is $S(\epsilon,z)$.
\end{proof}

\begin{lemma}\label{linear-b-lemma-m-ast}
The multiplicity of the characteristic root
$z_0=0$ of $S(0,z)$
is $\alpha=N/2$.
\end{lemma}

\begin{proof}
By \eqref{linear-b-s-z-eta},
$S(0,0)=(z-\eurl\sb{-})\pi\sb{A}\sp{-}$.
Since $\dim\ker(\eurl\sb{-})=1$,
one has
$\dim\ker(S(0,0))
=\rank\pi\sb{A}\sp{-}=N/2$.
Let $\bm{e}_i\in\C^{2N}$, $1\le i\le N/2$,
be the basis in $\range(\pi\sb{A}\sp{-})$;
then
$\big\{u_\kappa\bm{e}_i\big\}\sb{1\le i\le N/2}$
is the basis in $\ker(S(0,0))$.
We do not need to use the Riesz projectors
since the operator $S(0,z)$ is invariant in this space,
being represented by
$\mathbf{S}(0,z)=z I_{N/2}$;
thus,
$
\det\mathbf{S}(0,z)=z^{N/2}.
$
\end{proof}

\begin{lemma}\label{linear-b-lemma-characteristic-roots}
There is no sequence of characteristic roots
$Z_j\ne 0$ of $A(\epsilon_j,z)$
such that
$Z_j\to 0$ as $j\to\infty$.
\end{lemma}

\begin{proof}
Let $\delta>0$ be such that
$\p\mathbb{D}_\delta\subset\rho(A(\epsilon,z_0))$,
$z_0=0$.
Due to the continuity of the resolvent
in $z$ and $\epsilon$,
there is
$\epsilon\sb\ast\in(0,\upepsilon_2)$
and an open neighborhood
$U\subset\mathbb{D}\sb{1/(2m)}$,
$z_0\in U$,
such that
$\p\mathbb{D}\sb{\delta}\subset\rho(A(\epsilon,z))$
for all $z\in U$
and $\epsilon\in[0,\epsilon\sb\ast)$.

By \eqref{linear-b-U-small} and Lemma~\ref{linear-b-lemma-m-ast},
the sum of multiplicities of the characteristic roots of $S(0,z)$
in $U$ equals $N/2$,
and by Theorem~\ref{linear-b-theorem-c-roots-stability}
the same is true for $S(\epsilon,z)$
for all $\epsilon\in[0,\epsilon\sb\ast)$.
At the same time,
by Lemma~\ref{linear-b-lemma-two-omega-n}
and
Lemma~\ref{linear-b-lemma-c-roots-simple-0},
$z=0$ is a characteristic root
of $S(\epsilon,z)$,
$\epsilon\in[0,\epsilon\sb\ast)$,
of multiplicity
\emph{at least} $\alpha=N/2$.
Hence, there can be no other, nonzero characteristic roots
$z\in U$
of $S(\epsilon,z)$
for any $\epsilon\in[0,\epsilon\sb\ast)$,
and, in particular,
given a sequence $\epsilon_j\to 0$,
there is no sequence of characteristic roots $Z_j$ of $S(\epsilon_j,z)$
such that $Z_j\ne 0$ for $j\in\N$, $Z_j\to 0$ as $j\to\infty$.
By Lemma~\ref{linear-b-lemma-t-is-s},
the same conclusion holds true for $A(\epsilon,z)$.
\end{proof}

By Lemma~\ref{linear-b-lemma-characteristic-roots},
$Z_j=0$ for all but finitely many $j\in\N$.
By \eqref{linear-b-def-Lambda-j-0}, this implies that
$\lambda_j=2\omega_j\jj$
for all but finitely many $j\in\N$.
This concludes the proof of Theorem~\ref{linear-b-theorem-2m}~\itref{linear-b-theorem-2m-1a}.

The proof of
Theorem~\ref{linear-b-theorem-2m} is now complete.

\newpage

\appendix

\section{Appendix:
Analytic continuation of the free resolvent}
\label{linear-b-sect-analytic-continuation}

\ac{OLD VERSION:}
Let us remind the limiting absorption principle
for the free resolvent
\cite[Remark 2 in Appendix A]{MR0397194}
and
\cite[Theorem 8.1]{MR544248}.

\begin{lemma}[Limiting absorption principle for the Laplace operator]
\label{linear-b-lemma-lap-agmon-old}
Let $n\ge 1$.
For any
$k\in\N_0$,
\ac{Do we need $k\ne 0$ at all?}
$\nu\le 2+2k$,
\ac{Don't we also need $\nu\in\N_0$? Perhaps at least
$\nu\ge 0$?}
$s>1/2+k$,
and $\delta>0$,
there is $C=C(n,s,k,\nu,\delta)<\infty$
such that
\[
\norm{\p\sb z^k(-\Delta-z)^{-1}}
\sb{L^2\sb s(\R^n)\to H^\nu\sb{-s}(\R^n)}
\le C\abs{z}^{-(k+1-\nu)/2},
\qquad
z\in\C\setminus(\mathbb{D}\sb\delta\cup\R\sb{+}).
\]
\end{lemma}

\begin{proof}
For $\nu=0$,
the lemma rephrases \cite[Theorem 8.1]{MR544248}
(stated for $n=3$)
or \cite[Theorem A.1 and Remark 2 in Appendix A]{MR0397194}.
The recurrence based on the identities
\[
-\Delta(-\Delta-z)^{-1}=1+z(-\Delta-z)^{-1}
\quad\mbox{and}
\quad
\p\sb z^k(-\Delta-z)^{-1}=k!(-\Delta-z)^{-k-1},
\quad
k\in\N_0,
\]
provides all other cases.
\end{proof}

\ac{A new version for JFA:}

Let us remind the limiting  absorption principle
for the resolvent of the free Laplacian.

\begin{lemma}[Limiting absorption principle for the Laplace operator]
\label{linear-b-lemma-lap-agmon-old-better}
Let $n\ge 1$.
For any
$s>1/2$,
$\delta>0$,
and $\nu\in\N_0$, $\nu\le 2$,
there is $C=C(n,s,\delta,\nu)<\infty$
such that
\[
\norm{(-\Delta-z)^{-1}}
\sb{L^2\sb s(\R^n)\to H^\nu\sb{-s}(\R^n)}
\le C\abs{z}^{-(1-\nu)/2},
\qquad
\forall z\in\C\setminus(\mathbb{D}\sb\delta\cup\R\sb{+}).
\]
\end{lemma}

For the proof,
see \cite[Theorem A.1 and Remark 2 in Appendix A]{MR0397194}.

\ac{End of new version for JFA}

\ac{A new version for opus:}

We start with the following result:

\begin{lemma}
\label{linear-b-lemma-lap-agmon}
Let $n\ge 1$.
For any
$s>1/2$,
$\delta>0$,
and $\nu\in\N_0$, $\nu\le 2$,
there is $C=C(n,s,\delta,\nu)>0$
such that for any
$u\in H^2(\R^n)$
one has
\[
\norm{u}_{H^\nu\sb{-s}}
\le C\abs{z}^{-(1-\nu)/2}
\norm{(-\Delta-z)u}\sb{L^2\sb s},
\qquad
\forall z\in\C\setminus\mathbb{D}_\delta.
\]
\end{lemma}

This lemma follows from
\cite[Theorem A.1 and Remark 2 in Appendix A]{MR0397194}.
It also implies the limiting absorption principle
for the resolvent of the free Laplace operator
(see e.g. \cite[Theorem 8.1]{MR544248}),
showing that the mapping
\[
R^0(z)=(-\Delta-z)^{-1}:\,L^2_s(\R^n)\to L^2_{-s}(\R^n),
\qquad
s>1/2,
\qquad
z\in\C\setminus\overline{\R_{+}},
\]
is bounded uniformly in $z\in\C\setminus\overline{\R_{+}}$,
$\abs{z}\ge \delta$, for each $\delta>0$,
with the norm improving as $\abs{z}\to\infty$.
\ac{do we want -- ...and admits an extension to $\R_{+}$. ??}

\medskip


\ac{End of new version for opus}

Let $E_\mu:\, L^2(\R^n)\to L^2(\R^n)$
denote the operator of multiplication
by $e^{-\mu\langle r\rangle}$,
$\mu\in\R$.
Following \cite{MR0495958},
not to confuse the regularized resolvent
\[
R^0_\mu(\zeta^2):=E_\mu R^0(\zeta^2) E_\mu
=E_\mu(-\Delta-\zeta^2)^{-1}E_\mu,
\qquad
\zeta\in\C,
\quad
\Im\zeta>0
\]
with its analytic continuation
through the line $\Im\zeta=0$,
we will denote the latter by
$F_\mu^0(\zeta)$.

\ac{Below, do we need $k\ge 1$??
Again, $\nu\ge 0$?}

\begin{proposition}[Analytic continuation of the resolvent]
\label{linear-b-prop-rauch-prop3}
Let $n\ge 1$, $\mu>0$.
\begin{enumerate}
\item
There is an analytic function
$F_\mu^0(\zeta)$,
\[
F_\mu^0:\;
\{\Im\zeta>-\mu\}
\setminus(-\jj\overline{\R\sb{+}})
\longrightarrow
\mathscr{B}(L^2(\R^n),L^2(\R^n)),
\]
such that
$F_\mu^0(\zeta)=R_\mu^0(\zeta^2)$
for $\Im\zeta>0$,
and for any
$\delta>0$ and $\nu\in\N_0$, $\nu\le 2$,
there is
$C=C(n,\mu,\delta,\nu)<\infty$
such that
\begin{eqnarray}\label{linear-b-asdfasdf}
\norm{\p\sb\zeta^k F_\mu^0(\zeta)}\sb{L^2\to H^\nu}
\le \frac{C}{(1+\abs{\zeta})^{k+1-\nu}},
\qquad
\Im\zeta\ge-\mu+\delta,
\quad
\dist(\zeta,-\jj\R_{+})>\delta.
\end{eqnarray}
\item
  If $n$ is odd and satisfies
  $n\ge 3$, then \eqref{linear-b-asdfasdf}
  holds for all $\zeta\in\C$, $\Im\zeta\ge-\mu+\delta$.
\end{enumerate}
\end{proposition}

\begin{remark}
This result in dimension $n=3$ was stated and proved
in \cite[Proposition 3]{MR0495958},
as a consequence of the explicit expression
for the integral kernel of
$R_\mu^0(\zeta^2)$,
\[
-\frac{e^{-\mu\langle y\rangle}
e^{\jj\zeta\abs{y-x}}
e^{-\mu\langle x\rangle}
}{4\pi\abs{y-x}},
\qquad
\Im\zeta>0,
\qquad
x,\,y\in\R^3,
\]
which could be extended analytically
to the region $\Im\zeta>-\mu$
as a holomorphic function of $\zeta$ with values in
$L^2(\R^3\times \R^3)$.
In \cite{MR0495958},
the restriction on $\zeta$ was stronger:
$\Im\zeta>-\mu/2+\delta$ (with any $\delta>0$);
this was a pay-off for using an elegant argument
based on the Huygens principle.
(We note that our signs and inequalities
are often the opposite
to those of \cite{MR0495958}
since we consider the resolvent of $-\Delta$ instead of $\Delta$.)
\end{remark}

\begin{proof}
Let us define the analytic continuation
of $F_\mu^0(\zeta)$.
For $u,\,v\in L^2(\R^n)$
we define
$u\sb\mu$, $v\sb\mu\in L^{2,\mu}(\R^n)$
by
$u\sb\mu(x)=e^{-\mu \langle x\rangle}u(x)$,
$v\sb\mu(x)=e^{-\mu \langle x\rangle}v(x)$
and consider
\begin{eqnarray}\label{linear-b-ti-0}
 I(\zeta)
 =
 \langle v,F_\mu^0(\zeta) u\rangle
=
\int\sb{\R^n}\overline{\widehat{v\sb\mu}(\xi)}
\frac{1}{\xi^2-\zeta^2}\widehat{u\sb\mu}(\xi)
\,\frac{d^n\xi}{(2\pi)^n},
\end{eqnarray}
which is an analytic function in
$\zeta\in\C^{+}:=\{z\in\C:\;\Im z>0\}$.

Let us prove analyticity in
$\zeta$ for
$\Im\zeta>-\mu$,
$\Re\zeta>0$ (the case $\Re\zeta<0$ is considered similarly).
It is enough to prove that
for any $a>0$
and any $\delta>0$, $\delta\le a/3$,
$I(\zeta)$ extends analytically
into the rectangular neighborhood
\begin{eqnarray}\label{linear-b-def-k}
\mathcal{K}\sb{a}\sp{\delta}
=\left\{\zeta\in\C\sothat
a-\delta\le\Re\zeta\le a+\delta,
\
-\mu+\delta\le\Im\zeta\le 0
\right\}
\end{eqnarray}
(see Figure~\ref{linear-b-fig-eta-contour}),
satisfying there the bounds \eqref{linear-b-asdfasdf}
with constants $c\sb j$ independent of $a$.

\begin{figure}[htbp]
\begin{center}
\setlength{\unitlength}{1pt}
\begin{picture}(70,75)(65,-70)
\put(200,4){$\Re \lambda$}
\put(4,8){$\Im \lambda$}
\put(-7,-8){$0$}
\put(-24,-52){$\scriptstyle -\mu+\delta$}
\put(-14,-70){$\scriptstyle -\mu$}
\put(60,0){\circle*{2}}
\put(80,0){\circle*{2}}
\put(100,0){\circle*{2}}
\put(120,0){\circle*{2}}
\put(140,0){\circle*{2}}
\put(94,-23){$\mathcal{K}\sb{a}\sp{\delta}$}
\put(80,-50){\line(1,0){40}}
\put(80,0){\line(0,-1){50}}
\put(120,0){\line(0,-1){50}}
\put(57,-35){$\gamma_a$}
\put(42,-8){$\scriptstyle a-2\delta$}
\put(73,5){$\scriptstyle a-\delta$}
\put(100,5){$\scriptstyle a$}
\put(112,5){$\scriptstyle a+\delta$}
\put(140,-8){$\scriptstyle a+2\delta$}
\put(0,0){\vector(1,0){200}}
\put(0,-70){\vector(0,1){85}}
\multiput(-1,-50)(5,0){28}{\line(1,0){1}}
\multiput(-1,-60)(5,0){28}{\line(1,0){1}}
\multiput(-1,-70)(5,0){29}{\line(1,0){1}}
\multiput(70,0)(0,-5){15}{\line(0,1){1}}
\multiput(80,0)(0,-5){15}{\line(0,1){1}}
\multiput(120,0)(0,-5){15}{\line(0,1){1}}
\multiput(130,0)(0,-5){15}{\line(0,1){1}}
\multiput(140,0)(0,-5){15}{\line(0,1){1}}
\linethickness{1pt}
\qbezier( 60,0)( 66,3)( 70,-53)
\qbezier( 70,-53)(73,-60)(80,-60)
\qbezier(130,-53)(127,-60)(120,-60)
\qbezier(140,0)(134,3)(130,-53)
\put(80,-60){\line(1,0){40}}

\end{picture}
\caption{\footnotesize
The rectangular semi-neighborhood $\mathcal{K}\sb{a}\sp\delta$
around $a$
surrounded by the contour
 $\gamma_a=\left\{\lambda:\,\Im\lambda=g_a(\Re\lambda),
 \ a-2\delta
 \le\Re\lambda
 \le a+2\delta
 \right\}$;
$\ \mathop{\mathrm{dist}}
(\mathcal{K}\sb a\sp\delta,\gamma\sb a)
\ge\delta/2$.
}
\label{linear-b-fig-eta-contour}
\end{center}
\end{figure}

We pick $a>0$ and $\delta>0$, with $a\ge 3\delta$,
and break the integral \eqref{linear-b-ti-0} into two:
\begin{eqnarray}\label{linear-b-ti}
I(\zeta)
=
I_1^{(\delta)}(\zeta)+ I_2^{(\delta)}(\zeta)
=
\int\sb{\abs{\abs{\xi}-a}>2\delta}
\ +\ \int\sb{\abs{\abs{\xi}-a}<2\delta}.
\end{eqnarray}
The first integral
in \eqref{linear-b-ti}
is finite, being bounded by
\[
\abs{I_1^{(\delta)}(\zeta)}
\le
\int\limits\sb{\abs{\abs{\xi}-a}>2\delta}
\abs{\hat v\sb\mu(\xi)}\abs{\hat u\sb\mu(\xi)}
\frac{1}{2\abs{\zeta}}
\left|
\frac{1}{\abs{\xi}-\zeta}
-
\frac{1}{\abs{\xi}+\zeta}
\right|
\,\frac{d^n\xi}{(2\pi)^n}
\]
\[
\le
\int\limits\sb{\R^n}
\frac{\abs{\hat v\sb\mu(\xi)}\abs{\hat u\sb\mu(\xi)}}{2\abs{\zeta}}
\frac{2}{\delta}
\,\frac{d^n\xi}{(2\pi)^n}
\le
\frac{\norm{u\sb\mu}\norm{v\sb\mu}}
{\abs{\zeta}\delta},
\]
and therefore is analytic in $\zeta$
and is bounded by
$C/\abs{\zeta}$.
Above, to estimate the denominators,
we took into account that
for $\zeta\in\mathcal{K}\sb{a}\sp\delta$
and $\abs{\abs{\xi}-a}>2\delta$,
\[
\abs{\abs{\xi}\pm\zeta}
\ge\abs{(\abs{\xi}-a)+(a\pm\Re\zeta)}
\ge\abs{\abs{\xi}-a}-\abs{a\pm\Re\zeta}
>2\delta-\delta=\delta.
\]
To analyze the second integral
in \eqref{linear-b-ti},
we will deform the contour of integration in $\xi$.
Let $g_0\in C\sb{\mathrm{comp}}^\infty(\R)$
be even,
$g_0\le 0$,
$\supp g_0\in[-2\delta,2\delta]$,
with $g_0(0)=-\mu+\delta/2$
and monotonically increasing to zero as $\abs{t}\to 2\delta$.
Moreover,
we may assume that
$\abs{g_0'}<4\mu/\delta$
and that
$\mathop{\mathrm{dist}}(\gamma\sb 0,\mathcal{K}\sb 0\sp\delta)
\ge\delta/2$,
where
$\mathcal{K}\sb a\sp\delta$ is defined in \eqref{linear-b-def-k}
and
$\gamma\sb 0
=\{(\lambda,g_0(\lambda)):\;
\abs{\lambda}\le 2\delta\}$;
see Figure~\ref{linear-b-fig-eta-contour}.
Define $g_a(t)=g_0(t-a)$.

\begin{lemma}\label{linear-b-lemma-a-c}
Let $u\in L^{2,\mu}(\R^n)$,
so that
$\norm{u}\sb{L^{2,\mu}(\R^n)}
:=
\norm{e^{\mu\langle r\rangle }u}\sb{L^2(\R^n)}<\infty$.
Then its Fourier transform, $\hat u(\xi)$,
can be extended analytically into
the $\mu$-neighborhood of $\R^n\subset\C^n$,
which we denote by
\[
\Omega_\mu(\R^n)
=\{
\xi\in\C^n
\sothat
\abs{\Im\xi}<\mu
\}
\subset\C^n,
\]
and
there is $C_{\mu}>0$
such that
\begin{eqnarray}\label{linear-b-u-hat-u}
\norm{\hat u}\sb{L^2(\Omega_{\mu}(\R^n))}
\le C_{\mu}\norm{u}\sb{L^{2,\mu}(\R^n)},
\end{eqnarray}
where
$\Omega_\mu(\R^n)$
is interpreted as a region in
$\R^{2n}\cong\C^n$.
\end{lemma}

\begin{proof}
To prove this lemma,
one
defines $\hat{u}$ on $\Omega_\mu$ as the Fourier--Laplace transform of $u$
and then computes the $L^2$-norm for fixed $\Im\zeta$.
\end{proof}

By Lemma~\ref{linear-b-lemma-a-c},
the functions
$U(\xi)=\widehat{u\sb\mu}(\xi)$
and
$V(\xi)=\overline{\widehat{v\sb\mu}(\xi)}$
could be extended
analytically in $\xi\in\R^n$
into the strip
$\xi\in\C^n$, $\abs{\Im\xi}<\mu$.
We rewrite the second integral in
\eqref{linear-b-ti} in polar coordinates,
denoting $\lambda=\abs{\xi}\in[a-2\delta,a+2\delta]$,
$\bm\theta=\xi/\abs{\xi}\in\mathbb{S}^{n-1}$,
and then deform the contour
of integration in $\lambda$, arriving at
\begin{eqnarray}\label{linear-b-above-expression}
I_2^{(\delta)}(\zeta)
=
\int\sb{\gamma_a\times\mathbb{S}^{n-1}}
\frac{
V(\bm\theta\lambda)U(\bm\theta\lambda)
}{\lambda^2-\zeta^2}\lambda^{n-1}\,d\lambda
\frac{d\Omega\sb{\bm\theta}}{(2\pi)^n},
\end{eqnarray}
with $\gamma_a$ as on Figure~\ref{linear-b-fig-eta-contour}.
Clearly, \eqref{linear-b-above-expression} is analytic for
$\Re\zeta>0$ and $\Im\zeta>0$ (since
$\Im\lambda^2\le 0$ while $\Im\zeta^2>0$).

Let us argue that
\eqref{linear-b-above-expression} can also be
extended analytically into the box
$\mathcal{K}\sb{a}\sp\delta$.
For
$\lambda\in\gamma_a$ and
$\zeta\in\mathcal{K}\sb{a}\sp\delta$,
taking into account that
\[
\abs{\lambda-\zeta}
\ge\delta/2,
\qquad
\abs{\lambda+\zeta}
\ge
\Re\lambda+\Re\zeta
\ge
(a-2\delta)
+
(a-\delta)
=2a-3\delta
\ge a
\]
(recall that $\delta\le a/3$),
we see that \eqref{linear-b-above-expression}
defines an analytic function
which is bounded by
\begin{eqnarray}\label{linear-b-above-expression-2}
&&
\hskip -24pt
\abs{I_2^{(\delta)}(\zeta)}
\\
\nonumber
&&
\hskip -24pt
\le
\frac{2}{a\delta}
\left[
\int
\limits\sb{\gamma_a\times\mathbb{S}^{n-1}}
\hskip -12pt
\abs{V(\bm\theta\lambda)}^2
\abs{\lambda}^{n-1}\,\abs{d\lambda}
\frac{d\Omega\sb{\bm\theta}}{(2\pi)^n}
\int
\limits\sb{\gamma_a\times\mathbb{S}^{n-1}}
\hskip -12pt
\abs{U(\bm\theta\lambda)}^2
\abs{\lambda}^{n-1}\,\abs{d\lambda}
\frac{d\Omega\sb{\bm\theta}}{(2\pi)^n}
\right]^{\frac 1 2},
\end{eqnarray}
where the integration in $\lambda$ over the contour $\gamma_a$
could be parametrized by
\begin{eqnarray}\label{d-lambda}
\lambda(t)=t+\jj g_a(t),
\quad
t\in(a-2\delta,a+2\delta),
\quad
d\lambda=(1+\jj g_a'(t))\,dt,
\end{eqnarray}
so that
$\abs{d\lambda}$ is understood as
$(1+(g_a'(t))^2)^{1/2}\,dt$.
Our assumption that $a\ge 3\delta$
allows us to bound the first factor in \eqref{linear-b-above-expression-2}
by
$
\frac{2}{a\delta}
\le
\frac{2}{3\delta^2}.
$
Moreover,
if $\abs{\zeta}\ge 2(\mu+\delta)$,
the first factor in \eqref{linear-b-above-expression-2} is also bounded by
\[
\frac{2}{a\delta}
<\frac{2}{(\abs{\Re\zeta}-\delta)\delta}
<\frac{2}{(\abs{\zeta}-\mu-\delta)\delta}
<\frac{4}{\abs{\zeta}\delta},
\qquad
\forall\zeta\in
\mathcal{K}\sb{a}\sp\delta
\setminus \mathbb{D}\sb{2(\mu+\delta)}.
\]
Therefore, that factor
is bounded by $c/(1+\abs{\zeta})$
with certain $c=c(\mu,\delta)>0$.
To study the integrals in \eqref{linear-b-above-expression-2},
we define $\bm{G}:\,\R^n\to\R^n$ by
\begin{eqnarray}\label{def-bm-G}
\bm{G}(\eta)
=\frac{\eta}{\abs{\eta}}
g_a(\abs{\eta})
\rho(\abs{\eta}/\delta),
\qquad
\eta\in\R^n,
\end{eqnarray}
where $\rho\in C^\infty(\R)$
satisfies
$\rho(t)\equiv 1$ for $\abs{t}\ge 1$,
$\rho(t)\equiv 0$ for $\abs{t}\le 1/2$,
and parametrize $\xi$ as follows:
\[
\xi=\eta+\jj\bm{G}(\eta)\in\C^n,
\quad
\eta\in\R^n;
\qquad
\abs{\abs{\eta}-a}\le 2\delta.
\]
We have:
\[
\int\limits\sb{\gamma_a\times\mathbb{S}^{n-1}}
\!\!\!\!\!\!
\abs{U(\bm\theta\lambda)}^2\,
\abs{\lambda}^{n-1}\,\abs{d\lambda}
\frac{d\Omega\sb{\bm\theta}}{(2\pi)^n}
\le
\Big(1+\Big(\frac{4\mu}{\delta}\Big)^2\Big)^{\frac n 2}
\!\!\!\!\!\!
\int\limits\sb{\abs{\abs{\eta}-a}<2\delta}
\!\!\!\!\!\!\!\!
\abs{U(\eta+\jj\bm{G}(\eta))}^2
\,d^n\eta,
\]
where we took into account that
both $\abs{\lambda/\Re\lambda}$
and $\abs{d\lambda/\Re d\lambda}$
(which are understood with the aid of the parametrization \eqref{d-lambda}
as
$(t+g(t)^2)^{1/2}/t$ and $(1+(g_a'(t))^2)^{1/2}$,
respectively)
are bounded by
\[
\sqrt{1+(g_0')^2}\le\sqrt{1+(4\mu/\delta)^2}.
\]

One has:
\[
U(\eta+\jj\bm{G}(\eta))
=A_g u(\eta),
\]
where
\[
A_g u(\eta)
:=\int_{\R^n}
e^{-\jj x\cdot\eta}
e^{x\cdot\bm{G}(\eta)}
e^{-\mu\langle x\rangle}u(x)\,d^n x
\]
is an oscillatory integral operator
with the non-degenerate phase function
$\phi(x,\eta)=x\cdot\eta$
and a bounded smooth symbol
$a(x,\eta)=e^{x\cdot\bm{G}(\eta)-\mu\langle x\rangle}$.
Note that, by \eqref{def-bm-G},
\[
x\cdot\bm{G}(\eta)-\mu\langle x\rangle
\le
\abs{x}\abs{g_a(\abs{\eta})}-\mu\langle x\rangle
<0,
\qquad
\forall (x,\eta)\in\R^n\times\R^n.
\]
By the van der Corput-type arguments
applied to $A_g A_g\sp\ast$
\cite[Chapter IX]{MR1232192},
one can show that $A_g$ is continuous in $L^2(\R^n)$.
Then it follows that there is
$c=c(\mu,\delta)>0$
such that
\[
\int\limits\sb{\gamma_a\times\mathbb{S}^{n-1}}
\abs{U(\bm\theta\lambda)}^2
\abs{\lambda}^{n-1}\,\abs{d\lambda}
\frac{d\Omega\sb{\bm\theta}}{(2\pi)^n}
\le
c(\mu,\delta)\norm{u}^2.
\]
There is a similar bound for $V$.
Thus,
there is $C=C(\mu,\delta)>0$
such that
$
\abs{I_2^{(\delta)}(\zeta)}
 \le
 \frac{C(\mu,\delta)}{\abs{\zeta}\delta}
 \norm{v}\norm{u},
$
which is the desired bound.

The estimates on $\p\sb\zeta^j F_\mu^0(\zeta)$, $j\in\N$,
are proved similarly,
writing out the derivatives of $(\xi^2-\zeta^2)^{-1}$
and proceeding with the same decomposition
as in \eqref{linear-b-ti};
the only difference is the contribution from
higher powers of $\xi^2-\zeta^2$ in the denominator.

This settles the first part of Proposition~\ref{linear-b-prop-rauch-prop3}.

\medskip

Before we prove the second part of
Proposition~\ref{linear-b-prop-rauch-prop3},
we need the following technical lemma.

\begin{lemma}\label{linear-b-lemma-f-n}
Let $\rho>0$
and let $N\in\N$ be odd and satisfy $N\ge 3$.
The analytic function
\[
F_{N,\rho}(\zeta)
=
\int\sb{0}\sp{\rho}
\frac{\lambda^{N-1}\,d\lambda}{\lambda^2-\zeta^2},
\qquad
\zeta\in\C,
\qquad
\Im\zeta>0,
\]
extends analytically
into an open disc $\mathbb{D}_\rho$.
Moreover,
one has
\begin{eqnarray}\label{linear-b-bound-on-f}
\abs{F_{N,\rho}(\zeta)}
\le
\frac{\rho^{N-2}}{2}
\left(2+\ln N+\pi
+\ln\frac{\rho+\abs{\zeta}}{\rho-\abs{\zeta}}\right)
,
\qquad
\zeta\in\mathbb{D}_{\rho}.
\end{eqnarray}
\end{lemma}

\begin{proof}
Using the identity
$
\frac{\lambda^2}{\lambda^2-\zeta^2}
=
1+\frac{\zeta^2}{\lambda^2-\zeta^2}
$
(note that the denominator is nonzero
since $\lambda\ge 0$ and $\Im\zeta>0$)
and remembering that $N$ is odd,
we have:
\begin{eqnarray}\label{linear-b-f-n-alpha}
&&
F_{N,\rho}(\zeta)
=
\int\sb{0}\sp{\rho}
\frac{\lambda^{N-1}\,d\lambda}{\lambda^2-\zeta^2}
=
\int\sb{0}\sp{\rho}
\left(
\lambda^{N-3}
+
\zeta^2 \lambda^{N-5}
+
\dots
+
\zeta^{N-3}
+
\frac{\zeta^{N-1}}{\lambda^2-\zeta^2}
\right)\,d\lambda
\nonumber
\\
&&
\qquad
=
\frac{\rho^{N-2}}{N-2}
+
\frac{\zeta^2 \rho^{N-4}}{N-4}
+\dots
+
\zeta^{N-3}\rho
+
\frac{\zeta^{N-2}}{2}
\left[
\Ln\left(\frac{\rho-\zeta}{\rho+\zeta}\right)
+\pi\jj
\right].
\end{eqnarray}
Above, $\Ln$ denotes the analytic branch of the natural logarithm
on $\C\setminus\overline{\R\sb{-}}$
specified by $\Ln(1)=0$;
we also took into  account that, since $\Im\zeta>0$,
\[
\lim\sb{\lambda\to 0+}
\Ln\frac{\lambda-\zeta}{\lambda+\zeta}
=
\lim\sb{\lambda\to 0+}
\Ln\left(-1+\frac{2\lambda}{\zeta}\right)
=
\Ln(-1-0\jj)
=-\pi\jj.
\]
Due to the assumption $N\ge 3$,
the right-hand side of \eqref{linear-b-f-n-alpha}
extends to an analytic function of $\zeta$
as long as $\zeta\in\mathbb{D}_\rho$.
The bound \eqref{linear-b-bound-on-f} immediately follows
from the inequalities
\[
\abs{\zeta}<\rho,
\qquad
1+\frac{1}{3}+\frac{1}{5}+\dots+\frac{1}{N-2}
\le
1+\frac{1}{2}\sum_{j=2}^{N-2}\frac{1}{j}
\le
1+\frac 1 2\ln(N-2)
\]
(with the understanding that
the summation gives no contribution when $N=3$),
and the bound
\[
\Abs{\Ln\left(\frac{\rho-\zeta}{\rho+\zeta}\right)
+\pi\jj}
\le
\pi+
\ln\frac{\rho+\abs{\zeta}}{\rho-\abs{\zeta}}
\]
valid for
$\zeta\in\mathbb{D}_\rho\cap\C_{+}$
(when $\arg\frac{\rho-\zeta}{\rho+\zeta}\in(-\pi,0)$).
\end{proof}

\begin{remark}
 Note that the conclusion
 of the lemma
would not hold if $N$ were even:
in that case, one arrives
at functions which have a branching point
at $\zeta=0$; e.g.
\[
\int\sb{0}\sp{\rho}
\frac{\lambda\,d\lambda}{\lambda^2-\zeta^2}
=\frac{1}{2}\ln\left(1-\frac{\rho^2}{\zeta^2}\right),
\]
\[
\int\sb{0}\sp{\rho}
\frac{\lambda^3\,d\lambda}{\lambda^2-\zeta^2}
=
\int\sb{0}\sp{\rho}
\left(
\lambda+
\frac{\zeta^2 \lambda}{\lambda^2-\zeta^2}
\right)\,d\lambda
=
\frac{\rho^2}{2}
+
\frac{\zeta^2}{2}\ln\left(1-\frac{\rho^2}{\zeta^2}\right),
\]
which behave like
$\ln\left(-\frac{\rho}{\zeta}\right)$
and $\zeta^2\ln\left(-\frac{\rho}{\zeta}\right)$
when $\abs{\zeta}\ll\rho$
(hence have a branching point at $\zeta=0$).
\qedhere
\end{remark}

Now let us prove the second part of
Proposition~\ref{linear-b-prop-rauch-prop3};
from now on, we assume that $n$ is odd and satisfies $n\ge 3$.
It is enough to prove that
the function $I(\zeta)$ defined in \eqref{linear-b-ti-0}
is analytic inside the disc
$\mathbb{D}\sb{\mu}\subset\C$.

We pick $\rho\in(0,\mu)$
and break the integral \eqref{linear-b-ti-0} into two parts:
\begin{eqnarray}\label{linear-b-two-int}
I(\zeta)
=
 \int\sb{\R^n}
\frac{V(\xi)U(\xi)}{\xi^2-\zeta^2}\,d^n\xi
=
I_1^{(\rho)}(\zeta)+ I_2^{(\rho)}(\zeta),
\end{eqnarray}
where
\begin{eqnarray}
I_1^{(\rho)}(\zeta)
&:=&
\int\sb{\abs{\xi}\le\rho}
\frac{V(\xi)U(\xi)}{\xi^2-\zeta^2}\,d^n\xi,
\label{linear-b-two-int-1}
\\
I_2^{(\rho)}(\zeta)
&:=&
\int\sb{\abs{\xi}>\rho}
\frac{V(\xi)U(\xi)}{\xi^2-\zeta^2}\,d^n\xi.
\label{linear-b-two-int-2}
\end{eqnarray}
The function
$I_2^{(\rho)}(\zeta)$
in \eqref{linear-b-two-int-2}
is analytic in the disc $\zeta\in\mathbb{D}_\rho$,
and moreover
for any $r\in(0,\rho)$
one has
\[
\sup\sb{\zeta\in\mathbb{D}_r}
\abs{I_2^{(\rho)}(\zeta)}
\le
\sup\sb{\zeta\in\mathbb{D}_r}
\Abs{
\int\sb{\abs{\xi}>\rho}
\frac{V(\xi)U(\xi)}{\xi^2-\zeta^2}\,d^n\xi
}
\le
\frac{1}{\rho^2-r^2}
\norm{v}_{L^2}\norm{u}_{L^2}.
\]
Let us consider $I_1^{(\rho)}(\zeta)$
defined in \eqref{linear-b-two-int-1}.
Since both $V(\xi)$ and $U(\xi)$ are analytic
for $\xi\in\C^n$, $\abs{\xi}<\mu$,
we have the power series expansions
\[
V(\xi)
=\sum\sb{\alpha\in\N_{0}^n}
V_\alpha\xi^\alpha,
\qquad
U(\xi)
=\sum\sb{\alpha\in\N_{0}^n}
U_\alpha\xi^\alpha,
\qquad
V(\xi)U(\xi)
=\sum\sb{\alpha\in\N_{0}^n}
C_\alpha\xi^\alpha,
\]
which are absolutely convergent for $\abs{\xi}<\mu$.
Denote $\lambda=\abs{\xi}$,
$\bm\theta=\xi/\abs{\xi}\in\mathbb{S}^{n-1}$.
Then
\begin{eqnarray}\label{linear-b-i-1-alpha}
I_1\sp{(\rho)}(\zeta)
=
\int\limits\sb{\abs{\xi}\le\rho}
\frac{V(\xi)U(\xi)}{\xi^2-\zeta^2}\,d^n\xi
=
\sum\sb{\alpha\in\N_0^n}
C_\alpha
\int\limits\sb{\mathbb{S}^{n-1}}\bm\theta^\alpha\,d\Omega_{\bm\theta}
\int_0^\rho
\frac{\lambda^{\abs{\alpha}+n-1}\,d\lambda}{\lambda^2
-\zeta^2},
\end{eqnarray}
where $\bm\theta^\alpha=\theta_1^{\alpha_1}\dots\theta_n^{\alpha_n}$,
$\bm\theta=(\theta_1,\dots,\theta_n)\in\mathbb{S}^{n-1}$.
We note that, by parity considerations,
the terms corresponding to at least one $\alpha\sb j$
being odd are equal to zero,
hence
the summation in the right-hand side
is only over $\alpha\in(2\N_0)^n$.
We claim that the series \eqref{linear-b-i-1-alpha} defines an analytic function
in $\mathbb{D}_\rho$,
and moreover for each $r\in(0,\rho)$
there is $C>0$
such that
\[
\sup\sb{\zeta\in\mathbb{D}_r}
\abs{I_1\sp{(\rho)}(\zeta)}
\le
C\norm{v}_{L^2(\R^n)}\norm{u}_{L^2(\R^n)}.
\]
We have:
\begin{eqnarray}\label{linear-b-i-1-alpha-1}
I_1\sp{(\rho)}(\zeta)
&=&
\sum\sb{
\alpha\in(2\N_0)^n
}
C_\alpha
\int_{\mathbb{S}^{n-1}}\bm\theta^\alpha\,d\Omega_{\bm\theta}
F_{\abs{\alpha}+n,\rho}(\zeta
)
\nonumber
\\
&=&
\sum\sb{
\alpha\in(2\N_0)^n
}
C_\alpha
\int_{\mathbb{S}^{n-1}}\bm\theta^\alpha\,d\Omega_{\bm\theta}
R^{\abs{\alpha}}
\frac{F_{\abs{\alpha}+n,\rho}(\zeta
)}
{
R^{\abs{\alpha}}},
\end{eqnarray}
where
$R\in(\rho,\mu)$
and the analytic function $F_{N,\rho}(\zeta)$ was defined in Lemma~\ref{linear-b-lemma-f-n}.
By that lemma,
\[
\Abs{
\frac{(\abs{\alpha}+n)
F_{\abs{\alpha}+n,\rho}(\zeta)}{R^{\abs{\alpha}}}
}
\le
(\abs{\alpha}+n)
\frac{\rho^{n+\abs{\alpha}-2}}{2R^{\abs{\alpha}}}
\left(
2+\ln(\abs{\alpha}+n)+\pi+\ln\frac{\rho+\abs{\zeta}}{\rho-\abs{\zeta}}
\right)
\]
are analytic functions of $\zeta\in\mathbb{D}_r$,
$r\in(0,\rho)$,
which are bounded uniformly in $\alpha\in\N_0^n$
and $\zeta\in\mathbb{D}_r$,
by some $c\sb{r,\rho,R}>0$,
$0<r<\rho<R<\mu$.
Using this bound in \eqref{linear-b-i-1-alpha-1},
one has:
\begin{eqnarray}\label{linear-b-abs-i-1}
\abs{I_1^{(\rho)}(\zeta)}
\le
c\sb{r,\rho,R}
\sum
\sb{
\alpha\in(2\N_0)^n
}
\int_{\mathbb{S}^{n-1}}
\abs{C_\alpha\bm\theta^\alpha R^{\abs{\alpha}}}
\,d\Omega_{\bm\theta},
\qquad
\zeta\in\mathbb{D}_r.
\end{eqnarray}
Now we can argue that the series
\eqref{linear-b-i-1-alpha-1} is absolutely convergent.
To bound the right-hand side in \eqref{linear-b-abs-i-1},
we use the following lemma
with $\xi=R\bm\theta$.

\begin{lemma}\label{lemma-c-e}
For any $0<R<\mu$
there is $C_{R,\mu}>0$ such that
for any analytic function
$a(\xi)=\sum\sb{\alpha\in\N_0^n}a_\alpha \xi^\alpha$,
$\xi\in\mathbb{D}_\mu^n\subset\C^n$,
which has finite norm in $L^1(\mathbb{B}_{\mu}^{2n})$,
where
$\mathbb{B}_{\mu}^{2n}\subset\R^{2n}$
is identified with $\mathbb{D}_{\mu}^n\subset\C^n$,
one has
\[
\sup\sb{\xi\in\mathbb{D}_R^n}
\sum\sb{\alpha\in\N_0^n}\abs{a_\alpha\xi^\alpha}
\le C_{R,\mu}\norm{a}\sb{L^1(\mathbb{B}_{\mu}^{2n})}.
\]
\end{lemma}

\begin{proof}
One can prove Lemma~\ref{lemma-c-e}
using the Cauchy estimates
\[
\abs{a_\alpha}
\le
\int\limits\sb{\mathbb{S}^1_M
\times\dots\times
\mathbb{S}^1_M}
\frac{\abs{a(M e^{\jj\varphi_1},\dots,M e^{\jj\varphi_n})}}
{M^{\abs{\alpha}}}
\,\frac{d\varphi_1}{2\pi}\dots\,\frac{d\varphi_n}{2\pi},
\]
with
$M\in(R,\mu)$,
and changing each of the integrations
from $\varphi_j$
over $\mathbb{S}^1_M$ to
the integration
over a thin closed annulus
included in the region $R<\abs{z_j}<\mu$.
\end{proof}

This lemma,
together with the estimate
\eqref{linear-b-u-hat-u}
from Lemma~\ref{linear-b-lemma-a-c},
shows that,
for $\zeta\in\mathbb{D}_r$,
\eqref{linear-b-abs-i-1} is bounded by
\begin{eqnarray*}
\abs{I_1^{(\rho)}(\zeta)}
&\le&
c\sb{r,\rho,R}
\mathrm{vol}\,(\mathbb{S}^{n-1})
\sup\limits\sb{\bm\theta\in\mathbb{S}^{n-1}}
\sum\limits\sb{
\alpha\in(2\N_0)^n
}
\abs{C_\alpha\bm\theta^\alpha R^{\abs{\alpha}}}
\\
&\le&
c\sb{r,\rho,R}
C_{R,\mu}
\mathrm{vol}\,(\mathbb{S}^{n-1})
\norm{V U}\sb{L^1(\mathbb{B}_{\mu}^{2n})}
\\
&\le&
c\sb{r,\rho,R}
C_{R,\mu}
\mathrm{vol}\,(\mathbb{S}^{n-1})
\norm{V}\sb{L^2(\mathbb{B}_{\mu}^{2n})}
\norm{U}\sb{L^2(\mathbb{B}_{\mu}^{2n})}
\\
&\le&
c\sb{r,\rho,R}
C_{R,\mu}
C_{\mu}^2
\mathrm{vol}\,(\mathbb{S}^{n-1})
\norm{v}\sb{L^{2,\mu}(\R^n)}
\norm{u}\sb{L^{2,\mu}(\R^n)}
,
\end{eqnarray*}
where $V(\xi)$ and $U(\xi)$,
$\xi\in \Omega_\mu(\R^n)\subset\C^n$,
denote the analytic continuations of
$\hat v(\xi)$ and $\hat u(\xi)$, $\xi\in\R^n$,
into the $\mu$-neighborhood of $\R^n$ in $\C^n$.
We conclude that the series \eqref{linear-b-i-1-alpha}
is absolutely convergent
and therefore defines an analytic function.

Thus,
$I_1^{(\rho)}(\zeta)$
(and hence $I(\zeta)$ in \eqref{linear-b-two-int})
has an analytic continuation
into the disc
$\mathbb{D}_\rho$
for arbitrary $\rho\in(0,\mu)$,
and for any $r\in(0,\rho)$
the function $I_1^{(\rho)}(\zeta)$
(and hence $I(\zeta)$)
is bounded by $C(r)\norm{v}\norm{u}$
as long as $\zeta\in\mathbb{D}_r$.
This concludes the proof of
Proposition~\ref{linear-b-prop-rauch-prop3}.
\end{proof}

\section{Appendix: Spectrum of the linearized nonlinear Schr\"odinger equation}
\label{linear-b-sect-nls}

For the nonlinear Schr\"odinger equation
and several similar models,
real eigenvalues could only emerge from the origin,
and this emergence is controlled
by the Kolokolov stability condition~\cite{kolokolov-1973}.
Let us give the essence of the linear stability analysis
on the example of the (generalized)
nonlinear Schr\"odinger equation,
\[
\jj \p\sb t\psi=-\frac{1}{2m}\Delta\psi-f(\abs{\psi}\sp 2)\psi,
\qquad
\psi(t,x)\in\C,
\qquad
x\in\R\sp n,
\qquad
n\ge 1,
\qquad
t\in\R,
\]
where the nonlinearity satisfies $f\in C\sp\infty(\R)$, $f(0)=0$.
One can easily construct solitary wave solutions
$\phi(x)e\sp{-\jj \omega t}$,
for some $\omega\in\R$ and $\phi\in H\sp 1(\R\sp n)$:
$\phi(x)$ satisfies the stationary equation
$\omega\phi=-\frac{1}{2m}\Delta\phi-f(\phi\sp 2)\phi$,
and can be chosen strictly positive,
even, and monotonically decaying away from $x=0$.
The value of $\omega$ can not exceed $0$;
we will only consider the case $\omega<0$.
We use the Ansatz
$\psi(t,x)=(\phi(x)+\rho(t,x))e\sp{-\jj \omega t}$,
with $\rho(t,x)\in\C$.
The linearized equation on $\rho$
is called the linearization at a solitary wave:
\begin{equation}\label{linear-b-dv-nls-lin-0}
\p\sb t
\rho
=
\frac{1}{ \jj }
\big(-\frac{1}{2m}\Delta\rho
-\omega\rho
-f(\phi\sp 2)\rho
-2f'(\phi\sp 2)\phi\sp 2\Re\rho
\big).
\end{equation}

\begin{remark}\label{linear-b-remark-not-c-linear}
Because of the term with $\Re\rho$,
the operator in the right-hand side
is $\R$-linear but not $\C$-linear.
\end{remark}

To study the spectrum of the operator
in the right-hand side of \eqref{linear-b-dv-nls-lin-0},
we first write it in the
$\C$-linear form,
considering its
action onto
$\bmuprho(t,x)
=\left[\begin{matrix}\Re\rho(t,x)\\\Im\rho(t,x)\end{matrix}\right]
$:
\[
\p\sb t
\bmuprho
=
\begin{bmatrix}0&\eurl\sb{-}\\-\eurl\sb{+}&0\end{bmatrix}
\bmuprho,
\qquad
\bmuprho(t,x)=\left[\begin{matrix}\Re\rho(t,x)\\\Im\rho(t,x)\end{matrix}\right],
\]
with
\begin{equation}\label{def-l0-l1}
\eurl\sb{-}=-\frac{1}{2m}\Delta-\omega-f(\phi\sp 2),
\qquad
\eurl\sb{+}=\eurl\sb{-}-2\phi\sp 2 f'(\phi\sp 2).
\end{equation}
If $\phi\in\mathscr{S}(\R\sp n)$,
then by Weyl's theorem on the essential spectrum
one has
\[
\sigma\sb{\mathrm{ess}}(\eurl\sb{-})=\sigma\sb{\mathrm{ess}}(\eurl\sb{+})
=[\abs{\omega},+\infty).
\]

\begin{lemma}\label{linear-b-lemma-r-i}
$\sigma\Big(
\begin{bmatrix}0&\eurl\sb{-}\\-\eurl\sb{+}&0\end{bmatrix}
\Big)\subset\R\cup \jj \R$.
\end{lemma}

\begin{proof}
We consider
$
\begin{bmatrix}0&\eurl\sb{-}\\-\eurl\sb{+}&0\end{bmatrix}^2
=
-\left[\begin{matrix}
\eurl\sb{-}\eurl\sb{+}&0\\
0&\eurl\sb{+}\eurl\sb{-}
\end{matrix}\right].
$
Since $\eurl\sb{-}$ is positive-definite
($\phi\in\ker(\eurl\sb{-})$,
being nowhere zero,
corresponds to the smallest eigenvalue),
we can define the self-adjoint square root of $\eurl\sb{-}$;
then
\[
\sigma\sb{\mathrm{d}}
\Big(
\begin{bmatrix}0&\eurl\sb{-}\\-\eurl\sb{+}&0\end{bmatrix}^2
\Big)
\backslash\{0\}
=\sigma\sb{\mathrm{d}}(\eurl\sb{-}\eurl\sb{+})
\backslash\{0\}
=\sigma\sb{\mathrm{d}}(\eurl\sb{+}\eurl\sb{-})
\backslash\{0\}
=\sigma\sb{\mathrm{d}}(\eurl\sb{-}\sp{1/2}\eurl\sb{+}\eurl\sb{-}\sp{1/2})
\backslash\{0\}
\subset\R,
\]
with the inclusion
due to
$\eurl\sb{-}\sp{1/2}\eurl\sb{+}\eurl\sb{-}\sp{1/2}$
being self-adjoint.
Thus,
any eigenvalue $\lambda\in\sigma\sb{\mathrm{d}}
\Big(
\begin{bmatrix}0&\eurl\sb{-}\\-\eurl\sb{+}&0\end{bmatrix}
\Big)
$
satisfies
$\lambda\sp 2\in\R$.
\end{proof}

Given the family of solitary waves,
$\phi\sb\omega(x)e\sp{-\jj \omega t}$,
$\omega\in\mathcal{O}\subset\R$,
we would like to know
at which $\omega$
the eigenvalues
of the linearized equation
with $\Re\lambda>0$ appear.
Since $\lambda\sp 2\in\R$,
such eigenvalues can only be located
on the real axis,
having bifurcated from $\lambda=0$.
One can check that
$\lambda=0$ belongs to the discrete spectrum of
$\begin{bmatrix}0&\eurl\sb{-}\\-\eurl\sb{+}&0\end{bmatrix}$,
with
\[
\begin{bmatrix}0&\eurl\sb{-}\\-\eurl\sb{+}&0\end{bmatrix}
\left[\begin{matrix}0\\\phi\sb\omega\end{matrix}\right]=0,
\qquad
\begin{bmatrix}0&\eurl\sb{-}\\-\eurl\sb{+}&0\end{bmatrix}
\left[\begin{matrix}-\p\sb\omega\phi\sb\omega\\0\end{matrix}\right
]
=\left[\begin{matrix}0\\\phi\sb\omega\end{matrix}\right],
\]
for all $\omega$ which correspond to solitary waves.
Thus, if we will restrict our attention to functions which are
spherically symmetric in $x$,
the dimension of the generalized null space of
$\begin{bmatrix}0&\eurl\sb{-}\\-\eurl\sb{+}&0\end{bmatrix}$
is at least two.
Hence, the bifurcation follows the jump in the
dimension of its generalized null space.
Such a jump happens
at a particular value of $\omega$
if one can solve the equation
$\begin{bmatrix}0&\eurl\sb{-}\\-\eurl\sb{+}&0\end{bmatrix}
\upalpha=\left[\begin{matrix}\p\sb\omega\phi\sb\omega\\0
\end{matrix}\right]$.
This leads to the condition
that $\left[\begin{matrix}\p\sb\omega\phi\sb\omega\\0\end{matrix}\right]$
is orthogonal to the null space of the adjoint to
$\begin{bmatrix}0&\eurl\sb{-}\\-\eurl\sb{+}&0\end{bmatrix}$,
which contains the vector
$\left[\begin{matrix}\phi\sb\omega\\0\end{matrix}\right]$;
this results in
$\langle\phi\sb\omega,\p\sb\omega\phi\sb\omega\rangle
=\p\sb\omega\norm{\phi\sb\omega}\sb{L\sp 2}\sp 2/2=0$.
A slightly more careful analysis~\cite{MR1995870}
based on construction of the moving frame
in the generalized eigenspace of $\lambda=0$
shows that there are two real eigenvalues $\pm\lambda\in\R$
that have emerged from $\lambda=0$
when $\omega$ is such that
$\p\sb\omega\norm{\phi\sb\omega}\sb{L\sp 2}\sp 2$ becomes positive,
leading to a linear instability of the corresponding solitary wave.
The opposite condition,
$
\p\sb\omega\norm{\phi\sb\omega}\sb{L\sp 2}\sp 2<0,
$
is the Kolokolov stability criterion
which guarantees the absence of nonzero real eigenvalues
for the nonlinear Schr\"odinger equation.
It appeared in~\cite{kolokolov-1973,MR677997,MR723756,MR820338,MR901236,MR1221351}
in relation to linear and orbital stability
of solitary waves.

For the applications to the nonrelativistic limit
of the nonlinear Dirac equation,
we need to consider the linearization of the
nonlinear Schr\"odinger equation with pure power nonlinearity:
$f(\tau)=\abs{\tau}^\kappa$, $\kappa>0$:
\[
\jj \dot\psi=-\frac{1}{2m}\Delta\psi-\abs{\psi}\sp{2\kappa}\psi,
\qquad
\psi(t,x)\in\C,
\quad
x\in\R\sp n.
\]
We need the detailed knowledge of the spectrum
of the linearization operator
$
\begin{bmatrix}0&\eurl\sb{-}\\-\eurl\sb{+}&0\end{bmatrix}
$
corresponding to the linearization the solitary wave
$u_\kappa(x)e^{-\jj\omega t}$,
with
$u_\kappa$ a strictly positive
spherically symmetric solution to \eqref{dirac-existence-def-uu}
and $\omega=-\frac{1}{2m}$.

\begin{lemma}\label{linear-b-lemma-dim-ker}
The dimension of
the null space and
the generalized null space
of the linearization operator
$
\begin{bmatrix}0&\eurl\sb{-}\\-\eurl\sb{+}&0\end{bmatrix}
$
from \eqref{linear-b-lin-nls}
is given by
\[
\ker\Big(
\begin{bmatrix}0&\eurl\sb{-}\\-\eurl\sb{+}&0\end{bmatrix}
\Big)
=n+1,
\qquad\qquad
\frakL
\Big(
\begin{bmatrix}0&\eurl\sb{-}\\-\eurl\sb{+}&0\end{bmatrix}
\Big)
=
\begin{cases}
2n+2,
\qquad
&
\kappa\ne 2/n;
\\
2n+4,
\qquad
&
\kappa=2/n.
\end{cases}
\]
\end{lemma}

\begin{proof}
Such computations
have appeared in many articles.
The relation \eqref{dirac-existence-def-uu}
shows that
$\eurl\sb{-}u_\kappa=0$.
Taking the derivatives of this relation
with respect to $x\sp i$,
one also gets
$\eurl\sb{+}\p_{x^i} u_\kappa=0$,
$1\le i\le n$.
    From~\cite{kwong1989} or~\cite[Lemma 2.1]{chang2007}
we have that
$\dim\ker\Big(
\begin{bmatrix}\eurl\sb{+}&0\\0&\eurl\sb{-}\end{bmatrix}
\Big)=n+1$,
hence there are no other vectors in the kernel of
$
\begin{bmatrix}\eurl\sb{+}&0\\0&\eurl\sb{-}\end{bmatrix}
$.

Now let us study the generalized eigenvectors.
The relation
$\eurl\sb{-}u_\kappa=0$ leads to
\[
\eurl\sb{-}(x\sp i u_\kappa)=-\frac{1}{m}\p_{x^i} u_\kappa,
\quad
1\le i\le n.
\]
This shows that
\begin{equation}\label{linear-b-ker-2n}
\begin{bmatrix}0&\eurl\sb{-}\\-\eurl\sb{+}&0\end{bmatrix}
\begin{bmatrix}\p_{x^i} u_\kappa\\0\end{bmatrix}
=0,
\qquad
\begin{bmatrix}0&\eurl\sb{-}\\-\eurl\sb{+}&0\end{bmatrix}
\begin{bmatrix}
0\\ x\sp i u_\kappa
\end{bmatrix}
=-\frac 1 m\begin{bmatrix}\p_{x^i} u_\kappa\\0\end{bmatrix},
\qquad
1\le i\le n.
\end{equation}
We can not extend this sequence:
there is no $v$ such that
\[
\begin{bmatrix}0&\eurl\sb{-}\\-\eurl\sb{+}&0\end{bmatrix}
\begin{bmatrix}
v\\0
\end{bmatrix}
=
\begin{bmatrix}
0\\ x\sp i u_\kappa
\end{bmatrix},
\]
since
$x\sp i u_\kappa$ is not orthogonal
to the kernel of $\eurl\sb{+}$.
Indeed, as follows from the identity
\begin{equation}\label{linear-b-xd-uu}
\langle
x\sp i u_\kappa, \p_{x^i} u_\kappa\rangle
=\langle(-u_\kappa-x\sp i \p_{x^i} u_\kappa), u_\kappa\rangle
\end{equation}
(no summation in $i$),
one has
$
\langle
x\sp i u_\kappa, \p_{x^i} u_\kappa\rangle
=
-\frac 1 2\langle u_\kappa, u_\kappa\rangle<0.
$

By \eqref{dirac-existence-def-uu},
the function
$u\sb{\kappa,\lambda}(x)=\lambda\sp{1/\kappa} u_\kappa(\lambda x)$
satisfies the identity
\[
-\frac{\lambda\sp{2}}{2m}u\sb{\kappa,\lambda}
=-\frac{1}{2m}\Delta u\sb{\kappa,\lambda}
-u\sb{\kappa,\lambda}\sp{1+2\kappa}.
\]
Differentiating this identity
with respect to $\lambda$ at $\lambda=1$
yields
\begin{equation}\label{linear-b-weird}
0
=\eurl\sb{+}(\p\sb\lambda\at{\lambda=1}u\sb{\kappa,\lambda})+\frac{1}{m}u_\kappa
=\eurl\sb{+}
\Big(\frac{1}{\kappa}u_\kappa+x\cdot\nabla u_\kappa\Big)+\frac{1}{m}u_\kappa.
\end{equation}
This shows that
\begin{equation}\label{linear-b-ker-2}
\begin{bmatrix}0&\eurl\sb{-}\\-\eurl\sb{+}&0\end{bmatrix}
\begin{bmatrix}0\\u_\kappa\end{bmatrix}
=0,
\qquad
\begin{bmatrix}0&\eurl\sb{-}\\-\eurl\sb{+}&0\end{bmatrix}
\begin{bmatrix}
-\theta
\\0\end{bmatrix}
=
\begin{bmatrix}0\\u_\kappa\end{bmatrix},
\end{equation}
with
\begin{eqnarray}\label{linear-b-def-theta}
\theta=
-\frac{m}{\kappa}u_\kappa-m x\cdot\nabla u_\kappa,
\qquad
\eurl\sb{+}\theta=u_\kappa.
\end{eqnarray}
The relations \eqref{linear-b-ker-2n} and \eqref{linear-b-ker-2}
show that
$\dim\frakL\Big(
\begin{bmatrix}0&\eurl\sb{-}\\-\eurl\sb{+}&0\end{bmatrix}
\Big)\ge 2n+2$.
The dimension jumps above $2n+2$
in the case when one can find $\alpha$
such that
\begin{eqnarray}\label{linear-b-def-alpha}
\eurl\sb{-}\alpha=\theta,
\qquad
\mbox{hence}
\quad
\begin{bmatrix}0&\eurl\sb{-}\\-\eurl\sb{+}&0\end{bmatrix}
\begin{bmatrix}
0\\\alpha
\end{bmatrix}
=
\begin{bmatrix}
\theta
\\
0
\end{bmatrix}.
\end{eqnarray}
This happens
when
$\theta$ in \eqref{linear-b-def-theta} is orthogonal
to $\ker(\eurl\sb{-})=\Span\{u_\kappa\}$.
Using the identity \eqref{linear-b-xd-uu},
we see that
\begin{equation}\label{linear-b-K-n}
\frac{1}{m}
\langle
\theta,u_\kappa
\rangle
=
\Big\langle
-\frac 1 \kappa u_\kappa-x\cdot\nabla u_\kappa,u_\kappa
\Big\rangle
=
-\frac 1 \kappa
\langle
u_\kappa,u_\kappa
\rangle
+\frac n 2
\langle
u_\kappa,u_\kappa
\rangle.
\end{equation}
The right-hand side of \eqref{linear-b-K-n}
vanishes when $\kappa=2/n$
(that is, when the nonlinear Schr\"odinger equation
is charge-critical).
We may choose
$\alpha$ to be spherically symmetric
so that it is orthogonal to
$\ker(\eurl\sb{+})
=\Span\big\{\p_{x^i} u_\kappa\sothat 1\le i\le n\big\}$
(as the matter of fact, one can take
$\alpha(x)=-\frac{m}{2}x^2 u_\kappa(x)$);
then, by the Fredholm alternative,
there is $\beta\in L\sp 2(\R\sp n)$ such that
\begin{eqnarray}\label{linear-b-def-beta}
\eurl\sb{+}\beta=\alpha,
\qquad
\mbox{hence}
\quad
\begin{bmatrix}0&\eurl\sb{-}\\-\eurl\sb{+}&0\end{bmatrix}
\begin{bmatrix}-\beta\\0\end{bmatrix}
=
\begin{bmatrix}0\\\alpha\end{bmatrix}.
\end{eqnarray}
(Let us also point out that $\alpha$ and $\beta$ can be chosen real-valued.)
This process can not be continued:
there is no $\gamma\in L\sp 2(\R\sp n)$ such that $\eurl\sb{-}\gamma=\beta$
since
$\beta$ is never orthogonal to $\ker(\eurl\sb{-})$;
indeed, due to semi-positivity of $\eurl\sb{-}$, one has
\[
\langle \beta,u_\kappa\rangle
=\langle \beta,\eurl\sb{+}\theta\rangle
=\langle \beta,\eurl\sb{+}\eurl\sb{-}\alpha\rangle
=\langle\alpha,\eurl\sb{-}\alpha\rangle>0.
\qedhere
\]
\end{proof}

We also need the following technical result.

\begin{lemma}\label{linear-b-lemma-lm-inverse-h1}
For $z\in\rho(\eurl\sb{-})$,
the operator
$(\eurl\sb{-}-z)^{-1}:\;L^2(\R^n)\to H^2(\R^n)$
extends to a continuous mapping
\[
(\eurl\sb{-}-z)^{-1}:\;H^{-1}(\R^n)\to H^1(\R^n).
\]
\end{lemma}

\begin{proof}
Set $a=\sup\sb{x\in\R^n}u_\kappa(x)^{2\kappa}$.
Then
there is $C<\infty$ such that
for any $\varphi\in C^\infty\sb{\mathrm{comp}}(\R^n)$
\[
C
\norm{\varphi}\sb{H^1}^2
\ge
\abs{\langle\varphi,(\eurl\sb{-}+a)\varphi\rangle}
\ge
\Big\langle
\varphi,\Big(-\frac{1}{2m}\Delta+\frac{1}{2m}\Big)\varphi
\Big\rangle
=\frac{1}{2m}\norm{\varphi}\sb{H^1}^2,
\qquad
\forall \varphi\in C^\infty\sb{\mathrm{comp}}(\R^n),
\]
hence the self-adjoint operator
\begin{eqnarray}\label{linear-b-l-h2-l2}
\eurl\sb{-}+a:\;H^2(\R^n)\to L^2(\R^n)
\end{eqnarray}
is positive-definite and invertible.
We can extract its square root,
which is a positive-definite bounded invertible operator
\[
(\eurl\sb{-}+a)^{1/2}:\;H^1(\R^n)\to L^2(\R^n);
\]
then
\eqref{linear-b-l-h2-l2}
also defines a bounded invertible operator
$
(\eurl\sb{-}+a)^{1/2}:\;H^2(\R^n)\to H^1(\R^n),
$
and by duality there is also a bounded invertible mapping
$(\eurl\sb{-}+a)^{1/2}:\;L^2(\R^n)\to H^{-1}(\R^n)$.
We fix $z\in\rho(\eurl\sb{-})$;
then the mapping
\[
(\eurl\sb{-}+a)^{1/2}
(\eurl\sb{-}-z)
(\eurl\sb{-}+a)^{-1/2}:\;H^1(\R^n)\to H^{-1}(\R^n)
\]
is bounded and invertible.
Since
$\eurl\sb{-}+a$ and its square root
commute with $\eurl\sb{-}-z$
(when restricted e.g. to the space of Schwartz functions),
we apply the density argument to conclude that
$\eurl\sb{-}-z$ extends to a bounded invertible mapping
$H^1(\R^n)\to H^{-1}(\R^n)$.
\end{proof}

\section{The limiting absorption principle for ${\eurl\sb{-}}$}

\label{sect-lap-lm}

We will need the estimates on the resolvent of
$\eurl\sb{-}=\frac{1}{2m}-\frac{\Delta}{2m}-u_\kappa^{2\kappa}$
from \eqref{linear-b-def-l-small-pm}
in the spaces with exponential weights.

\begin{lemma}\label{linear-b-lemma-j-n}
Let $\varOmega\subset\C$ be a compact set.
Assume that
\[
\varOmega\cap\sigma\sb{\mathrm{d}}(\eurl\sb{-})=\emptyset.
\]
If $z=1/(2m)$ is either an eigenvalue or a virtual level of $\eurl\sb{-}$,
assume further that
\[
1/(2m)\not\in\varOmega.
\]
Then there is $C=C(\varOmega)>0$
such that the map
\[
u_\kappa^\kappa\circ (\eurl\sb{-}-z)^{-1}\circ u_\kappa^\kappa:\;L^2(\R^n)\to H^2(\R^n),
\qquad
z\in\C\setminus
\sigma(\eurl\sb{-})
\]
satisfies
\[
\norm{u_\kappa^\kappa\circ (\eurl\sb{-}-z)^{-1}\circ u_\kappa^\kappa}
_{L^2\to H^2}\le C,
\qquad
\forall z\in\varOmega\setminus[1/(2m),+\infty).
\]
\end{lemma}

\begin{remark}
The $L^2_s\to L^2_{-s}$ estimate on the resolvent
improves as the spectral parameter goes to infinity,
while $L^2_s\to H^2_{-s}$ estimate does not necessarily so
(see e.g. Lemma~\ref{linear-b-lemma-lap-agmon}
with different values of $\nu$);
this is why in the above lemma we need to consider the spectral parameter
restricted to a compact set $\varOmega\subset\C$.
\end{remark}

\begin{proof}
The first part
(bounds away from an open neighborhood of
the threshold $z=1/(2m)$)
follows from \cite[Appendix A]{MR0397194}.

If the threshold $z=1/(2m)$
is a regular point of the essential spectrum of $\eurl\sb{-}$
(neither an eigenvalue nor a virtual level),
then the result follows
from \cite{MR544248} for $n=3$
and from \cite[Proposition 7.4.6]{MR2598115} for $n\ge 3$.

The case $n\le 2$, which is left to prove, follows from~\cite{MR1841744}.
We recall the terminology from that article.
Given the operators $H_0=-\Delta$ and $H=H_0+V$ in $L^2(\R^n)$,
we denote
\[
U=\begin{cases}1,&V\ge 0;\\-1,&V<0\end{cases};
\quad v=\abs{V}^{1/2},\quad
w=U v,
\]
\[
M(\varsigma)=U+v(H_0+\varsigma^2)^{-1}v
:\;L^2(\R^n)\to L^2(\R^n),
\qquad
\Re \varsigma>0.
\]
There is the identity
$
(1-w(H+\varsigma^2)^{-1}v)(1+w(H_0+\varsigma^2)^{-1}v)=1;
$
hence, if $M(\varsigma)$ is invertible,
\[
1-w(H+\varsigma^2)^{-1}v=(1+w(H_0+\varsigma^2)^{-1}v)^{-1}
=(U+v(H_0+\varsigma^2)^{-1}v)^{-1}U=M(\varsigma)^{-1}U,
\]
\[
U-w(H+\varsigma^2)^{-1}w=M(\varsigma)^{-1},
\qquad
w(H+\varsigma^2)^{-1}w=U-M(\varsigma)^{-1}
\]
(cf. \cite[Equation (4.8)]{MR1841744}).
In the case at hand, $V=-u_\kappa^{2\kappa}$, $U=-1$, $v=-w=u_\kappa^\kappa$
(understood as operators of multiplication);
thus,
\begin{equation}\label{linear-b-def-m-kappa}
M(\varsigma)=
-1+u_\kappa^\kappa\circ
\Big(-\frac{\Delta}{2m}+\varsigma^2\Big)^{-1}\circ u_\kappa^\kappa,
\end{equation}
and, when $M(\varsigma)$ is invertible,
\begin{equation}
\label{linear-b-u-l-u-m}
u_\kappa^\kappa\circ (\eurl\sb{-}-z)^{-1}\circ u_\kappa^\kappa
=
u_\kappa^\kappa\circ (-\Delta-u_\kappa^{2\kappa}+\varsigma^2)^{-1}\circ u_\kappa^\kappa
=-1-M(\varsigma)^{-1},
\end{equation}
with $z$ and $\varsigma$ related by
\[
\frac{1}{2m}-z=\varsigma^2.
\]
Using the expression for the integral kernels of
$(H_0+\varsigma^2)^{-1}$
in dimensions $n\le 2$
(see \cite[Section 3, (3.13) and (3.14)]{MR1841744}), the operator
$M(\varsigma)$ extends by continuity to the region
\[
\{\varsigma\in \C\setminus \{0\}
\sothat \Re \varsigma \geq 0\},
\]
with a singularity $-\frac{1}{2\pi}\log(\varsigma)$
at $\varsigma=0$.
If $-\Delta-u_\kappa^{2\kappa}$ has no eigenvalue or virtual level at $\lambda=0$,
then $M(\varsigma)$ is invertible in the orthogonal complement
of $v$ (with the inverse bounded uniformly as $\varsigma\to 0$),
while $M(\varsigma)/\log(\varsigma)$ is always invertible
(with the inverse bounded uniformly as $\varsigma\to 0$)
in the span of $v$.
We deduce that, as long as $\abs{\varsigma}$ is small enough and $\varsigma\ne 0$,
$M(\varsigma)$
is invertible,
with a uniform bound
on $\norm{M(\varsigma)^{-1}}$
in an open neighborhood
of $0$ in the half-plane $\Re \varsigma\ge 0$.
Now the conclusion of the lemma
for the case $n=2$
follows from \eqref{linear-b-u-l-u-m}.
The one-dimensional case is dealt with similarly.
\end{proof}

\section{The Schur complement}

\label{linear-b-sect-Schur}

In the proofs, we often use the concept of the Schur complement,
which is sometimes called the Feshbach map.
Let us give the formulation in the operator form that we will need
(see also 
\cite[Theorem IV.1]{bach1998quantum} and \cite[Lemma 2.3]{MR1841744}).

\begin{lemma}
\label{lemma-schur-complement}
Let
$\bfX_1$ and $\bfX_2$
be Banach spaces
and let $A$ be a closed linear operator on $\bfX_1\oplus\bfX_2$
with dense domain $\dom(A)\subset\bfX_1\oplus\bfX_2$, such that
\[
A=\begin{bmatrix}A_{11}&A_{12}\\A_{21}&A_{22}\end{bmatrix},
\]
where $A_{j i}:\,\bfX_i\to\bfX_j$, $1\le i,\,j\le 2$ are closed linear operators
with domains
\[
\dom(A_{j1})=\{x_1\in\bfX_1\sothat (x_1,0)\in\dom(A)\},
\quad
\dom(A_{j2})=\{x_2\in\bfX_2\sothat (0,x_2)\in\dom(A)\}
\]
which are dense in $\bfX_1$ and $\bfX_2$, respectively.
Assume that there are Banach spaces
\[
\bfW_i \subset \bfX_i \subset \bfY_i,
\qquad
1\le i\le 2,
\]
with continuous dense embeddings,
such that $A_{j i}$ extends to a bounded operator from $\bfW_i$ to $\bfY_j$,
with $1\le i,\,j\le 2$.
Assume further that $A_{11}:\,\bfW_1\to\bfY_1$
has a bounded inverse.
Let $S$ be the Schur complement of $A_{11}$ defined by
\begin{eqnarray}\label{def-s-c-11}
S=A_{22}-A_{21}A_{11}^{-1}A_{12}:\;\bfW_2\to\bfY_2.
\end{eqnarray}
Then $A$
(considered as an operator from $\dom(A)$ to $\bfX_1\oplus \bfX_2$)
has a bounded inverse 
if and only if
$S:\bfW_2 \to \bfY_2$
is invertible
(not necessarily with the bounded inverse)
and 
the restrictions
$S^{-1}\at{\bfX_2}$, $A_{11}^{-1}A_{12}S^{-1}\at{\bfX_2}$,
$S^{-1}A_{21}A_{11}^{-1}\at{\bfX_1}$,
and $A_{11}^{-1}A_{12}S^{-1}A_{21}A_{11}^{-1}\at{\bfX_1}$
define the bounded operators
\begin{eqnarray*}
S^{-1}:\,\bfX_2&\to&\bfX_2,
\\
A_{11}^{-1}A_{12}S^{-1}:\,\bfX_2&\to&\bfX_1,
\\
S^{-1}A_{21}A_{11}^{-1}:\,\bfX_1&\to&\bfX_2,
\\
A_{11}^{-1}A_{12}S^{-1}A_{21}A_{11}^{-1}:\;\bfX_1&\to&\bfX_1.
\end{eqnarray*}
\end{lemma}

\begin{proof}
We consider $A$ as a bounded operator from $\bfW_1 \oplus \bfW_2$ to $\bfY_1 \oplus \bfY_2$. Since $A_{11}:\,\bfW_1\to\bfY_1$
has a bounded inverse, the operators
\[ 
\begin{bmatrix}I&A_{11}^{-1}A_{12}\\0&I\end{bmatrix} \qquad \mbox{and} \qquad
\begin{bmatrix}I&0\\A_{21}A_{11}^{-1}&I\end{bmatrix}
\]
are bounded as linear operators in $\bfW_1 \oplus \bfW_2$ and in $\bfY_1 \oplus \bfY_2$,
respectively,
with 
bounded inverses given by
\[ 
\begin{bmatrix}I&-A_{11}^{-1}A_{12}\\0&I\end{bmatrix} \qquad \mbox{and} \qquad
\begin{bmatrix}I&0\\-A_{21}A_{11}^{-1}&I\end{bmatrix}.
\]
The following identities show that
$S$
is invertible
as a map from $\bfW_2$ to $\bfY_2$
if and only if $A$ is invertible as a map from 
$\bfW_1 \oplus \bfW_2$ to $\bfY_1 \oplus \bfY_2$:
\begin{eqnarray}
\label{linear-b-t-3}
A
=
\begin{bmatrix}A_{11}&A_{12}\\A_{21}&A_{22}\end{bmatrix}
=
\begin{bmatrix}I&0\\A_{21}A_{11}^{-1}&I\end{bmatrix}
\begin{bmatrix}A_{11}&0\\0&S\end{bmatrix}
\begin{bmatrix}I&A_{11}^{-1}A_{12}\\0&I\end{bmatrix},
\qquad
\\[1ex]
\label{linear-b-t-3-inverse}
\begin{bmatrix}I&0\\-A_{21}A_{11}^{-1}&I\end{bmatrix}
\begin{bmatrix}A_{11}&A_{12}\\A_{21}&A_{22}\end{bmatrix}
\begin{bmatrix}I&-A_{11}^{-1}A_{12}\\0&I\end{bmatrix}
=
\begin{bmatrix}A_{11}&0\\0&S\end{bmatrix}.
\qquad
\end{eqnarray}
Explicitly, the inverse of $A$ in terms of the inverse of $S$ is given by
\begin{eqnarray}\label{linear-b-Schur}
A^{-1}
=
\begin{bmatrix}
A_{11}^{-1}+A_{11}^{-1}A_{12}S^{-1}A_{21}A_{11}^{-1}
&
-A_{11}^{-1}A_{12}S^{-1}
\\[1ex]
-S^{-1}A_{21}A_{11}^{-1}
&
S^{-1}
\end{bmatrix};
\end{eqnarray}
this expression is known as the Banachiewicz inversion formula.
One can see that $A$
(considered as a map from $\dom(A)$ to $\bfX_1\oplus\bfX_2$)
has a bounded inverse
if and only if the mappings
$S^{-1}$,
$A_{11}^{-1}A_{12}S^{-1}$,
$S^{-1}A_{21}A_{11}^{-1}$,
and
$A_{11}^{-1}A_{12}S^{-1}A_{21}A_{11}^{-1}$
are bounded in appropriate spaces.
\end{proof}


\begin{remark}
Swapping the indices in Lemma~\ref{lemma-schur-complement},
we have the following result.
Assume that $A_{22}:\,\bfW_2\to\bfY_2$ has a bounded inverse.
Let $T$ be the Schur complement of $A_{22}$ defined by
\begin{eqnarray}\label{def-s-c-22}
T=A_{11}-A_{12}A_{22}^{-1}A_{21}:\;\bfW_1\to\bfY_1.
\end{eqnarray}
Then $A$
(considered as an operator from $\dom(A)$ to $\bfX_1\oplus \bfX_2$)
has a bounded inverse
if and only if
$T:\bfW_1 \to \bfY_1$
is invertible and 
the restrictions
$T^{-1}\at{\bfX_1}$, $A_{22}^{-1}A_{21}T^{-1}\at{\bfX_1}$,
$T^{-1}A_{12}A_{22}^{-1}\at{\bfX_2}$, and
$A_{22}^{-1}A_{21}T^{-1}A_{12}A_{22}^{-1}\at{\bfX_2}$
define the bounded operators
\begin{eqnarray*}
T^{-1}:\,\bfX_1&\to&\bfX_1,
\\
A_{22}^{-1}A_{21}T^{-1}:\,\bfX_1&\to&\bfX_2,
\\
T^{-1}A_{12}A_{22}^{-1}:\,\bfX_2&\to&\bfX_1,
\\
A_{22}^{-1}A_{21}T^{-1}A_{12}A_{22}^{-1}:\;\bfX_2&\to&\bfX_2.
\end{eqnarray*}
Similarly to
\eqref{linear-b-Schur},
 the inverse of $A$ in terms of the inverse of $T$
 is given in the explicit form by
 \begin{eqnarray}\label{linear-b-Schur-22}
 A^{-1}
 =
 \begin{bmatrix}
 T^{-1}
 &
 -T^{-1}A_{12}A_{22}^{-1}
 \\[1ex]
 -A_{22}^{-1}A_{21}T^{-1}
 &
 A_{22}^{-1}+A_{22}^{-1}A_{21}T^{-1}A_{12}A_{22}^{-1}
 \end{bmatrix}.
 \end{eqnarray}
\end{remark}

\bibliographystyle{sima-doi}
\bibliography{bibcomech}
\end{document}